\DeclareSymbolFont{AMSb}{U}{msb}{m}{n}
\documentclass[11pt,a4paper,twoside,leqno,noamsfonts]{amsart}
           \usepackage{setspace}
\usepackage[english]{babel}
\usepackage[dvipsnames]{xcolor}
\definecolor{britishracinggreen}{rgb}{0.0, 0.26, 0.15}
\definecolor{cobalt}{rgb}{0.0, 0.28, 0.67}
\usepackage[utopia]{mathdesign}
    \DeclareSymbolFont{usualmathcal}{OMS}{cmsy}{m}{n}
    \DeclareSymbolFontAlphabet{\mathcal}{usualmathcal}
\usepackage[a4paper,top=3.5cm,bottom=3cm,left=3.05cm,
           right=3.05cm,bindingoffset=5mm]{geometry}
\usepackage[utf8]{inputenc}
\usepackage{braket,caption,comment,mathtools,stmaryrd}
\usepackage{multirow,booktabs,microtype}
\usepackage{latexsym}
\usepackage{todonotes}
\usepackage{fancyhdr}
\usepackage{soul} 
\usepackage[colorlinks,bookmarks]{hyperref} %
\hypersetup{colorlinks,%
            citecolor=britishracinggreen,%
            filecolor=black,%
            linkcolor=cobalt,%
            urlcolor=black}
\numberwithin{equation}{section}



\def\into{\hookrightarrow}

\def\ra{\rightarrow}

\def\KS{\mathsf{KS}}
\def\ad{\mathsf{ad}}

\def\R{\mathbb R}
\def\C{\mathbb C}

\def\Z{\mathbb Z}

\def\FF{\mathscr{F}}

\def\O{\mathscr O}

\DeclareMathOperator{\Mor}{Mor}

\DeclareMathOperator{\Ob}{Ob}

\DeclareMathOperator{\Spec}{Spec\,}

\DeclareMathOperator{\GL}{GL}

\DeclareMathOperator{\Aut}{Aut}

\DeclareMathOperator{\Hom}{Hom}

\DeclareMathOperator{\End}{End}

\DeclareMathOperator{\Der}{Der}

\theoremstyle{definition}

\newtheorem*{lemma*}{Lemma}
\newtheorem*{theorem*}{Theorem}
\newtheorem*{example*}{Example}
\newtheorem*{fact*}{Fact}
\newtheorem*{notation*}{Notation}
\newtheorem*{definition*}{Definition}
\newtheorem*{prop*}{Proposition}
\newtheorem*{remark*}{Remark}
\newtheorem*{corollary*}{Corollary}
\newtheorem*{conventions*}{Conventions}
\newtheorem*{caution*}{Caution}

\newtheorem{definition}{Definition}[section]

\newtheorem{remark}[definition]{Remark}

\newtheoremstyle{thm} 
        {3mm}
        {3mm}
        {\slshape}
        {0mm}
        {\bfseries}
        {.}
        {1mm}
        {}
\theoremstyle{thm}
\newtheorem{theorem}[definition]{Theorem}
\newtheorem{corollary}[definition]{Corollary}
\newtheorem{lemma}[definition]{Lemma}
\newtheorem{prop}[definition]{Proposition}

\newtheoremstyle{sol} 
        {3mm}
        {3mm}
        {\normalfont}
        {0mm}
        {\scshape}
        {.}
        {1mm}
        {}
\theoremstyle{sol}

\newtheorem{notation}[definition]{Notation}
\newtheorem*{ssolution*}{Solution (sketch)}
\newtheorem*{solution*}{Solution}

\usepackage{tikz}
\usepackage{tikz-cd}
\usepackage{rotating}

\usetikzlibrary{matrix,shapes,arrows,decorations.pathmorphing}
\tikzset{commutative diagrams/arrow style=math font}
\tikzset{commutative diagrams/.cd,
mysymbol/.style={start anchor=center,end anchor=center,draw=none}}

\tikzset{
shift up/.style={
to path={([yshift=#1]\tikztostart.east) -- ([yshift=#1]\tikztotarget.west) \tikztonodes}
}
}

\DeclareMathAlphabet{\mathpzc}{OT1}{pzc}{m}{it}

\newcommand*{\defeq}{\mathrel{\vcenter{\baselineskip0.5ex \lineskiplimit0pt
                     \hbox{\scriptsize.}\hbox{\scriptsize.}}}%
                     =}

\title[Deformations of holomorphic pairs and $2d$-$4d$ wall-crossing]{Deformations of holomorphic pairs and $2d$-$4d$ wall-crossing\\ [1ex]
  }

\author{
Veronica Fantini 
}
\address{SISSA Trieste, Via Bonomea 265, 34136 Trieste, Italy}
\email{ vfantini@sissa.it}

\begin{document}
\hbadness=150
\vbadness=150

\begin{abstract}
We show how wall-crossing formulas in coupled $2d$-$4d$ systems, introduced by Gaiotto, Moore and Neitzke, can be interpreted geometrically in terms of the deformation theory of holomorphic pairs, given by a complex manifold together with a holomorphic vector bundle. The main part of the paper studies the relation between scattering diagrams and deformations of holomorphic pairs, building on recent work by Chan, Conan Leung and Ma.      
\end{abstract}

\maketitle

{
\hypersetup{linkcolor=black}
\tableofcontents}

\section{Introduction}


In the mathematical physics literature, {wall-crossing formulas} (WCFs for short) express the dependence of physically admissible ground states (``BPS states") in a class of theories on certain crucial parameters (``central charge"). For instance the WCF in coupled $2d$-$4d$ systems introduced by Gaiotto, Moore and Neitzke in \cite{WCF2d-4d} governs wall-crossing of BPS states in $\mathcal{N}=2$ supersymmetric $4d$ gauge theory coupled with a surface defect. This generalizes both the formulas of Cecotti--Vafa \cite{CV} in the pure $2d$ case and those of Kontsevich--Soibelman in the pure $4d$ case \cite{WCF4d, WCFKS}. 

From a mathematical viewpoint, in the pure $2d$ case the WCFs look like braiding identities for matrices representing certain monodromy data (Stokes matrices, see e.g. \cite{dubrovin}). On the other hand Kontsevich--Soibelman \cite{WCFKS} start from the datum of the (``charge") lattice $\Gamma$, endowed with an antisymmetric (``Dirac") pairing $\langle\cdot,\cdot\rangle_D\colon\Gamma\times\Gamma\to\Z$,  and define a Lie algebra closely related to the Poisson algebra of functions on the algebraic torus $(\C^*)^{\operatorname{rk} \Gamma}$. 
Then their WCFs are expressed in terms of formal Poisson automorphisms of this algebraic torus.

In the coupled $2d$-$4d$ case studied in \cite{WCF2d-4d} the setting becomes rather more complicated. In particular the lattice $\Gamma$ is upgraded to a pointed groupoid $\mathbb{G}$, whose objects are indices $\lbrace i,j,k\cdots\rbrace$ and whose morphisms include the charge lattice $\Gamma$ as well as arrows parametrised by $\amalg_{ij}\Gamma_{ij}$, where $\Gamma_{ij}$ is a $\Gamma$-torsor. Then the relevant wall-crossing formulas involve two types of formal automorphisms of the groupoid algebra $\C[\mathbb{G}]$: type $S$, corresponding to Cecotti--Vafa monodromy matrices, and type $K$, which generalize the formal torus automorphisms of Kontsevich--Soibelman. The main new feature is the nontrivial interaction of automorphisms of type $S$ and $K$.

These $2d$-$4d$ formulas have been first studied with a categorical approach by Kerr and Soibelman in \cite{2d-4dSK}. In this paper we take a very different point of view. We show how a large class of $2d$-$4d$ formulas can be constructed geometrically by using the deformation theory of holomorphic pairs, given by a complex manifold $\check{X}$ together with a holomorphic bundle $E$.

Our construction is a variant of the remarkable recent results of Chan, Conan Leung and Ma \cite{MCscattering}. That work shows how consistent scattering diagrams, in the sense of Kontsevich--Soibelman and Gross--Siebert (see e.g. \cite{GPS}) can be constructed via the asymptotic analysis of deformations of a complex manifold $\check{X}\defeq TM/\Lambda$, given by a torus fibration over a smooth tropical affine manifold $M$. The complex structure depends on a parameter $\hbar$, and the asymptotic analysis is performed in the semiclassical limit $\hbar \to 0$. The gauge group acting on the set of solutions of the Maurer-Cartan equation (which governs deformations of $\check{X}$) contains the \textit{tropical vertex group} $\mathbb{V}$ of Kontsevich--Soibelman and Gross--Siebert. Elements of the tropical vertex group are formal automorphisms of an algebraic torus and are analogous to the type $K$ automorphisms described above, and \textit{consistent} scattering diagrams with values the tropical vertex group reproduce wall-crossing formulas in the pure $4d$ case. 

In the present paper we introduce an extension $\tilde{\mathbb{V}}$ of the tropical vertex group whose definition is modelled on the deformation theory of holomorphic pairs $(\check{X}, E)$. In our applications, $\check{X}$ is defined as above and $E$ is a holomorphically trivial vector bundle on $\check{X}$. In order to simplify the exposition we always assume $\check{X}$ has complex dimension $2$, but we believe that this restriction can be removed along the lines of \cite{MCscattering}. Our first two main results give the required generalization of the construction of Chan, Conan Leung and Ma.
\begin{theorem}[Theorem \ref{thm:asymptotic_gauge}]
Let $\mathfrak{D}$ be an initial scattering diagram, with values in the extended tropical vertex group $\tilde{\mathbb{V}}$, consisting of two non-parallel walls. Then there exists an associated solution $\Phi$ of the Maurer-Cartan equation, which governs deformations of the holomorphic pair $(\check{X}, E)$, such that the asymptotic behaviour of  $\Phi$ as $\hbar \to 0$ defines uniquely a scattering diagram $\mathfrak{D}_{\infty}$ (see Definition \ref{def:D_infty}), with values in $\tilde{\mathbb{V}}$. 
\end{theorem}

\begin{theorem}[Theorem \ref{thm:consistentD}]
The scattering diagram $\mathfrak{D}_{\infty}$ is consistent. 
\end{theorem}

We briefly highlight the main steps of the construction, which follows closely that of \cite{MCscattering}, adapting it to pairs $(\check{X}, E)$.  

\textbf{Step 1} We first introduce a \textit{symplectic dgLa} as the Fourier-type transform of the Kodaira-Spencer dgLa $\KS(\check{X},E)$ which governs deformation of the pair $(\check{X}, E)$. Although the two dgLas are isomorphic, we find that working on the symplectic side makes the results more transparent. In particular we define the Lie algebra $\tilde{\mathfrak{h}}$ as a subalgebra, modulo terms which vanish as $\hbar \to 0$, of the Lie algebra of infinitesimal gauge transformations on the symplectic side. The extension of the tropical vertex group is then defined as $\tilde{\mathbb{V}}\defeq\exp(\tilde{\mathfrak{h}})$.  

\textbf{Step 2.a} Starting from the data of a wall in a scattering diagram, namely from the automorphism $\theta$ attached to a line $P$, we construct a solution $\Pi$ supported along the wall, i.e. such that there exists a unique normalised infinitesimal gauge transformation $\varphi$ which takes the trivial solution to $\Pi$ and has asymptotic behaviour with leading order term given by $\log(\theta)$ (see Proposition \ref{prop:asymp1}). The gauge-fixing condition $\varphi$ is given by choosing a suitable homotopy operator $H$.    

\textbf{Step 2.b} Let $\mathfrak{D}=\lbrace\mathsf{w}_1, \mathsf{w}_2\rbrace$ be an initial scattering diagram with two non-parallel walls. By \textbf{Step 2.a.}, there are Maurer-Cartan solutions $\Pi_1$, $\Pi_2$, which respectively supported along the walls $\mathsf{w}_1$, $\mathsf{w}_2$. Using Kuranishi's method we construct a solution $\Phi$ taking as input $\Pi_1 +\Pi_2$, of the form $\Phi=\Pi_1+\Pi_2+\Xi$, where $\Xi$ is a correction term. In particular $\Xi$ is computed using a different homotopy operator $\mathbf{H}$. 

\textbf{Step 2.c } By using \textit{labeled ribbon trees} we write $\Phi$ as a sum of contributions $\Phi_a$ over ${a\in\left(\Z^2_{\geq 0}\right)_{\text{prim}}}$, each of which turns out to be independently a Maurer-Cartan equation (Lemma \ref{lem:Phi_aMC}). Moreover we show that each $\Phi_a$ is supported on a ray of rational slope, meaning that for every $a$, there is a unique normalised infinitesimal gauge transformation $\varphi_a$ whose asymptotic behaviour is an element of our Lie algebra $\tilde{\mathfrak{h}}$ (Theorem \ref{thm:asymptotic_gauge}). The transformations $\varphi_a$ allow us to define the saturated scattering diagram $\mathfrak{D}_\infty$ (Definition \ref{def:D_infty}) from the solution $\Phi$.

Note that in fact the results of \cite{MCscattering} have already been extended to a large class of dgLas (see \cite{CLMaY19}). For our purposes however we need a more ad hoc study of a specific differential-geometric realization of $\KS(\check{X},E)$: for example, there is a background Hermitian metric on $E$ which needs to be chosen carefully.\\ 

We can now discuss the application to $2d$-$4d$ wall-crossing. As we mentioned WCFs for coupled $2d$-$4d$ systems involve automorphisms of type $S$ and $K$. In Section \ref{sec:WCF} we consider their infinitesimal generators (i.e. elements of the Lie algebra of derivations of $\Aut(\C[\mathbb{G}][\![ t ]\!])$), and we introduce the Lie ring $\mathbf{L}_\Gamma$ which they generate as a $\C[\Gamma]$-module. On the other hand we construct a Lie ring $\tilde{\mathbf{L}}$ generated as $\C[\Gamma]$-module by certain special elements of the extended tropical vertex Lie algebra for holomorphic pairs, $\tilde{\mathfrak{h}}$. Our main result compares these two Lie rings.
\begin{theorem}[Theorem \ref{thm:homomLiering}]
Let $\left(\mathbf{L}_{\Gamma},[\cdot,\cdot]_{\Der(\C[\mathbb{G}])}\right)$ and $\left(\tilde{\mathbf{L}},[\cdot,\cdot]_{\tilde{\mathfrak{h}}}\right)$ be the $\C[\Gamma]$-modules discussed above (see Section \ref{sec:WCF}). Under an assumption on the BPS spectrum, there exists a homomorphism of $\C[\Gamma]$-modules and of Lie rings $\Upsilon\colon \mathbf{L}_{\Gamma}\to\tilde{\mathbf{L}}$.
\end{theorem}
This result shows that a saturated scattering diagram with values in (the formal group of) $\tilde{\mathbf{L}}$ is the same as a $2d$-$4d$ wall-crossing formula. Thus, applying our main construction with suitable input data, we can recover a large class of WCFs for coupled $2d$-$4d$ systems from the deformation theory of holomorphic pairs $(\check{X}, E)$. At the end of Section \ref{sec:WCF} we show how this works explicitly in two examples.  

\subsection{Plan of the paper}
The paper is organized as follows: in Section \ref{sec:background} we provide some background on deformations of holomorphic pairs in terms of differential graded Lie algebras, and on scattering diagrams, according to the definition given by Gross--Pandharipande--Siebert \cite{GPS}. In Section \ref{sec:single wall}, we construct a deformation $\Pi$, supported on a wall $\mathsf{w}$ of a scattering diagram. In Section \ref{sec:two_walls}, given an initial scattering diagram $\mathfrak{D}=\lbrace\mathsf{w}_1, \mathsf{w}_2\rbrace$ with two non-parallel walls, we first construct a solution $\Phi$ from the input of the solutions $\Pi_1$ and $\Pi_2$ supported respectively on $\mathsf{w}_1$ and $\mathsf{w}_2$. Then from the asymptotic analysis of $\Phi$ we compute the consistent scattering diagram $\mathfrak{D}_\infty $. Finally in Section \ref{sec:WCF} we recall the setting of wall-crossing formulas in coupled $2d$-$4d$ systems and prove our correspondence result, Theorem \ref{thm:homomLiering}, between $2d$-$4d$ wall-crossing and deformations of holomorphic pairs. 

\subsection{Acknowledgements}
I would like to thank my advisor Jacopo Stoppa for proposing this problem and for continuous support, helpful discussions, suggestions and corrections.

\section{Background}\label{sec:background}
\subsection{Setting}
Let $M$ be an affine tropical two dimensional manifold. Let $\Lambda$ be a lattice subbundle of $TM$ locally generated by $\frac{\partial}{\partial x^1},\frac{\partial}{\partial x^2} $, for a choice of affine coordinates $x=(x_1,x_2)$ on a contractible open subset $U\subset M$. We denote by $\Lambda^*=\Hom_\Z(\Lambda,\Z)$ the dual lattice and by $\langle\cdot,\cdot\rangle\colon\Lambda\times\Lambda^*\to\C$ the natural pairing. 

Define $\check{X}\defeq TM/\Lambda$ to be the total space of the torus fibration $\check{p}\colon\check{X}\to M$ and similarly define $X\defeq T^*M/\Lambda^* $. Then, let $\lbrace \check{y}_1, \check{y}_2\rbrace$ be the coordinates on the fibres of  $\check{X}(U)$ with respect to the basis $\frac{\partial}{\partial x^1},\frac{\partial}{\partial x^2}$, and define a one-parameter family of complex structures on $\check{X}$: $J_{\hbar}=\begin{pmatrix}
0 & \hbar I\\
-\hbar^{-1}I & 0
\end{pmatrix}$, with respect to the basis $\lbrace \frac{\partial}{\partial{x_1}},\frac{\partial}{\partial{x_2}},\frac{\partial}{\partial{\check{y}_1}},\frac{\partial}{\partial{\check{y}_2}}\rbrace$, parametrized by $\hbar\in\R_{>0}$. Notice that a set of holomorphic coordinates with respect to $J_{\hbar}$ is defined by $z_j\defeq\check{y}_j+i\hbar x_j$, $j=1,2$; in particular we will denote by $\textbf{w}_j\defeq e^{2\pi iz_j}$. On the other hand $X$ is endowed with a natural symplectic structure $\omega_{\hbar}\defeq\hbar^{-1}dy_j\wedge dx_j$, where $\lbrace y_j\rbrace$ are coordinates on the fibres of $X(U)$.

In the next section we are going to consider the limit as $\hbar\to 0$, which corresponds to the large volume limit of the dual torus fibration $(X,\omega_{\hbar})$. This is motived by mirror symmetry which predicts that quantum corrections to the complex side should arise in the large volume limit.

\subsection{Deformations of holomorphic pairs}
Let $E$ be a rank $r$ holomorphic vector bundle on $\check{X}$ with fixed hermitian metric $h_E$. We are going to set up the deformation problem of the pair $(\check{X}, E)$, with the approach of differential graded Lie algebras (dgLa) following \cite{defpair} and \cite{Manetti}. 
\begin{definition}
The Kodaira-Spencer dgLa which governs deformations of a holomorphic pair $(\check{X},E)$ is defined as follows
\begin{equation}
\KS(\check{X},E)\defeq (\Omega^{0,\bullet}(\check{X},\End E\oplus T^{1,0}X),\bar{\partial}=\begin{pmatrix}
\bar{\partial}_E & B\\
0 & \bar{\partial}_{\check{X}}
\end{pmatrix},[\cdot,\cdot]),
\end{equation}
where $B\colon\Omega^{0,q}(\check{X},T^{1,0}\check{X})\ra\Omega^{0,q}(\check{X},\End E)$ acts on $\varphi\in\Omega^{0,q}(\check{X}, T^{1,0}\check{X})$ as \[B\varphi\defeq\varphi\lrcorner F_E\] 
where $F_E$ is the Chern curvature of $E$ (defined with respect to the hermitan metric $h_E$ and the complex structure $\bar{\partial}_E$). The symbol $\lrcorner$ is the contraction of forms with vector fields; in particular since $F_E$ is of type $(1,1)$ and $\varphi$ is valued in $T^{1,0}\check{X}$, $\varphi\lrcorner F_E$ is the contraction of the type $(1,0)$ of $F_E$ with respect to the $T^{1,0}$ part of $\varphi$ and the wedge product of the $(0,q)$ form of $\varphi$ with the type $(0,1)$ of $F_E$. 
The Lie bracket is 
\begin{equation}
[(A,\varphi),(N,\psi)]\defeq(\varphi\lrcorner\nabla^EN-(-1)^{pq}\psi\lrcorner\nabla^EA+[A,N]_{\End E}, [\varphi,\psi]) 
\end{equation}
where $\nabla^E$ is the Chern connection of $(E,\bar{\partial}_E,h_E)$, $(A,\varphi)\in\Omega^{0,p}(\check{X},\End E\oplus T^{1,0}\check{X})$, i.e. $(A,\varphi)=(\tilde{A}^K,\tilde{\varphi}^K)d\bar{z}_K$ with $\tilde{A}^K\in\End E $, $\tilde{\varphi}^K\in T^{1,0}\check{X}$ and the multi-index $|K|=p$, and $(N,\psi)\in\Omega^{0,q}(\check{X},\End E\oplus T^{1,0}\check{X})$.
\end{definition}
The definition does not depend on the choice of a hermitian metric on $E$ but only on the cohomology class of $F_E$: if $h_E'$ is another metric such that its Chern curvature $F_E'\in [F_E]$, then the corresponding dgLas are quasi-isomorphic (see appendix of \cite{defpair}). 

An infinitesimal deformation of $(\check{X},E)$ is a pair  $(A,\varphi)\in\Omega^{0,1}(\check{X},\End E\oplus T^{1,0}\check{X}){\otimes}_{\C}\C[\![ t ]\!]$ 
 such that it is a solution of the so called Maurer-Cartan equation:
\begin{equation}\label{eq:MC}
    \bar{\partial}(A,\varphi) +\frac{1}{2}\big[(A,\varphi),(A,\varphi)\big]=0.
\end{equation}
In addition, there exists an infinitesimal gauge group action on the set of solutions of \eqref{eq:MC}, defined by $h\in\Omega^{0}(\check{X},\End E\oplus T^{1,0}\check{X})[\![ t ]\!]$ which acts on $f\in\Omega^{0,1}(\check{X},\End E\oplus T^{1,0}\check{X})[\![ t ]\!]$ as
\begin{equation}
e^h\ast f\defeq f-\displaystyle\sum_{k=0}^\infty \frac{\ad_h^k}{(k+1)!}(\bar{\partial} h-[f,h]),
\end{equation}
where $\ad_h(\cdot)=[h,\cdot]$. 
In order to fix the notation, let us recall the definition of the Kodaira-Spencer dgLa $\KS(\check{X})$ which governs formal deformations of the complex manifold $\check{X}$:
\begin{equation}
    \KS(\check{X})\defeq\left(\Omega^{0,\bullet}(\check{X},T^{1,0}\check{X}), \bar{\partial}_{\check{X}}, [\cdot,\cdot]\right)
\end{equation}
where $\bar{\partial}_{\check{X}}$ is the Dolbeault operator of the complex manifold $\check{X}$ associated to $J_{\hbar}$ and $[\cdot,\cdot]$ is the standard Lie bracket on vector fields and the wedge on the form part. 

\subsection{Scattering diagrams}
Let $e_1$ and $e_2$ be a basis for $\Lambda$, then the group ring $\mathbb{C}[\Lambda]$ is the ring of Laurent polynomial in the variable $z^m$, where $z^{e_1}=x$ and $z^{e_2}=y$. Let $\mathsf{m}_t$ be the maximal ideal of ring of formal power series $\mathbb{C}[\![ t ]\!]$ and define the Lie algebra $\mathfrak{g}$
\begin{equation}
    \mathfrak{g}\defeq\mathsf{m}_t\big(\mathbb{C}[\Lambda]{\otimes}_{\mathbb{C}}\C[\![ t ]\!]\big)\otimes_{\mathbb{Z}} \Lambda^*
\end{equation}
where every $n\in\Lambda^*$ is associated to a derivation $\partial_n$ such that $\partial_n(z^m)=\langle m,n\rangle z^m$ and the natural Lie bracket on $\mathfrak{g}$ is 
\begin{equation}
[z^m\partial_n,z^{m'}\partial_{n'}]\defeq z^{m+m'}\partial_{\langle m',n\rangle n'-\langle m, n'\rangle n}.
\end{equation}
In particular there is a Lie sub-algebra $\mathfrak{h}\subset\mathfrak{g}$:
\begin{equation}
\mathfrak{h}\defeq\bigoplus_{m\in \Lambda\smallsetminus\lbrace0\rbrace}z^m\cdot\big(\mathsf{m}_t\otimes m^{\perp}\big),
\end{equation}
where $m^{\perp}\in\Lambda^*$ is identified with the derivation $\partial_n$ and $n$ the unique primitive vector such that $\langle m,n\rangle=0$ and it is positive oriented according with the orientation induced by $\Lambda_{\mathbb{R}}\defeq\Lambda\otimes\mathbb{R}$. 
\begin{definition}[\cite{GPS}]
The tropical vertex group $\mathbb{V}$ is the sub-group of $\Aut_{\C[\![ t ]\!]}\big(\mathbb{C}[\Lambda]{\otimes}_{\mathbb{C}}\C[\![ t ]\!]\big)$, such that $\mathbb{V}\defeq\exp(\mathfrak{h})$. The product on $\mathbb{V}$ is defined by the Baker-Campbell-Hausdorff (BCH) formula, namely  
\begin{equation}
g\circ g'=\exp(h)\circ \exp(h')\defeq\exp(h\bullet h')=\exp(h+h'+\frac{1}{2}[h,h']+\cdots)
\end{equation}
where $g=\exp(h), g'=\exp(h')\in\mathbb{V}$.  
\end{definition}
The tropical vertex group was introduced by Gross, Pandharipande and Siebert in \cite{GPS} and in the simplest case its elements are formal one parameter families of symplectomorphisms of the algebraic torus $\C^*\times\C^*=\Spec\C[x,x^{-1},y,y^{-1}]$ with respect to the holomorphic symplectic form $\frac{dx}{x}\wedge \frac{dy}{y}$. 
\begin{definition}[Scattering diagram]
A scattering diagram $\mathfrak{D}$ is a collection of \textit{walls} $\textsf{w}_i=(m_i,P_i,\theta_i)$, where 
\begin{itemize}
    \item $m_i\in \Lambda$,
    \item $P_i$ can be either a \textit{line} through $m_0$, i.e. $P_i=m_0-m_i\mathbb{R}$ or a \textit{ray} (half line) $P_i=m_0-m_i\mathbb{R}_{\geq 0}$,
    \item $\theta_i\in\mathbb{V}$ is such that $\log(\theta_i)=\sum_{j,k}a_{jk}t^j z^{km_i}\partial_{n_i}$.
\end{itemize}
Moreover for any $k>0$ there are finitely many $\theta_i$ such that $\theta_i\not\equiv 1$ mod $t^k$.
\end{definition}

As an example, the scattering diagram \[\mathfrak{D}=\lbrace \mathsf{w}_1=\big(m_1=(1,0), P_1=m_1\R, \theta_1\big), \mathsf{w}_2=\big(m_2=(0,1), P_2=m_2\R, \theta_2\big)\rbrace\] can be represented as if figure \ref{fig:D1}.
\begin{figure}[h]
\center
\begin{tikzpicture}
\draw (2,0) -- (2,4);
\draw (0,2) -- (4,2);
\node [below right, font=\tiny] at (2,2) {0};
\node [below right, font=\tiny] at (2,2) {0};
\node [font=\tiny,above right] at (4,2) {$\theta_1$};
\node [font=\tiny,above left] at (2,4) {$\theta_2$};
\end{tikzpicture}
\caption{A scattering diagram with only two walls $\mathfrak{D}=\lbrace\mathsf{w}_1,\mathsf{w}_2\rbrace$}
\label{fig:D1}
\end{figure}
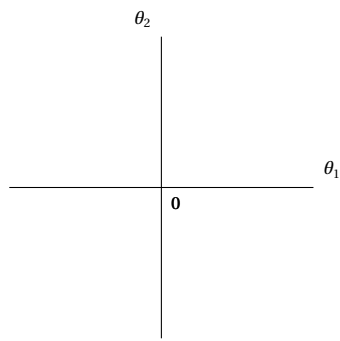

Denote by $\emph{Sing}(\mathfrak{D})$ the singular set of $\mathfrak{D}$: 
\[\emph{Sing}(\mathfrak{D})\defeq\bigcup_{\mathsf{w}\in\mathfrak{D}}\mathsf{\partial}P_\mathsf{w}\cup\bigcup_{\mathsf{w}_1,\mathsf{w}_2}P_{\mathsf{w}_1}\cap P_{\mathsf{w}_2}\] where $\partial P_\mathsf{w}=m_0$ if $P_{\mathsf{w}}$ is a ray and zero otherwise. There is a notion of order product for the automorphisms associated to each lines of a given scattering diagram, and it defined as follows:
\begin{definition}[Path order product]\label{def:pathorderedprod}
Let $\gamma:[0,1]\rightarrow \Lambda\otimes_{\mathbb{R}}\mathbb{R}\setminus\emph{Sing}(\mathfrak{D})$ be a smooth immersion with starting point that does not lie on a ray of the scattering diagram $\mathfrak{D}$ and such that it intersects transversally the rays of $\mathfrak{D}$ (as in figure \ref{fig:Dloop}). For each power $k>0$, there are times $0< \tau_{1}\leq\cdots\leq\tau_{s} <1$ and rays $P_i\in\mathfrak{D}$ such that $\gamma(\tau_j)\cap P_j\neq 0$. Then, define $\Theta_{\gamma,\mathfrak{D}}^k\defeq\prod_{j=1}^s\theta_{j}$. The path order product is given by: 
\begin{equation}
\Theta_{\gamma,\mathfrak{D}}\defeq\lim_{k\rightarrow\infty}\Theta_{\gamma,\mathfrak{D}}^k
\end{equation}
\end{definition}

\begin{figure}[h]
\begin{tikzpicture}
\draw[thick] (2,0) -- (2,4);
\draw[thick] (0,2) -- (4,2);
\node [below right, font=\tiny] at (2,2) {0};
\node [font=\small, right] at (4,2) {$\theta_1$};
\node [font=\small, right] at (2,4) { $\theta_2$};
\node [font=\small,left] at (0,2) {$\theta_1^{-1}$};
\node [font=\small, right] at (2,0) {$\theta_2^{-1}$};
\draw[red, thick] (2,2) -- (4,4);
\node [red, font=\small, right] at (4,4) {$\theta_m$ };
\draw [-latex, blue] (3,1) arc (-45:310:1.41);
\node [blue, below left, font=\tiny] at (1,1) {$\gamma$};
\end{tikzpicture}
\caption{$\Theta_{{\gamma}, \mathfrak{D}_{\infty}}=\theta_{1}\circ{\theta_{m}}\circ\theta_{2}\circ\theta_{1}^{-1}\circ\theta_{2}^{-1}$}
\label{fig:Dloop}
\end{figure}
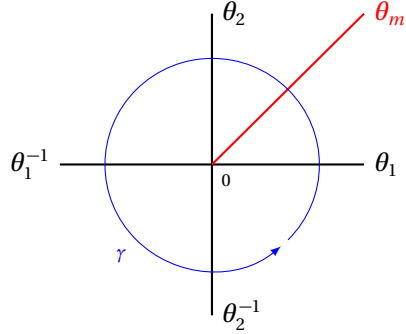

\begin{definition}[Consistent scattering diagram]A scattering diagram $\mathfrak{D}$ is consistent if for any closed path $\gamma$ intersecting $\mathfrak{D}$ generically, $\Theta_{\gamma,\mathfrak{D}}=\text{Id}_{\mathbb{V}}$. 
\end{definition}
The following theorem by Kontsevich and Soibelman is an existence (ad uniqueness) result of complete scattering diagram: 
\begin{theorem}[\cite{KS}]
{Let $\mathfrak{D}$ be a scattering diagram with two non parallel walls. There exists a unique minimal scattering diagram $\mathfrak{D}_\infty\supseteq\mathfrak{D}$ such that $\mathfrak{D}_{\infty}\setminus\mathfrak{D}$ consists only of rays, and it is {consistent}.}
\end{theorem}
The diagram is minimal meaning that we do not consider rays with trivial automorphisms. 

\subsubsection{Extension of the tropical vertex group}
In \cite{MCscattering} the authors prove that some elements of the gauge group acting on $\Omega^{0,1}(\check{X},T^{1,0}\check{X})$ can be represented as elements of the tropical vertex group $\mathbb{V}$. Here we are going to define an extension of the Lie algebra $\mathfrak{h}$, which will be related with the infinitesimal generators of the gauge group acting on $\Omega^{0,1}(\check{X},\End E\oplus T^{1,0}\check{X})$.
Let $\mathfrak{gl}(r,\mathbb{C})$ be the Lie algebra of the Lie group $\GL(r,\mathbb{C})$, then we define 
\begin{equation}
\tilde{\mathfrak{h}}\defeq\bigoplus_{m\in \Lambda\smallsetminus\lbrace 0\rbrace}z^{m}\cdot\left(\mathsf{m}_t\mathfrak{gl}(r,\mathbb{C})\oplus \left(\mathsf{m}_t\otimes m^{\perp}\right)\right). 
\end{equation}
\begin{lemma}\label{lem:dgLa}
$\Big(\tilde{\mathfrak{h}} , [\cdot ,\cdot]_{\sim}\Big)$ is a Lie algebra, where the bracket $[\cdot,\cdot]_{\sim}$ is defined by: 
\begin{equation}\label{eq:Liebracket}
[(A,\partial_n)z^m , (A',\partial_{n'})z^{m'} ]_{\sim}\defeq([A,A']_{\mathfrak{gl}}z^{m+m'}+A'\langle m',n\rangle z^{m+m'}- A\langle m,n'\rangle z^{m+m'}, [z^m\partial_n,z^{m'}\partial_{n'}]_{\mathfrak{h}} ).
\end{equation}
\end{lemma}
The definition of the Lie bracket $[\cdot, \cdot]_\sim$ is closely related with the Lie bracket of $\KS(\check{X}, E)$ and we will explain it below, in \eqref{subsec:relation}. A proof of this lemma  can be found in the appendix (see \ref{proof:dgLa}).


\subsection{Fourier transform}

In order to construct scattering diagrams from deformations of holomorphic pairs, it is more convenient to work with a suitable Fourier transform $\FF$ of the dgLa $\KS(\check{X},E)$. Following \cite{MCv1} we start with the definition of $\FF(\KS(\check{X}, E))$. Let $\mathcal{L}$ be the space of fibre-wise homotopy classes of loops with respect to the fibration $p\colon X\to M$ and the zero section $s\colon M\to X$, 
\[\mathcal{L}=\bigsqcup_{x\in M}\pi_1(p^{-1}(x), s(x)).\]
Define a map ${ev}\colon \mathcal{L}\to X$, which maps a homotopy class $[\gamma]\in\mathcal{L}$ to $\gamma(0)\in X$ and define ${pr}\colon \mathcal{L}\to M$ the projection, such that the following diagram commutes:
\[
\begin{tikzcd}
\mathcal{L}\arrow{rr}{{ev}}\arrow[swap]{rd}{{pr}}& & X\arrow{ld}{p}\\
& M&
\end{tikzcd}
\] 
In particular $pr$ is a local diffeomorphism and on a contractible open subset $U\subset M$ it induces an isomorphism $\Omega^\bullet(U,TM)\cong\Omega^\bullet(U_{\textbf{m}}, T\mathcal{L})$, where $U_{\textbf{m}}\defeq \lbrace\textbf{m}\rbrace\times U\in pr^{-1}(U)$, $\textbf{m}\in\Lambda$. In addition, there is a one-to-one correspondence between $\Omega^0(U,T^{1,0}\check{X})$ and $\Omega^0(U,TM)$, 
\[
\frac{\partial}{\partial z_j}\longleftrightarrow\frac{\hbar}{4\pi}\frac{\partial}{\partial x_j}
\] 
which leads us to the following definition:
\begin{definition} 
The Fourier transform is a map $\FF\colon\Omega^{0,k}(\check{X},T^{1,0}\check{X})\ra\Omega^k(\mathcal{L},T\mathcal{L})$, such that  
\begin{equation}\label{def:F}
\big(\FF(\varphi)\big)_{\textbf{m}}(x)\defeq\bigg(\frac{4\pi}{\hbar}\bigg)^{|I|-1}\int_{\check{p}^{-1}(x)}\varphi^{I}_j(x,\check{y})e^{-2\pi i(\textbf{m},\check{y})}d\check{y}\,\, dx_{I}\otimes \frac{\partial}{\partial x_j},
\end{equation}
where $\textbf{m}\in\Lambda$ represents an affine loop in the fibre $p^{-1} (x)$ with tangent vector $\sum_{j=1}^2m_j\frac{\partial}{\partial y_j}$ and $\varphi$ is locally given by $\varphi=\varphi^{I}_j(x,\check{y})d\bar{z}_I\otimes\frac{\partial}{\partial z_j}$, $|I|=k$.
\end{definition}
The inverse Fourier transform is then defined by the following formula, providing the coefficients have enough regularity: 
\begin{equation}
\FF^{-1}\big(\alpha\big)(x,\check{y})=\bigg(\frac{4\pi}{\hbar}\bigg)^{-|I|+1}\sum_{\textbf{m}\in\Lambda}\alpha_{j,\textbf{m}}^I e^{2\pi i(\textbf{m},\check{y})} d\bar{z}_I\otimes \frac{\partial}{\partial z_j}
\end{equation}
where $\alpha_{j,\textbf{m}}^I(x)dx_I\otimes\frac{\partial}{\partial x^j}\in\Omega^k(U_{\textbf{m}},T\mathcal{L})$ is the $\textbf{m}$-th Fourier coefficient of $\alpha\in\Omega^k(\mathcal{L},T\mathcal{L})$ and $|I|=k$.


The Fourier transform can be extended to $\KS(\check{X},E)$ as a map
\[\FF\colon\Omega^{0,k}(\check{X},\End E\oplus T^{1,0}\check{X})\ra\Omega^{k}(\mathcal{L}, \End E\oplus T\mathcal{L})\] 
\begin{equation}
\begin{split}
\FF\left(\left(A^{I}d\bar{z}_I,\varphi^{I}_jd\bar{z}_I\otimes\frac{\partial}{\partial z_j}\right)\right)_{\textbf{m}}\defeq&\left(\frac{4\pi}{\hbar}\right)^{|I|}\Big(\int_{\check{p}^{-1}(x)}A^{I}(x,\check{y})e^{-2\pi i (\textbf{m},\check{y})}d\check{y} dx_I,\\
&\left(\frac{4\pi}{\hbar}\right)^{-1}\int_{\check{p}^{-1}(x)}\varphi^{I}_j(x,\check{y})e^{-2\pi i(\textbf{m},\check{y})}d\check{y} dx_{I}\otimes\frac{\partial}{\partial x_j}\Big)
\end{split}
\end{equation} 
where the first integral is meant on each matrix element of $A^I$.

\subsection{Symplectic dgLa} In order to define a dgLa isomorphic to $\KS(\check{X},E)$, we introduce the so called Witten differential $d_W$ and the Lie bracket $\lbrace\cdot, \cdot\rbrace_\sim$, acting on $\Omega^\bullet(\mathcal{L},\End E\oplus T\mathcal{L})$. It is enough for us to consider the case in which $E$ is holomorphically trivial $E=\O_{\check{X}}\oplus\O_{\check{X}}\oplus\cdots\oplus\O_{\check{X}}$ and the hermitian metric $h_E$ is diagonal, $h_E=diag(e^{-\phi_1},\cdots, e^{-\phi_r)}$ and $\phi_j\in\Omega^0(\check{X})$, $j=1,\cdots,r$.  The differential $d_W$ is defined as follows:
\begin{equation*}
\begin{split}
d_W&\colon\Omega^k(\mathcal{L},\End E\oplus T\mathcal{L})\to\Omega^{k+1}(\mathcal{L},\End E\oplus T\mathcal{L})\\
d_W&\defeq\begin{pmatrix}
d_{W,E} & \hat{B}\\
0 & d_{W,\mathcal{L}}
\end{pmatrix}.
\end{split}
\end{equation*} 
In particular, $d_{W,E}$ is defined as:
\begin{equation}
\begin{split}
\left(d_{W,E}\left(A^Jdx_J\right)\right)_{\textbf{n}}&\defeq\FF(\bar{\partial}_E(\FF^{-1}(A^Jdx_J))_{\textbf{n}}\\
&=\FF\left(\bar{\partial}_E\left(\left(\frac{4\pi}{\hbar}\right)^{-|J|}\sum_{\textbf{m}\in\Lambda}e^{2\pi i(\textbf{m},\check{y})}A^J_{\textbf{m}}d\bar{z}_J\right)\right)_{\textbf{n}} \\
&=\left(\frac{4\pi}{\hbar}\right)^{-|J|}\FF\left(\sum_{\textbf{m}}e^{2\pi i(\textbf{m},\check{y})}\left(2\pi i\textbf{m}_kA^J_{\textbf{m}}+i\hbar\frac{\partial A^J_{\textbf{m}}}{\partial x_k}\right)d\bar{z}_k\wedge d\bar{z}_J\right)_{\textbf{n}} \\
&=\frac{4\pi}{\hbar}\int_{\check{p}^{-1}(x)}\left(\sum_{\textbf{m}}e^{2\pi i(\textbf{m},\check{y})}\left(2\pi i\textbf{m}_kA^J_{\textbf{m}}+i\hbar\frac{\partial A^J_{\textbf{m}}}{\partial x_k}\right)e^{-2\pi i(\textbf{n},\check{y})}\right)d\check{y} dx_k\wedge dx_J\\
&=\frac{4\pi}{\hbar}\left(2\pi i\textbf{n}_kA^J_{\textbf{n}}+i\hbar\frac{\partial A^J_{\textbf{n}}}{\partial x_k}\right)dx_k\wedge dx_J.
\end{split}
\end{equation} 
The operator $\hat{B}$ is then defined by
\begin{equation}
\begin{split}
\left(\hat{B}(\psi_{j}^I dx_I\otimes\frac{\partial}{\partial x_j})\right)_{\textbf{n}}&\defeq\FF(\FF^{-1}(\psi_{j}^I dx_I\otimes\frac{\partial}{\partial x_j})\lrcorner F_E)\\
&=\bigg(\frac{4\pi}{\hbar}\bigg)^{1-|I|}\FF\Big(\sum_\textbf{m}\psi_{\textbf{m},j}^Ie^{2\pi i(\textbf{m},\check{y})}d\bar{z}_I\otimes\frac{\partial}{\partial z_j}\lrcorner F_{pq}(\phi)dz_p\wedge d\bar{z}_q)\Big)\\
&=\bigg(\frac{4\pi}{\hbar}\bigg)^{1-|I|}\bigg(\frac{4\pi}{\hbar}\bigg)^{1+|I|}\int_{\check{p}^{-1}(x)}\sum_\textbf{m}\psi_{\textbf{m},j}^Ie^{2\pi i (\textbf{m}-\textbf{n},\check{y})}F_{jq}(\phi)d\check{y} dx_I\wedge dx_q 
\end{split}
\end{equation}
where $F_{pq}(\phi)$ is the curvature matrix. Then the $d_{W,\mathcal{L}}$ is defined by:
\begin{equation}
\left(d_{W,\mathcal{L}}(\psi_{j}^I dx_I\otimes\frac{\partial}{\partial x_j})\right)_{\textbf{n}}\defeq e^{-2\pi\hbar^{-1}({\textbf{n}},x)}d\left(\psi_{j}^I dx_I\otimes\frac{\partial}{\partial x_j}e^{2\pi\hbar^{-1}({\textbf{n}},x)}\right),
\end{equation}
where $d$ is the differential on the base $M$. Notice that by definition $d_W=\FF(\bar{\partial})\FF^{-1}$. Analogously we define the Lie bracket $\lbrace\cdot ,\cdot\rbrace_\sim\defeq\FF([\cdot,\cdot]_\sim)\FF^{-1}$. If we compute it explicitly in local coordinates we find:
\begin{align*}
\lbrace\cdot ,\cdot\rbrace_\sim&\colon\Omega^p(\mathcal{L},\End E\oplus T\mathcal{L})\times\Omega^q(\mathcal{L},\End E\oplus T\mathcal{L})\to\Omega^{p+q}(\mathcal{L},\End E\oplus T\mathcal{L})\\
\lbrace(A,\varphi),(N,\psi)\rbrace_{\sim }&= \big(\lbrace A,N\rbrace+\textbf{ad}(\varphi,N)-(-1)^{pq}\textbf{ad}(\psi,A), \lbrace\varphi, \psi\rbrace\big).
\end{align*}

In particular locally on $U_{\textbf{m}}$ we consider \[(A,\varphi)=\big(A_{\textbf{m}}^Idx_I,\varphi_{j,\textbf{m}}^{I}dx_{I}\otimes\frac{\partial}{\partial x_j}\big)\in\Omega^p(U_{\textbf{m}},\End E\oplus T\mathcal{L})\] and on $U_{\textbf{m}'}$ we consider \[(N,\psi)=(N_{\textbf{m}'}^Jdx_J,\psi_{J,\textbf{m}'}^{J}dx_{J}\otimes\frac{\partial}{\partial x\_{\textbf{m}}})\in\Omega^q(U_{\textbf{m}'},\End E\oplus T\mathcal{L})\] then   
\begin{equation}
\lbrace A,N\rbrace_{\textbf{n}}\defeq \sum_{\textbf{m}+\textbf{m'}=\textbf{n}}[A_{\textbf{m}}^I,N_{\textbf{m'}}^J]dx_I\wedge dx_J
\end{equation}
where the sum over $\textbf{m}+\textbf{m'}=\textbf{n}$ makes sense under the assumption of enough regularity of the coefficients. The operator $\textbf{ad}\colon\Omega^{p}(\mathcal{L},T\mathcal{L})\times\Omega^q(\mathcal{L},\End E)\to \Omega^{p+q}(\mathcal{L},\End E)$ is explicitly 
\begin{multline*}
\left(\textbf{ad}\left(\varphi,N\right)\right)_{\textbf{n}}\defeq\frac{4\pi}{\hbar}\Big(\int_{\check{p}^{-1}(x)}\Big(\sum_{\textbf{m}+\textbf{m}'=\textbf{n}}\varphi_{I,\textbf{m}}^j \Big(\frac{\partial N_{\textbf{m}'}^J}{\partial z_j}+2\pi i{\textbf{m}'}_j+A_j(\phi)N_{\textbf{m}'}^{J}\Big)\cdot\\
\cdot e^{2\pi i(\textbf{m}+\textbf{m}'-\textbf{n},\check{y})}\Big)d\check{y}\Big) dx_{I}\wedge dx_{J}
\end{multline*}
where $A_j(\phi)dz_j$ is the connection $\nabla^E$ one-form matrix. Finally the Lie bracket $\lbrace\psi,\phi\rbrace_\sim$ is 
\begin{equation}
\begin{split}
\lbrace\varphi, \psi\rbrace_{\textbf{n}}&\defeq\bigg(\sum_{\textbf{m}'+\textbf{m}=\textbf{n}}e^{-2\pi\hbar^{-1}(\textbf{n},x)}\Big(\varphi_{j,\textbf{m}}^I e^{2\pi (\textbf{m},x)}\nabla_{\frac{\partial}{\partial x_j}}\big(e^{2\pi (\textbf{m}',x)}\psi_{k,\textbf{m}'}^J\frac{\partial}{\partial x_k}\big)\\
&-(-1)^{pq}\psi_{k,\textbf{m}'}^Je^{2\pi\hbar^{-1} (\textbf{m}',x)}\nabla_{\frac{\partial}{\partial x_k}}\big(e^{2\pi\hbar^{-1}(\textbf{m},x)}\varphi_{j,\textbf{m}}^I\frac{\partial}{\partial x_j}\big)\Big)\bigg)dx_I\wedge dx_J
\end{split}
\end{equation}
where $\nabla$ is the flat connection on $M$.

\begin{definition}\label{def:mirror dgLa} 
The symplectic dgLa is defined as follows:
\begin{equation*}
G\defeq\left(\Omega^\bullet(\mathcal{L},\End E\oplus T\mathcal{L}), d_W, \lbrace\cdot, \cdot\rbrace_\sim\right)
\end{equation*}
and it is isomorphic to $\KS(\check{X},E)$ via $\FF$. 
\end{definition} 
As we mention above, the gauge group on the symplectic side $\Omega^0(\mathcal{L},\End E\oplus T\mathcal{L})$ is related with the extended Lie algebra $\tilde{\mathfrak{h}}$. However to figure it out, some more work has to be done, as we prove in the following subsection.  
\subsubsection{Relation with the Lie algebra $\tilde{\mathfrak{h}}$}\label{subsec:relation}

Let $Aff_M^{\Z}$ be the sheaf of affine linear transformations over $M$ defined for any open affine subset $U\subset M$ by $f_{{m}}(x)=({m},x)+b\in Aff_M^{\Z}(U)$ where $x\in U$, $m\in\Lambda$ and $b\in\R$. Since there is an embedding of $Aff_M^{\Z}(U)$ into $\O_{\check{X}}(\check{p}^{-1}(U))$ which maps $f_{{m}}(x)=({m},x)+b\in Aff_M^{\Z}(U)$ to $e^{2\pi i({m}, z)+2\pi i b}\in\O_{\check{X}}(\check{p}^{-1}(U))$, we define $\O_{aff}$ the image sub-sheaf of $Aff_M^{\Z}$ in $\O_{\check{X}}$. Then consider the embedding of the dual lattice $\Lambda^*\into T^{1,0}\check{X}$ which maps
\[
n\to n^j\frac{\partial}{\partial z_j}=:\check{\partial}_n.\] 
It follows that the Fourier transform $\FF$ maps \[\left(Ne^{2\pi i((m,z)+b)}, e^{2\pi i((m,z)+b)}\check{\partial}_n\right)\in\O_{aff}\left(\check{p}^{-1}(U), \mathfrak{gl}(r,\C)\oplus T^{1,0}\check{X} \right)\] to \[\left(Ne^{2\pi ib}\mathfrak{w}^m, \frac{\hbar}{4\pi}e^{2\pi i b}\mathfrak{w}^mn^j\frac{\partial}{\partial x_j}\right)\in\mathfrak{w}^m\cdot \underline{\C}\left(U_m,\mathfrak{gl}(r,\C)\oplus T\mathcal{L}\right)\] where $\mathfrak{w}^m\defeq\FF(e^{2\pi i({m}, z)})$, i.e. on $U_k$
\[
\mathfrak{w}^m=\begin{cases} 
e^{2\pi\hbar^{-1} (m,x)} & \text{if $k=m$}\\
0 & \text{if $k\neq m$}
\end{cases}
\] and we define \[\partial_n\defeq \frac{\hbar}{4\pi}n^j\frac{\partial}{\partial x_j}.\]
Let $\mathcal{G}$ be the sheaf over $M$ defined as follows: for any open subset $U\subset M$
\[
\mathcal{G}(U)\defeq\bigoplus_{m\in\Lambda\setminus 0}\mathfrak{w}^m\cdot\underline{\C}(U,\mathfrak{gl}(r,\C)\oplus TM).
\]  
In particular $\tilde{\mathfrak{h}}$ is a subspace of $\mathcal{G}(U)$ once we identify $z^m$ with $\mathfrak{w}^m$ and $m^\perp$ with $\partial_{m^\perp}$. In order to show how the Lie bracket on $\tilde{\mathfrak{h}}$ is defined, we need to make another assumption on the metric: assume that the metric $h_E$ is constant along the fibres of $\check{X}$, i.e. in an open subset $U\subset M$ $\phi_j=\phi_j(x_1,x_2)$, $j=1,\cdots, r$. Hence, the Chern connection becomes $\nabla^E=d+\hbar A_j(\phi)dz_j$ while the curvature becomes $F_E=\hbar^2 F_{jk}(\phi)dz_j\wedge dz_k$. 
We now show $\mathcal{G}(U)$ is a Lie sub-algebra of $\left(\Omega^0(pr^{-1}(U)),\End E\oplus T\mathcal{L}), \lbrace\cdot,\cdot\rbrace_\sim\right)\subset G(U)$ and we compute the Lie bracket $\lbrace\cdot,\cdot\rbrace_{\sim}$ explicitly on functions of $\mathcal{G}(U)$.

\begin{multline*}
\lbrace \left( A\mathfrak{w}^m, \mathfrak{w}^m\partial_n\right), \left(N\mathfrak{w}^{m'}, \mathfrak{w}^{m'}\partial_{n'}\right)\rbrace_\sim=\Big([A,N]\mathfrak{w}^{m+m'}+\textbf{ad}(\mathfrak{w}^m\partial_n,N\mathfrak{w}^{m'})-\textbf{ad}(\mathfrak{w}^{m'}\partial_{n'},A\mathfrak{w}^{m}),\\
\lbrace\mathfrak{w}^m\partial_n,\mathfrak{w}^{m'}\partial_{n'}\rbrace\Big)
\end{multline*} 
\begin{equation}
\begin{split}
\textbf{ad}(\mathfrak{w}^m\partial_n,N\mathfrak{w}^{m'})_s&=\frac{4\pi}{\hbar}\sum_{k+k'=s}\mathfrak{w}^m\frac{\hbar}{4\pi} n^j\left(\frac{\partial N\mathfrak{w}^{m'}}{\partial x_j}+2\pi im'_j+\hbar A_j(\phi)N\right)\\
&=\mathfrak{w}^{m+m'}n^j\left( 2\pi im'_j+i\hbar\frac{\partial\phi}{\partial x_j}N\right)\\
&=\mathfrak{w}^{m+m'}\left(2\pi i\langle m',n\rangle+i\hbar n^jA_j(\phi)N\right)
\end{split}
\end{equation}
where in the second step we use the fact that $\mathfrak{w}^m$ is not zero only on $U_m$, and in the last step we use the pairing of $\Lambda$ and $\Lambda^*$ given by $\langle m,n'\rangle =\sum_jm_jn'^j $. Thus
\begin{equation*}
\begin{split}
\lbrace \left( A\mathfrak{w}^m, \mathfrak{w}^m\partial_n\right), \left(N\mathfrak{w}^{m'}, \mathfrak{w}^{m'}\partial_{n'}\right)\rbrace_\sim &=([A,N]\mathfrak{w}^{m+m'}+N(2\pi im'_j)n^j\mathfrak{w}^{m+m'}+\\
&+\hbar NA_j(\phi)n^j\mathfrak{w}^{m+m'}-A(2\pi im_j){n'}^j \mathfrak{w}^{m+m'}+\\
&-\hbar A A_j(\phi){n'}^j\mathfrak{w}^{m+m'},\lbrace\mathfrak{w}^m\partial_n,\mathfrak{w}^{m'}\partial_{n'}\rbrace)
\end{split}
\end{equation*} 
Taking the limit as $\hbar\to 0$, the Lie bracket $\lbrace\left( A\mathfrak{w}^m, \mathfrak{w}^m\partial_n\right), \left(N\mathfrak{w}^{m'}, \mathfrak{w}^{m'}\partial_{n'}\right)\rbrace_\sim$ converges to
\begin{equation*}
\left([A,N]\mathfrak{w}^{m+m'}+2\pi i\langle m',n\rangle N\mathfrak{w}^{m+m'}-2\pi i\langle m,n'\rangle A\mathfrak{w}^{m+m'}, \lbrace\mathfrak{w}^m\partial_n,\mathfrak{w}^{m'}\partial_{n'}\rbrace\right)
\end{equation*} 
and we finally recover the definition of the Lie bracket of $[\cdot,\cdot]_{\tilde{\mathfrak{h}}}$ \eqref{eq:Liebracket}, up to a factor of $2\pi i$. Hence $(\tilde{\mathfrak{h}},[\cdot,\cdot]_\sim)$ is the \textit{asymptotic} subalgebra of $\left(\Omega^0(pr^{-1}(U)),\End E\oplus T\mathcal{L}), \lbrace\cdot,\cdot\rbrace_\sim\right)$.

\section{Deformations associated to a single wall diagram}\label{sec:single wall}
In this section we are going to construct a solution of the Maurer-Cartan equation from the data of a single wall. We work locally on a contractible, open affine subset $U\subset M$.

Let $(m,P_m,\theta_m)$ be a wall and assume $\log(\theta_m)=\sum_{j,k}\big(A_{jk}t^j\mathfrak{w}^{km}, a_{jk}t^j\mathfrak{w}^{km}\partial_n\big)$, where $A_{jk}\in{\mathfrak{gl}(r,\C)}$ and $a_{jk}\in\C$, for every $j,k$. 

\begin{notation}
We need to introduce a suitable set of local coordinates on $U$, namely $(u_m,u_{m,\perp})$, where $u_m$ is the coordinate in the direction of $P_m$, while $u_{m^\perp}$ is normal to $P_m$, according with the orientation of $U$. We further define $H_{m,+}$ and $H_{m,-}$ to be the half planes in which $P_m$ divides $U$, according with the orientation. 
\end{notation}

\begin{notation}
We will denote by the superscript \textit{CLM} the elements already introduced in \cite{MCscattering}. 
\end{notation}

\subsection{Ansatz for a wall}\label{sec:ansatz}

Let $\delta_m\defeq\frac{e^{-\frac{u_{m^\perp}^2}{\hbar}}}{\sqrt{\pi \hbar}}du_{m^\perp}$ be a normalized Gaussian one-form, which is \textit{supported} on $P_m$.
Then, let us define \[\Pi\defeq(\Pi_E,\Pi^{CLM})\]
where $\Pi_E=-\sum_{j,k}A_{jk}t^j\delta_m\mathfrak{w}^{km}$ and $\Pi^{CLM}=-\sum_{j,k\geq 1}a_{kj}\delta_mt^j\mathfrak{w}^{km}\partial_n$. 

From section 4 of \cite{MCscattering} we are going to recall the definition of \textit{generalized Sobolev space} suitably defined to compute the asymptotic behaviour of Gaussian k-forms like $\delta_m$ which depend on $\hbar$. 
Let $\Omega_{\hbar}^k(U)$ denote the set of $k$-forms on $U$ whose coefficients depend on the real positive parameter $\hbar$.  
\begin{definition}[Definition 4.15 \cite{MCscattering}]
\[\mathcal{W}_k^{-\infty}(U)\defeq\big\lbrace \alpha\in\Omega_{\hbar}^k(U)\vert \forall q\in U \exists V\subset U\,,q\in V\,\text{s.t.} \sup_{x\in V}\left|\nabla^j\alpha(x)\right|\leq C(j,V)e^{-\frac{c_V}{\hbar}}, C(j,V),\textcolor{britishracinggreen}{c_V}>0\big\rbrace\] is the set of exponential k-forms.
\end{definition}
\begin{definition}[Definition 4.16 \cite{MCscattering}]
\[\mathcal{W}_k^\infty(U)\defeq\big\lbrace \alpha\in\Omega_{\hbar}^k(U)\vert \forall q\in U \exists  V\subset U\,,q\in V\,\text{s.t.} \sup_{x\in V}\left|\nabla^j\alpha(x)\right|\leq C(j,V)\hbar^{-N_{j,V}},\, C(j,V),N_{j,V}\in\Z_{>0}\big\rbrace\] is the set of polynomially growing k-forms.
\end{definition}

\begin{definition}[Definition 4.19 \cite{MCscattering}]
Let $P_m$ be a ray in $U$. The set $\mathcal{W}_{P_m}^s(U)$ of $1$-forms $\alpha$ which have \textit{asymptotic support} of order $s\in\Z$ on $P_m$ is defined by the following conditions:
\begin{enumerate}
\item for every $q_*\in U\setminus P_m$, there is a neighbourhood $V\subset U\setminus P_m$ such that $\alpha|_V\in\mathcal{W}_{1}^{-\infty}(V)$;
\item for every $q_*\in P_m$ there exists a neighbourhood $q_*\in W\subset U$ where in local coordinates $u_q=(u_{q,m},u_{q,m^\perp})$ centred at $q_*$, $\alpha$ decomposes as
\[ 
\alpha=f(u_q,\hbar)du_{q,m^\perp}+\eta
\]
$\eta\in\mathcal{W}_{1}^{-\infty}(W)$ and for all $j\geq 0$ and for all $\beta\in\Z_{\geq 0}$
\begin{equation}\label{eq:stima W^s}
\int_{(0,u_{q,m^\perp})\in W}(u_{m^\perp})^\beta\big(\sup_{(u_{q,m},u_{m^\perp})\in W}\left|\nabla^j(f(u_q,\hbar))\right|\big)du_{m^\perp}\leq C(j,W,\beta)\hbar^{-\frac{j+s-\beta-1}{2}}
\end{equation}
for some positive constant $C(\beta, W,j)$.
\end{enumerate}
\end{definition}

\begin{remark}
A simpler way to figure out what is the space $\mathcal{W}_{P_m}^s(U)$, is to understand first the case of a $1$-form $\alpha\in\Omega_{\hbar}^1(U)$ which depends only on the coordinate $u_{m^\perp}$. Indeed $\alpha=\alpha(u_{m^\perp},\hbar)du_{m^\perp}$ has asymptotic support of order $s$ on a ray $P_m$ if for every $q\in P_m$, there exists a neighbourhood $q\in W\subset U$ such that
\[
\int_{(0,u_{q,m^\perp})\in W} u_{q,m^\perp}^{\beta} \left|\nabla^j \alpha(u_{q,m^\perp},\hbar)\right|du_{q,m^\perp} \leq C(W,\beta,j)\hbar^{-\frac{\beta+s-1-j}{2}}
\]
for every $\beta\in\Z_{\geq 0}$ and $j\geq 0$.

In particular for $\beta=0$ the estimate above reminds to the definition of the usual Sobolev spaces $L_1^j(U)$.  
\end{remark}

\begin{lemma}
The one-form $\delta_m$ defined above, has asymptotic support of order $1$ along $P_m$, i.e. $\delta_m\in\mathcal{W}_{P_m}^1(U)$. 
\end{lemma}
\begin{proof}
We claim that
\begin{equation}\label{eq:estimate}
\int_{-a}^b(u_{m^\perp})^\beta \nabla^j\left(\frac{e^{-\frac{u_{m^\perp}^2}{\hbar}}}{\sqrt{\hbar\pi}}\right)du_{m^\perp}\leq C(\beta,W,j)\hbar^{-\frac{j-\beta}{2}}
\end{equation}
for every $j\geq 0$, $\beta\in\Z_{\geq 0}$, for some $a,b> 0$.
This claim holds for $\beta=0=j$, indeed $\int_{-a}^b\frac{e^{-\frac{u_{m^\perp}^2}{\hbar}}}{\sqrt{\hbar\pi}}du_{m^\perp}$ is bounded by a constant $C=C(a,b)>0$. 

Then we prove the claim by induction on $\beta$, at $\beta=0$ it holds true by the previous computation. Assume that
\begin{equation}
\int_{-a}^b(u_{m\perp})^\beta \frac{e^{-\frac{u_{m^\perp}^2}{\hbar}}}{\sqrt{\hbar\pi}}du_{m^\perp} \leq C(\beta,a,b)\hbar^{\beta/2}
\end{equation} 
holds for $\beta$, then 
\begin{equation}
\begin{split}
\int_{-a}^b(u_{m^\perp})^{\beta+1}\frac{e^{-\frac{u_{m^\perp}^2}{\hbar}}}{\sqrt{\hbar\pi}}du_{m^\perp}&=-\frac{\hbar}{2}\int_{-a}^b(u_{m^\perp})^\beta\left(-2\frac{u_{m^\perp}}{\hbar}\frac{e^{-\frac{u_{m^\perp}^2}{\hbar}}}{\sqrt{\hbar\pi}}\right)du_{m^\perp}\\
&=-\frac{\hbar}{2}\left[(u_{m^\perp})^\beta\frac{e^{-\frac{u_{m^\perp}^2}{\hbar}}}{\sqrt{\hbar\pi}}\right]_{-a}^b +\beta\frac{\hbar}{2}\int_{-a}^b(u_{m^\perp})^{\beta-1}\frac{e^{-\frac{u_{m^\perp}^2}{\hbar}}}{\sqrt{\hbar\pi}}du_{m^\perp}\\
&\leq C(\beta, a,b)\hbar^{\frac{1}{2}}+\tilde{C}(\beta,a,b)\hbar^{1+\frac{\beta-1}{2}}\\
&\leq C(a,b,\beta)\hbar^{\frac{\beta+1}{2}}.
\end{split}
\end{equation}
Analogously let us prove the estimate by induction on $j$. At $j=0$ it holds true, and assume that 
\begin{equation}
\int_{-a}^b(u_{m^\perp})^\beta \nabla^j\left( \frac{e^{-\frac{u_{m^\perp}^2}{\hbar}}}{\sqrt{\hbar\pi}}\right)du_{m^\perp}\leq C(a,b,\beta,j)\hbar^{-\frac{j-\beta}{2}}
\end{equation} 
holds for $j$. Then at $j+1$ we have the following
\begin{equation*}
\begin{split}
\int_{-a}^b(u_{m^\perp})^\beta\nabla^{j+1}\left(\frac{e^{-\frac{u_{m^\perp}^2}{\hbar}}}{\sqrt{\hbar\pi}}\right)du_{m^\perp}&=\left[u_{m^\perp}^{\beta}\nabla^j\left(\frac{e^{-\frac{u_{m^\perp}^2}{\hbar}}}{\sqrt{\hbar\pi}}\right)\right]_{-a}^b-\beta\int_{N_V}(u_{m^\perp})^{\beta-1}\nabla^j\left(\frac{e^{-\frac{u_{m^\perp}^2}{\hbar}}}{\sqrt{\hbar\pi}}\right)du_{m^\perp}\\
&\leq \tilde{C}(\beta,a,b,j)\hbar^{-j-\frac{1}{2}}+ C(a,b,\beta, j)\hbar^{-\frac{j-\beta+1}{2}}\\
&\leq C(a,b,\beta,j)\hbar^{-\frac{j+1-\beta}{2}}
\end{split}
\end{equation*}
This ends the proof.  
\end{proof}

\begin{notation}
We say that a function $f(x,\hbar)$ on an open subset $U\times\R_{\geq 0}\subset M\times\R_{\geq 0}$ belongs to $O_{loc}(\hbar^l)$ if it is bounded by $C_K\hbar^l$ on every compact subset $K\subset U$, for some constant $C_K$ (independent on $\hbar$), $l\in\R$. 
\end{notation}
In order to deal with $0$-forms ``asymptotically supported on $U$'', we define the following space $\mathcal{W}_{0}^{s}$: 
\begin{definition}\label{def:I(P)}
A function $f(u_q,\hbar)\in\Omega^0_{\hbar}(U)$ belongs to $\mathcal{W}_{0}^{s}(U)$ if and only if for every $q_*\in U$ there is a neighbourhood $q_*\in W\subset U$ such that  
\[\sup_{q\in W}\left|\nabla^j f(u_q,\hbar)\right|\leq C(W,j)\hbar^{-\frac{s+j}{2}}\]
for every $j\geq 0$.  

\end{definition}
 
\begin{notation}
Let us denote by $\Omega_{\hbar}^k(U,TM)$ the set of $k$-forms valued in $TM$, which depends on the real parameter $\hbar$ and analogously we denote by $\Omega_{\hbar}^k(U,\End E)$ the set of $k$-forms valued in $\End E$ which also depend on $\hbar$. 
We say that $\alpha=\alpha_K(x,\hbar)dx^K\otimes\partial_n\in\Omega_{\hbar}^k(U,TM)$ belongs to $\mathcal{W}_P^s(U,TM)/ \mathcal{W}_k^\infty(U, TM)/ \mathcal{W}_k^{-\infty}(U, TM)$ if $\alpha_K(x,\hbar)dx^K\in \mathcal{W}_P^s(U)/ \mathcal{W}_k^\infty(U)/ \mathcal{W}_k^{-\infty}(U)$. Analogously we say that $A=A_K(x,\hbar)dx^K\in\Omega_{\hbar}^k(U,\End E)$ belongs to $\mathcal{W}_P^s(U, \End E)/ \mathcal{W}_k^\infty(U, \End E)/ \mathcal{W}_k^{-\infty}(U, \End E)$ if for every $p,q=1,\cdots, r$ then $(A_K)_{ij}(x,\hbar)dx^K\in \mathcal{W}_P^s(U)/ \mathcal{W}_k^\infty(U)/ \mathcal{W}_k^{-\infty}(U)$.
\end{notation}

\begin{prop}\label{rmk:closed}
$\Pi$ is a solution of the Maurer-Cartan equation $d_W\Pi+\frac{1}{2}\lbrace\Pi,\Pi\rbrace_\sim=0$, up to higher order term in $\hbar$, i.e. there exists $\Pi_{E,R}\in\Omega^1(U,\End E\oplus T\mathcal{L})$ such that $\bar{\Pi}\defeq (\Pi_E+\Pi_{E,R}, \Pi^{CLM})$ is a solution of Maurer-Cartan and $\Pi_{E,R}\in\mathcal{W}_{P_m}^{-1}(U)$. 
\end{prop}

\begin{proof}
First of all let us compute $d_W\Pi$:
\[
\begin{split}
d_W\Pi&=\big(d_{W,E}\Pi_E+\hat{B}\Pi^{CLM},d_{W,\mathcal{L}}\Pi^{CLM}\big)\\
&=\big(-A_{jk}t^j\mathfrak{w}^{km}d\delta_m+\hat{B}\Pi^{CLM}, -a_{jk}t^j\mathfrak{w}^{km}d(\delta_m)\otimes\partial_n\big)
\end{split}
\]
and notice that $d(\delta_m)=0$. Then, let us compute $\hat{B}\Pi^{CLM}$:
\[
\begin{split}
\hat{B}\Pi^{CLM}&=\FF(\FF^{-1}(\Pi^{CLM})\lrcorner F_E)\\
&=-\FF\left(\left(\frac{4\pi}{\hbar}\right)^{-1}a_{jk}t^j\textbf{w}^{km}\check{\delta}_m\otimes\check{\partial}_n\lrcorner \left(\hbar^2F_j^q(\phi)dz^j\wedge d\bar{z}_q\right)\right)\\
&=-\left(\frac{4\pi}{\hbar}\right)^{-1}\FF\big(a_{jk}t^j\textbf{w}^{km} n^l\hbar^2F_{jq}(\phi)\check{\delta}_m\wedge d\bar{z}^q\big)\\
&=-\hbar^{2}\Big(\frac{4\pi}{\hbar}\Big)^{-1}\Big(\frac{4\pi}{\hbar}\Big)^{2}a_{jk}t^j\mathfrak{w}^{km} n^lF_{lq}(\phi)\delta_m\wedge dx^q\\
&=-4\pi\hbar (a_{jk}t^j\mathfrak{w}^{km} n^lF_{lq}(\phi)\delta_m\wedge dx^q)
\end{split}
\]
where we denote by $\check{\delta}_m$ the Fourier transform of $\delta_m$. Notice that $\hat{B}\Pi^{CLM}$ is an exact two form, thus since $F_{lq}(\phi)dx^q=dA_l(\phi)$ (recall that the hermitian metric on $E$ is diagonal)we define 
\[\Pi_{E,R}\defeq 4\pi\hbar (a_{jk}t^j\mathfrak{w}^{km} n^lA_l(\phi)\delta_m)\] 
i.e. as a solution of $d_W\Pi_{E,R}=-\hat{B}\Pi^{CLM}$.

In particular, since $\delta_m\in\mathcal{W}_{P_m}^1(U)$ then $\hbar\delta_m\in\mathcal{W}_{P_m}^{-1}(U)$. Therefore $\Pi_{E,R}$ has the expected asymptotic behaviour and $d_W\bar{\Pi}=0$.  
Let us now compute the commutator:
\[
\begin{split}
\lbrace\bar{\Pi}, \bar{\Pi}\rbrace_\sim&=\big(2\FF\big(\FF^{-1}\Pi^{CLM}\lrcorner\nabla^E\FF^{-1}(\Pi_E+\Pi_{E,R})\big)+\lbrace\Pi_E+\Pi_{E,R},\Pi_E+\Pi_{E,R}\rbrace,\lbrace\Pi^{CLM},\Pi^{CLM}\rbrace\big)\\
&=\big(2\FF\big(\FF^{-1}\Pi^{CLM}\lrcorner\nabla^E\FF^{-1}(\Pi_E+\Pi_{E,R})\big)+2(\Pi_E+\Pi_{E,R})\wedge(\Pi_E+\Pi_{E,R}),0\big)
\end{split}
\] 
Notice that, since both $\Pi_E$ and $\Pi_{E,R}$ are matrix valued one forms where the form part is given by $\delta_m$, the wedge product $(\Pi_E+\Pi_{E,R})\wedge(\Pi_E+\Pi_{E,R})$ vanishes as we explicitly compute below
\[
\begin{split}
&(\Pi_E+\Pi_{E,R})\wedge(\Pi_E+\Pi_{E,R})=A_{jk}A_{rs}t^{j+r}\mathfrak{w}^{km+sm}\delta_m\wedge\delta_m+\\
&+8\pi\hbar a_{jk}t^{j+r}\mathfrak{w}^{km+sm} n^lA_l(\phi)A_{rs}\delta_m\wedge\delta_m+4\pi\hbar (a_{jk}t^j\mathfrak{w}^{km} n^lA_l(\phi))^2\delta_m\wedge\delta_m=0.
\end{split}
\]
 
Hence we are left to compute $\FF\left(\FF^{-1}\Pi^{CLM}\lrcorner\nabla^E\FF^{-1}\left((\Pi_E+\Pi_{E,R})\right)\right)$:
\[
\begin{split}
&\FF\big(\FF^{-1}\Pi^{CLM}\lrcorner\nabla^E\FF^{-1}(\Pi_E+\Pi_{E,R})\big)=\\
&=\FF\bigg(\Big(\frac{4\pi}{\hbar}\Big)^{-1}a_{jk}t^j\textbf{w}^{km}\check{\delta}_m\check{\partial}_n\lrcorner d \Big(\Big(\frac{4\pi}{\hbar}\Big)^{-1}\big(At\textbf{w}^m\check{\delta}_m+4\pi\hbar a_{rs}t^r\textbf{w}^{sm}n^lA_l(\phi)\check{\delta}_m\big)\Big)+\\
&\quad+\Big(\frac{4\pi}{\hbar}\Big)^{-1}a_{jk}t^j\textbf{w}^{km}\check{\delta}_m\check{\partial}_n\lrcorner \Big(i\hbar A_q(\phi)dz^q\wedge \Big(\Big(\frac{4\pi}{\hbar}\Big)^{-1}\big(At\textbf{w}^m\check{\delta}_m+4\pi\hbar a_{rs}t^r\textbf{w}^{sm}n^lA_l(\phi)\check{\delta}_m\big)\Big)\Big)\bigg)\\
&=\Big(\frac{4\pi}{\hbar}\Big)^{-1}\FF\bigg(a_{jk}t^j\textbf{w}^{km}\check{\delta}_m\check{\partial}_n\lrcorner\Big(At\partial_l(\textbf{w}^m)dz^l\wedge\check{\delta}_m+At\textbf{w}^md(\check{\delta}_m)+\\
&\quad+4\pi\hbar a_{rs}t^r\partial_l(n^qA_q(\phi)\textbf{w}^{sm})\check{\delta}_m+ 4\pi\hbar a_{rs}t^rn^qA_q(\phi)\textbf{w}^{sm}d(\check{\delta}_m)\Big)\bigg)\\
&=\Big(\frac{4\pi}{\hbar}\Big)^{-1}\FF\bigg(a_{jk}At^{j+1}\textbf{w}^{km}\check{\delta}_mn^l\partial_l(\textbf{w}^m)\wedge\check{\delta}_m+a_{jk}At^{j+1}\textbf{w}^{km+m}n^l\gamma_l(i\hbar^{-1}\gamma_pz^p)\check{\delta}_m\wedge\check{\delta}_m+\\
&\quad+4\pi\hbar a_{jk}a_{rs}t^{j+r}\textbf{w}^{km}\check{\delta}_m n^l\partial_l(n^qA_q(\phi)\textbf{w}^{sm})\check{\delta}_m+\\
&\quad+4\pi\hbar a_{jk}a_{rs}t^{j+r}\textbf{w}^{km+sm}n^qA_q(\phi)n^l\gamma_l(i\hbar^{-1}\gamma_pz^p)\check{\delta}_m\wedge\check{\delta}_m\bigg)\\
&=0
\end{split}
\]  
where $\check{\delta}_m=\frac{e^{-\frac{u_{m^\perp}^2}{\hbar}}}{\sqrt{\pi\hbar}}\gamma_pd\bar{z}^p$ for some constant $\gamma_p$ such that $u_{m^\perp}=\gamma_1x^1+\gamma_2 x^2$, and $\partial_l$ is the partial derivative with respect to the coordinate $z^l$. In the last step we use that $\check{\delta}_m\wedge\check{\delta}_m=0$.  
\end{proof}
\begin{remark}
In the following it will be useful to consider $\bar{\Pi}$ in order to compute the solution of Maurer-Cartan from the data of two non-parallel walls (see section \eqref{sec:two_walls}). However, in order to compute the asymptotic behaviour of the gauge it is enough to consider $\Pi$. 
\end{remark}
Since $\check{X}(U)\cong U\times \C/\Lambda$ has no non trivial deformations and $E$ is holomorphically trivial, then also the pair $(\check{X}(U),E)$ has no non trivial deformations. Therefore there is a gauge $\varphi\in\Omega^0(U, \End E\oplus TM)[\![ t ]\!]$ such that 
\begin{equation}
e^\varphi\ast 0=\bar{\Pi}
\end{equation}
namely $\varphi$ is a solution of the following equation
\begin{equation}\label{eq:gauge}
d_W\varphi=-\bar{\Pi}-\sum_{k\geq 0}\frac{1}{(k+1)!}\textsf{ad}_{\varphi}^kd_W\varphi.
\end{equation}
In particular the gauge $\varphi$ is not unique, unless we choose a gauge fixing condition (see Lemma \ref{lem:uniq}). In order to define the gauge fixing condition we introduce the so called homotopy operator.
  
\subsubsection{Gauge fixing condition and homotopy operator}  

Since $\mathcal{L}(U)=\bigsqcup_{m\in\Lambda}U_m$, it is enough to define the homotopy operator $H_m$ for every frequency $m$. Let us first define morphisms $p\defeq\bigoplus_{m\in \Lambda\setminus\lbrace 0\rbrace}p_m$ and $\iota\defeq\bigoplus_{m\in\Lambda\setminus\lbrace 0\rbrace}\iota_m$. We define $p_m:\mathfrak{w}^m\cdot\Omega^{\bullet}(U)\rightarrow\mathfrak{w}^m\cdot H^{\bullet}(U)$ which acts as $p_m(\alpha \mathfrak{w}^m)=\alpha({q_0})\mathfrak{w}^m$ if $\alpha\in\Omega^0(U)$ and it is zero otherwise. 

Then $\iota_m\colon\mathfrak{w}^m\cdot H^{\bullet}(U)\rightarrow\mathfrak{w}^m\Omega^{\bullet}(U)$ is the embedding of constant functions on $\Omega^{\bullet}(U)$ at degree zero, and it is zero otherwise. Then let $q_0\in H_{-}$ be a fixed base point, then since $U$ is contractible, there is a homotopy $\varrho\colon [0,1]\times U\to U$ which maps $(\tau,u_m,u_{m,\perp})$ to $(\varrho_1(\tau,u_m,u_{m,\perp}),\varrho_2(\tau,u_m,u_{m,\perp}))$ and such that $\varrho(0,\cdot)=q_0=(u_0^1,u_0^2)$ and $\varrho(1,\cdot)=\text{Id}$. We define $H_m$ as follows:
\begin{equation}\label{def:homotopy_1wall}
\begin{split}
H_m &\colon\mathfrak{w}^m\cdot \Omega^{\bullet}(U)\rightarrow \mathfrak{w}^m\cdot\Omega^{\bullet}(U)[-1]\\
&H_m(\mathfrak{w}^m\alpha)\defeq\mathfrak{w}^m\int_0^1 d\tau\wedge\frac{\partial}{\partial \tau}\lrcorner\varrho^*(\alpha)
\end{split}
\end{equation} 

\begin{lemma}
The morphism $H$ is a homotopy equivalence of $\text{id}_{\Omega^{\bullet}}$ and $\iota\circ p$, i.e. the identity
\begin{equation}\label{eq:PHi}
\text{id}-\iota\circ p=d_WH+Hd_W
\end{equation}
holds true.
\end{lemma}
\begin{proof}
At degree zero, let $f\in\Omega^0(U)$: then $\iota_m\circ p_m(f\mathfrak{w}^m)=f(q_0)\mathfrak{w}^m$. By degree reason $H_m(f\mathfrak{w}^m)=0$ and 
\begin{equation*}
H_md_W(\mathfrak{w}^mf)=\mathfrak{w}^m\int_0^1d\tau\wedge\frac{\partial}{\partial\tau}\lrcorner(d_M (f(\varrho))+d\tau\frac{\partial f(\varrho)}{\partial\tau})=\mathfrak{w}^m\int_0^1d\tau\frac{\partial f(\varrho)}{\partial\tau}=\mathfrak{w}^m(f(q)-f(q_0)).
\end{equation*} 
At degree $k=1$, let $\alpha=f_idx^i\in \Omega^1(U)$ then: $\iota_m\circ p_m(\alpha\mathfrak{w}^m)=0$,
\begin{equation*}
\begin{split}
H_md_W(\alpha\mathfrak{w}^m)&=\mathfrak{w}^m\int_0^1d\tau\wedge\frac{\partial}{\partial\tau}\lrcorner (d(\varrho^*(\alpha))\\
&=\mathfrak{w}^m\int_0^1d\tau\wedge\frac{\partial}{\partial\tau}\lrcorner (d_M(\varrho^*(\alpha))+d\tau\wedge\frac{\partial}{\partial\tau}(f_i(\varrho)\frac{\partial\varrho_i}{\partial x^i})dx^i)\\
&=-\mathfrak{w}^md_M\left(\int_0^1d\tau\wedge\frac{\partial}{\partial\tau}\lrcorner (\varrho^*(\alpha))\right)+\mathfrak{w}^m\int_0^1d\tau \frac{\partial}{\partial\tau}(f_i(\varrho)\frac{\partial\varrho_i}{\partial x^i})dx^i\\
&=-\mathfrak{w}^md_M\left(\int_0^1d\tau\wedge\frac{\partial}{\partial\tau}\lrcorner (\varrho^*(\alpha))\right)+\mathfrak{w}^m\alpha
\end{split}
\end{equation*}
and 
\begin{equation*}
d_WH_m(\mathfrak{w}^m\alpha)=\mathfrak{w}^md_M\left(\int_0^1d\tau\wedge\frac{\partial}{\partial\tau}\lrcorner\varrho^*(\alpha)\right).
\end{equation*}
Finally let $\alpha\in\Omega^2(U)$, then $p_m(\alpha\mathfrak{w}^m)=0$ and $d_W(\alpha\mathfrak{w}^m)=0$. Then it is easy to check that $d_WH_m(\mathfrak{w}^m\alpha)=\alpha$.
\end{proof}

\begin{lemma}[Lemma 4.7 in \cite{MCscattering}]\label{lem:uniq}
Among all solution of $e^{\varphi}\ast 0=\bar{\Pi}$, there exists a unique one such that $p(\varphi)=0$. 
\end{lemma} 
\begin{proof}
First of all, let $\sigma\in\Omega^0(U)$ such that $d\sigma=0$. Then $e^{\varphi\bullet\sigma}\ast 0=\bar{\Pi}$, indeed $e^{\sigma}\ast 0=0-\sum_{k}\frac{[\sigma,\cdot]^k}{k!}(d\sigma)=0$. Thus $e^{\varphi\bullet\sigma}\ast 0=e^{\varphi}\ast(e^{\sigma}\ast 0)=e^{\varphi}\ast 0=\bar{\Pi}$. Thanks to the BCH formula \[\varphi\bullet\sigma=\varphi+\sigma+\frac{1}{2}\lbrace\varphi,\sigma\rbrace_\sim +\cdots\] we can uniquely determine $\sigma$ such that $p(\varphi\bullet\sigma)=0$. Indeed working order by order in the formal parameter $t$, we get:
\begin{enumerate}
\item $p(\sigma_1+\varphi_1)=0$, hence by definition of $p$, $\sigma_1(q_0)=-\varphi_1(q_0)$;
\item $p(\sigma_2+\varphi_2+\frac{1}{2}\lbrace\varphi_1,\sigma_1\rbrace_\sim)=0$, hence $\sigma_2(q_0)=-\big(\varphi_2(q_0)+\frac{1}{2}\lbrace\varphi_1,\sigma_1\rbrace_\sim (q_0)\big)$;
\end{enumerate} 
and any further order is determined by the previous one. 
\end{proof}

Now that we have defined the homotopy operator and the gauge fixing condition (as in Lemma \ref{lem:uniq}), we are going to study the asymptotic behaviour of the gauge $\varphi$ such that it is a solution of \eqref{eq:gauge} and $p(\varphi)=0$. Equations \eqref{eq:gauge}, \eqref{eq:PHi} and $p(\varphi)=0$ together say that the unique gauge $\varphi$ is indeed a solution of the following equation:
\begin{equation}\label{eq:sol_gauge}
\varphi=-Hd_W(\varphi)=-H\big(\bar{\Pi}+\sum_{k}{\textsf{ad}_{\varphi}^k\over (k+1)!}d_W\varphi\big).
\end{equation}

Up to now we have used a generic homotopy $\varrho$, but from now on we are going to choose it in order to get the expected asymptotic behaviour of the gauge $\varphi$. In particular we choose the homotopy $\varrho$ as follows: for every $q=(u_{q,m},u_{q,m^\perp})\in U$   
\begin{equation}\label{varro}
\varrho(\tau, u_q)=\begin{cases}
\left((1-2\tau)u_{1}^0+2\tau u_{q,m},u_{2}^0\right) & \text{if} \tau\in [0,\frac{1}{2}] \\
\left(u_{q,m},(2\tau-1)u_{q,m^\perp}+(2-2\tau)u_{2}^0\right) & \text{if} \tau\in [\frac{1}{2},1]
\end{cases}
\end{equation}

where $(u_0^1, u^2_0)$ are the coordinates for the fixed point $q_0$ on $U$. Then we have the following result:

\begin{lemma}\label{lem:H(W^s)}
Let $P_m$ be a ray in $U$ and let $\alpha\in\mathcal{W}_{P_m}^s(U)$. Then $H(\alpha\mathfrak{w}^m)$ belongs to $\mathcal{W}_0^{s-1}(U)$. 
\end{lemma}
\begin{proof}
Let us first consider $q_*\in U\setminus P_m$. By assumption there is a neighbourhood of $q_*$, $V\subset U$ such that $\alpha\in\mathcal{W}_1^{-\infty}(V)$. Then by definition  
\begin{equation*}
H(\alpha\mathfrak{w}^m)=\mathfrak{w}^m\int_0^1d\tau\wedge\frac{\partial}{\partial \tau}\lrcorner\varrho^*(\alpha)=\int_0^1d\tau\alpha(\varrho)\left(\frac{\partial\varrho_1}{\partial\tau}+\frac{\partial\varrho_2}{\partial\tau} \right)
\end{equation*}
hence, since $\varrho$ does not depend on $\hbar$ 
\begin{equation*}
\sup_{q\in V}\nabla^j\left|\int_0^1d\tau\alpha(\varrho)\left(\frac{\partial\varrho_1}{\partial\tau}+\frac{\partial\varrho_2}{\partial\tau} \right)\right|\leq\int_0^1d\tau\sup_{q\in V}\left|\nabla^j(\alpha(\varrho)\left(\frac{\partial\varrho_1}{\partial\tau}+\frac{\partial\varrho_2}{\partial\tau} \right))\right|\leq C(V,j)e^{-\frac{c_v}{\hbar}}.
\end{equation*}

Let us now consider $q_*\in P_m$. By assumption there is a neighbourhood of $q$, $W\subset U$ such that for all $q=(u_{q,m},u_{q,m^\perp})\in W$  $\alpha=h(u_q, \hbar)du_{m_q^\perp}+\eta$ and $\eta\in\mathcal{W}_{1}^{-\infty}(W)$. By definition 

\begin{equation*}
\begin{split}
H(\alpha\mathfrak{w}^m)&=\mathfrak{w}^m\int_0^1d\tau\wedge\frac{\partial}{\partial \tau}\lrcorner\varrho^*(\alpha)\\
&=2\int_{\frac{1}{2}}^1d\tau h(u_{q,m},(2\tau-1)u_{q,m^\perp}+(2-2\tau)u_0^2)(u_{q,m^\perp}-u_0^2)+ \int_0^1d\tau\eta(\varrho)\frac{\partial\varrho_1}{\partial\tau}\\
&=\int_{u_0^2}^{u_{q,m^\perp}}du_m^\perp h(u_{q,m},u_{m^\perp})+ \int_0^1d\tau\eta(\varrho)\frac{\partial\varrho_1}{\partial\tau}
\end{split}
\end{equation*} 

and since $\eta\in\mathcal{W}_1^{-\infty}(W)$ the second term $\int_0^1d\tau\eta(\varrho)\frac{\partial\varrho_1}{\partial\tau}$ belongs to $\mathcal{W}_0^{\infty}$. The first term is computed below:

\begin{equation}
\begin{split}
\sup_{q\in W}\left|\nabla^j\left(\int_{u_0^2}^{u_{q,m^\perp}}du_m^\perp h(u_{q,m},u_{m^\perp})\right)\right|&=\sup_{q\in W}\Big|\int_{u_0^2}^{u_{q,m^\perp}}du_m^\perp \nabla^j(h(u_{q,m},u_{m^\perp}))+\\
&\quad+\left[\frac{\partial^{j-1}}{\partial u_{m^\perp}^{j-1}}(h(u_{q,m},u_{m^\perp)})\right]_{u_{m^\perp}=u_{q,m^\perp}}\Big|\\
&\leq \sup_{u_{q,m^\perp}}\int_{u_0^2}^{u_{q,m^\perp}}du_m^\perp \sup_{u_{q,m}}\left|\nabla^j(h(u_{q,m},u_{m^\perp}))\right|+\\
&\quad +\left[\sup_{q\in W}\left|\frac{\partial^{j-1}}{\partial u_{m^\perp}^{j-1}}(h(u_{q,m},u_{m^\perp)})\right|\right]_{u_{m^\perp}=u_{0}^2}\\
&\leq C(j,W)\hbar^{-\frac{s+j-1}{2}}
\end{split}
\end{equation}
where in the last step we use that $\left[\frac{\partial^{j-1}}{\partial u_{m^\perp}^{j-1}}(h(u_{q,m},u_{m^\perp)})\right]_{u_{m^\perp}=u_{0}^2} $ is outside the support of $P_m$.

\end{proof}

\begin{corollary}\label{lem:H(delta_m)}
Let $P_m$ be a ray in $U$, then $H(\delta_m\mathfrak{w}^m)\in\mathcal{W}_{0}^{0}(U)\mathfrak{w}^m$.
\end{corollary}

\subsection{Asymptotic behaviour of the gauge $\varphi$}
We are going to compute the asymptotic behaviour of $\varphi=\sum_j\varphi^{(j)}t^j\in\Omega^0(U,\End E\oplus TM)[\![ t ]\!]$ order by order in the formal parameter $t$. In addition since $\Pi_{E,R}$ gives a higher $\hbar$-order contribution in the definition of $\bar{\Pi}$ we get rid of it by replacing $\bar{\Pi}$ with $\Pi$ in equation \eqref{eq:sol_gauge}. 

\begin{prop}\label{prop:asymp1}
Let $(m,P_m,\theta_m)$ be a wall with $\log\theta_m=\sum_{j,k\geq 1}\big(A_{jk}t^j\mathfrak{w}^{km},a_{jk}\mathfrak{w}^{km}t^j\partial_n\big)$. Then, the unique gauge $\varphi=(\varphi_E,\varphi^{CLM})$ such that $e^\varphi\ast 0=\Pi$ and $P(\varphi)=0$, has the following  asymptotic jumping behaviour along the wall, namely 
\begin{equation}\label{eq:lim gauge}
\varphi^{(s+1)}\in\begin{cases} \sum_{k\geq 1}\big(A_{s+1,k}t^{s+1}\mathfrak{w}^{km},a_{s+1,k}\mathfrak{w}^{km}t^{s+1}\partial_n\big)+\bigoplus_{k\geq 1}\mathcal{W}_{0}^{-1}(U,\End E\oplus TM)\mathfrak{w}^{km}t^{s+1} & H_{m,+}\\
\bigoplus_{k\geq 1}\mathcal{W}_0^{-\infty}(U, \End E\oplus TM)\mathfrak{w}^{km} t^{s+1}& H_{m,-}.
\end{cases}
\end{equation} 
\end{prop}


Before giving the proof of Proposition \ref{prop:asymp1}, let us introduce the following Lemma which are useful to compute the asymptotic behaviour of one-forms asymptotically supported on a ray $P_m$.
\begin{lemma}\label{lem:wedge_sobolev}
Let $P_m$ be a ray in $U$. Then $\mathcal{W}_{P_m}^s(U)\wedge\mathcal{W}_{0}^{r}(U)\subset\mathcal{W}^{r+s}_{P_{m}}(U)$.
\end{lemma}
\begin{proof}
Let $\alpha\in\mathcal{W}_{P_m}^s(U)$ and let $f\in\mathcal{W}_{0}^{r}(U)$. Pick a point $q_*\in P_m$ and let $W\subset U$ be a neighbourhood of $q_*$ where $\alpha=h(u_q,\hbar)du_{m^\perp}+\eta$, we claim 
\begin{equation}
\int_{-a}^b u_{m^\perp}^\beta\sup_{u_m}\left|\nabla^j(h(u_q,\hbar)f(u_q\hbar))\right|du_{m^\perp}\leq C(a,b,j,\beta)\hbar^{-\frac{r+s+j-\beta-1}{2}}
\end{equation}
for every $\beta\in\Z_{\geq 0}$ and for every $j\geq0$.

\begin{equation*}
\begin{split}
&\int_{-a}^b u_{m^\perp}^{\beta} \sup_{u_m}\left|\nabla^j\left(h(u_q,\hbar)f(u_q\hbar)\right)\right|du_{m^\perp}=\\
&=\sum_{j_1+j_2=j}\int_{-a}^b u_{m^\perp}^{\beta}  \sup_{u_m}\left|\nabla^{j_1}(h(u_q,\hbar))\nabla^{j_2}(f(u_q\hbar))\right|du_{m^\perp}\\
&\leq \sum_{j_1+j_2=j}C(a,b,j_2)\hbar^{-\frac{r+j_2}{2}}\int_{-a}^b u_{m^\perp}^{\beta} \sup_{u_m}\left|\nabla^{j_1}(h(u_q,\hbar))\right|du_{m^\perp}\\
&\leq \sum_{j_1+j_2=j}C(a,b,j_2,j_1)\hbar^{-\frac{r+j_2}{2}}\hbar^{-\frac{s+j_1-\beta-1}{2}}\\
&\leq C(a,b,j)\hbar^{-\frac{r+s+j-\beta-1}{2}}
\end{split}
\end{equation*}
Finally, since $\eta\in\mathcal{W}_1^{-\infty}(W)$ also $f(x,\hbar)\eta$ belongs to $\mathcal{W}_1^{-\infty}(W)$.  
\end{proof}

\begin{lemma}\label{lem:bracket} 
Let $P_{m}$ be a ray in $U$. If $(A\mathfrak{w}^{m}, \varphi\mathfrak{w}^{m}\partial_{n})\in\mathcal{W}_{P_{m}}^{r}(U,\End E\oplus TM)\mathfrak{w}^{m}$ for some $r\geq 0$ and $(T\mathfrak{w}^{m}, \psi\mathfrak{w}^{m}\partial_{n})\in\mathcal{W}_{0}^{s}(U,\End E\oplus TM)\mathfrak{w}^{m}$ for some $s\geq$, then 
\begin{equation}
\lbrace (A\mathfrak{w}^{m}, \alpha\mathfrak{w}^{m}\partial_{n}),(T\mathfrak{w}^{m}, f\mathfrak{w}^{m}\partial_{n})\rbrace_\sim\in \mathcal{W}_{P_{m}}^{r+s}(U,\End E\oplus TM)\mathfrak{w}^{2m}.
\end{equation}
\end{lemma}
\begin{proof}
We are going to prove the following:
\begin{equation*}
\begin{split}
(1)&\lbrace A\mathfrak{w}^{m},T\mathfrak{w}^{m}\rbrace_{\End E}\subset \mathcal{W}_{P_{m}}^{r+s}(U,\End E)\mathfrak{w}^{2m} \\
(2)&\textbf{ad}\left(\varphi\mathfrak{w}^{m}\partial_{n}, T\mathfrak{w}^{m}\right)\subset \mathcal{W}_{P_{m}}^{r+s-1}(U,\End E)\mathfrak{w}^{2m}\\
(3)&\textbf{ad}\left(\psi\mathfrak{w}^{m}\partial_{n}, A\mathfrak{w}^{m}\right)\subset \mathcal{W}_{P_{m}}^{r+s-1}(U,\End E)\mathfrak{w}^{2m}\\
(4)&\lbrace\varphi\mathfrak{w}^{m}\partial_{n},\psi\mathfrak{w}^{m}\partial_{n}\rbrace\subset \mathcal{W}_{P_{m}}^{r+s}(U, TM)\mathfrak{w}^{m}.
\end{split}
\end{equation*}  
The first one is a consequence of Lemma \ref{lem:wedge_sobolev}, indeed by definition \[\lbrace A_k(x)dx^k, T(x)\rbrace_{\End E}=[A_k(x), T(x)]dx^k\]
which is an element in $\End E$ with coefficients in $\mathcal{W}_{P_{m}}^{r+s}(U)$.
 
The second one is less straightforward and need some explicit computations to be done. 
\[
\begin{split}
&\textbf{ad}(\varphi_k(x)\mathfrak{w}^mdx^k\partial_n,T(x)\mathfrak{w}^m)=\FF\big(\FF^{-1}(\varphi_k(x)\textbf{w}^m dx^k\otimes\partial_n)\lrcorner\nabla^E\FF^{-1}(T(x)\mathfrak{w}^m)\big)\\
&=\FF\bigg(\Big(\frac{4\pi}{\hbar}\Big)^{-1}\varphi_k(x)\textbf{w}^m d\bar{z}^k\otimes\check{\partial}_n\lrcorner\nabla^E(T(x)\textbf{w}^m)\bigg)\\
&=\Big(\frac{4\pi}{\hbar}\Big)^{-1}\FF\bigg(\varphi_k(x)\textbf{w}^m d\bar{z}^k\check{\partial}_n\lrcorner\Big(\partial_j(T(x)\textbf{w}^m)dz^j \Big)\bigg)\\
&\quad+i\hbar\Big(\frac{4\pi}{\hbar}\Big)^{-1}\FF\bigg(\varphi_k(x)\textbf{w}^m d\bar{z}^k\check{\partial}_n\lrcorner\Big(A_j(\phi)T(x) \textbf{w}^mdz^j\Big)\bigg)\\
&=\Big(\frac{4\pi}{\hbar}\Big)^{-1}\FF\bigg(\varphi_k(x)\textbf{w}^md\bar{z}^kn^l\frac{\partial}{\partial z_l}\lrcorner\Big(\partial_j(T(x))\textbf{w}^{m}dz^j +m_jA(x)\textbf{w}^{m}dz^j\Big)\bigg)\\
&\quad+i\hbar\Big(\frac{4\pi}{\hbar}\Big)^{-1}\FF\bigg(\varphi_k(x)\textbf{w}^m d\bar{z}^kn^l\frac{\partial}{\partial z_l}\lrcorner\Big(A_j(\phi)T(x)\textbf{w}^mdz^j\Big)\bigg)\\
&=\Big(\frac{4\pi}{\hbar}\Big)^{-1}\FF\bigg(\varphi_k(x)d\bar{z}^k
n^lm_lT(x)\textbf{w}^{2m}\bigg)+\\
&\quad+i\hbar\Big(\frac{4\pi}{\hbar}\Big)^{-1}\FF\bigg(\varphi_k(x)d\bar{z}^k\big(n^lT(x)A_l(\phi)+n^l\frac{\partial T(x)}{\partial x^l} \big)\textbf{w}^{2m}\Big)\bigg)\\
&=i\hbar \varphi_k(x)
\big(n^lT(x)A_l(\phi)+n^l\frac{\partial T(x)}{\partial x^l}\big) \mathfrak{w}^{2m} dx^k
\end{split}
\]
Notice that as a consequence of Lemma \ref{lem:wedge_sobolev}, $\hbar\varphi_k(x)dx^k A_l(\phi)T(x)\in\mathcal{W}_{P_{m}}^{s+r-2}(U)$ while $\hbar\varphi_k(x)\frac{\partial T(x)}{\partial x^l}dx^k\in\mathcal{W}_{P_m}^{r+s-1}$.

The third one is
\begin{align*}
&\textbf{ad}(\psi(x)\mathfrak{w}^m\partial_n,A_k(x)\mathfrak{w}^mdx^k)=\\
&=\FF\big(\FF^{-1}(\psi(x)\partial_n)\lrcorner\nabla^E\FF^{-1}(A_k(x)\mathfrak{w}^mdx^k)\big)\\
&=\FF\bigg(\psi(x)\textbf{w}^m\check{\partial}_n\lrcorner\nabla^E(\Big(\frac{4\pi}{\hbar}\Big)^{-1}A_k(x)\textbf{w}^md\bar{z}^k)\bigg)\\
&=\Big(\frac{4\pi}{\hbar}\Big)^{-1}\FF\bigg(\psi(x)\textbf{w}^m\check{\partial}_n\lrcorner\Big(\partial_j(A_k(x)\textbf{w}^m)\Big)dz^j\wedge d\bar{z}^k\bigg)\\
&\quad+i\hbar\Big(\frac{4\pi}{\hbar}\Big)^{-1}\FF\bigg(\psi(x)\textbf{w}^m\check{\partial}_n\lrcorner\Big(A_j(\phi)A_k(x) \textbf{w}^m\Big)\wedge d\bar{z}^k\bigg)\\
&=\Big(\frac{4\pi}{\hbar}\Big)^{-1}\FF\bigg(\psi(x)\textbf{w}^mn^l\frac{\partial}{\partial z_l}\lrcorner\Big(\partial_j(A_k(x))\textbf{w}^mdz^j+m_jA_k(x)\textbf{w}^m dz^j\Big)\wedge d\bar{z}^k\bigg)\\
&\quad+i\hbar\Big(\frac{4\pi}{\hbar}\Big)^{-1}\FF\bigg(\psi(x)\textbf{w}^m\check{\partial}_n\lrcorner\Big(A_j(\phi)A_k(x) \textbf{w}^mdz^j\Big)\wedge d\bar{z}^k\bigg)\\
&=\Big(\frac{4\pi}{\hbar}\Big)^{-1}\FF\bigg(\psi(x)\textbf{w}^{2m}n^lm_lA(x) d\bar{z}^k\bigg)+\\
&\quad+i\hbar\Big(\frac{4\pi}{\hbar}\Big)^{-1}\FF\bigg(\psi(x)\textbf{w}^{2m}n^l
A_k(x)A_l(\phi)+\psi(x)n^l\frac{\partial A_k(x)}{\partial x^l}\textbf{w}^{2m}d\bar{z}^k\bigg)\\
&=i\hbar \psi(x)\Big(n^lA_k(x)A_l(\phi)+n^l\frac{\partial A_k(x)}{\partial x^l}\Big)\mathfrak{w}^{2m}dx^k
\end{align*}

Notice that $\hbar\psi(x)A_k(x)A(\phi)dx^k\in\mathcal{W}_{P_m}^{r+s-2}(W)$ and $\hbar\psi(x)\frac{\partial A_k(x)}{\partial x^l}dx^k\in\mathcal{W}_{P_m}^{r+s-1}(W)$.

In the end $\lbrace\varphi\mathfrak{w}^{m}\partial_{n},\psi\mathfrak{w}^{m}\partial_{n}\rbrace$ is equal to zero, indeed by definition 
\[
\lbrace\varphi_{k}(x)dx^k{\partial_n}, \psi(x)\partial_n \rbrace=\left(\varphi_{k}dx^k\wedge\psi\right)[\mathfrak{w}^m\partial_n, \mathfrak{w}^m\partial_n]
\]
and $[\mathfrak{w}^m\partial_n,\mathfrak{w}^m\partial_n]=0$.
\end{proof}

\begin{proof}{(Proposition \ref{prop:asymp1})}

It is enough to show that for every $s\geq 0$
\begin{equation}
\label{claim}
\Big(\sum_{k\geq 1}{\ad_{\varphi^s}^k\over (k+1)!} d_W\varphi^s\Big)^{(s+1)}\in\mathcal{W}_{P_m}^{0}(U,\End E\oplus TM),
\end{equation}
where $\varphi^s=\sum_{j=1}^s\varphi^{(j)}t^j$.
Indeed from equation \eqref{eq:sol_gauge}, at the order $s+1$ in the formal parameter $t$, the solution $\varphi^{(s+1)}$ is:
\begin{equation}
\varphi^{(s+1)}=-H(\Pi^{(s+1)})-H\left(\left(\sum_{k\geq 1}{\ad_{\varphi^s}^k\over (k+1)!} d_W\varphi^s\right)^{(s+1)}\right).    
\end{equation}

In particular, if we assume equation \eqref{claim} then by Lemma \ref{lem:H(W^s)} \[H\left(\left(\sum_{k\geq 1}{\ad_{\varphi^s}^k\over (k+1)!} d_W\varphi^s\right)^{(s+1)}\right)\in\mathcal{W}_{0}^{-1}(U,\End E\oplus TM).\] By definition of $H$, \[H(\Pi^{(s+1)})=\sum_{k}(A_{s+1,k}t^{s+1}H(\mathfrak{w}^{km}\delta_m),a_{s+1,k}t^{s+1}H(\mathfrak{w}^{km}\delta_m)\partial_n)\] and by Corollary \ref{lem:H(delta_m)} $H(\delta_m\mathfrak{w}^{km})\in\mathcal{W}_{0}^{0}(U, \End E\oplus TM)\mathfrak{w}^{km}$ for every $k\geq 1$. Hence $H(\Pi^{(s+1)})$ is the leading order term and $\varphi^{(s+1)}$ has the expected asymptotic behaviour.   

Let us now prove the claim \eqref{claim} by induction on $s$. At $s=0$,
\begin{equation}
\varphi^{(1)}=-H(\Pi^{(1)})
\end{equation}

and there is nothing to prove. Assume that \eqref{claim} holds true for $s$, then at order $s+1$ we get contributions for every $k=1,\cdots,s$. Thus let start at $k=1$ with $\ad_{\varphi^s}d_W\varphi^s$:
\begin{equation}
\begin{split}
\ad_{\varphi^s}d_W\varphi^s&=\lbrace\varphi^s,d_W\varphi^s\rbrace_\sim\\
&\in\lbrace H(\Pi^s)+\mathcal{W}_{0}^{-1}(U), \Pi^s+\mathcal{W}_{P_m}^0(U)\rbrace_\sim\\
&=\lbrace H(\Pi^s), \Pi^s\rbrace_\sim +\lbrace H(\Pi^s), \mathcal{W}_{P_m}^0(U)\rbrace_\sim +\lbrace\mathcal{W}_{0}^{-1}(U), \Pi^s\rbrace_\sim +\lbrace\mathcal{W}_{0}^{-1}(U), \mathcal{W}_{P_m}^0(U)\rbrace_\sim\\
&\in \lbrace H(\Pi^s), \Pi^s\rbrace_\sim +\mathcal{W}_{P_m}^{0}(U)
\end{split}
\end{equation}
where in the first step we use the inductive assumption on $\varphi^s$ and $d_W\varphi^s$ and the identity \eqref{eq:PHi}. In the last step since $H(\Pi^s)\in\mathcal{W}_{0}^{0}(U)$ then by Lemma \ref{lem:bracket} $\lbrace H(\Pi^s), \mathcal{W}_{P_m}^0(U)\rbrace_\sim\in\mathcal{W}_{P_m}^{0}(U)$. Then, since $\Pi^s\in\mathcal{W}_{P_m}^1(U)$, still by Lemma \ref{lem:bracket} $\lbrace\mathcal{W}_{0}^{-1}(U), \Pi^s\rbrace_\sim\in\mathcal{W}_{P_m}^{0}(U)$ and $\lbrace\mathcal{W}_{0}^{-1}(U), \mathcal{W}_{P_m}^0(U)\rbrace_\sim\in\mathcal{W}_{P_m}^{-1}(U)\subset\mathcal{W}_{P_m}^0(U)$. In addition $\lbrace H(\Pi^s), \Pi^s\rbrace_\sim\in\mathcal{W}_{P_m}^0(U) $, indeed 
\begin{multline*}
\lbrace H(\Pi^s), \Pi^s\rbrace_\sim =\Big(\lbrace H(\Pi_E^s), \Pi_E^s\rbrace_{\End E}+\textbf{ad}(H(\Pi^{CLM,s}), \Pi_E^s)- \textbf{ad}(\Pi^{CLM,s},H(\Pi_E^s)),\\ \lbrace H(\Pi^{CLM,s}), \Pi^{CLM,s}\rbrace\Big)
\end{multline*}
Notice that since $[A,A]=0$ then $\lbrace H(\Pi_E^s), \Pi_E^s\rbrace_{\End E}=0$ and because of the grading \[\lbrace H(\Pi^{CLM,s}), \Pi^{CLM,s}\rbrace=0.\] 
Then by the proof of Lemma \ref{lem:bracket} identities $(2)$ and $(3)$ we get \[\ad(H(\Pi^{CLM,s}), \Pi_E^s),  \ad(\Pi^{CLM,s},H(\Pi_E^s))\in\mathcal{W}_{P_m}^{0}(U)\] therefore 
\begin{equation}
\lbrace H(\Pi^s), \Pi^s\rbrace_\sim\in\mathcal{W}_{P_m}^{0}(U).
\end{equation}   

Now at $k>1$ we have to prove that:
\begin{equation}
\ad_{\varphi^s}\cdots \textsf{ad}_{\varphi^s}d_W\varphi^s\in \mathcal{W}_{P_m}^{0}(U)
\end{equation}
By the fact that $ H(\Pi^s)\in\mathcal{W}_{0}^{0}(U)$, applying Lemma \ref{lem:bracket} $k$ times we finally get:
\begin{equation}
\ad_{\varphi^s}\cdots \ad_{\varphi^s}d_W\varphi^s\in\lbrace H(\Pi^s),\cdots, \lbrace H(\Pi^s), \lbrace H(\Pi^s), \Pi^s\rbrace_\sim\rbrace_\sim\cdots\rbrace_\sim+\mathcal{W}_{P_m}^{0}(U)\in\mathcal{W}_{P_m}^{0}(U).
\end{equation}
\end{proof}

\section{Scattering diagrams from solutions of Maurer-Cartan}
\label{sec:two_walls}
In this section we are going to construct consistent scattering diagrams from solutions of the Maurer-Cartan equation. In particular we will first show how to construct a solution $\Phi$ of the Maurer-Cartan equation from the data of an initial scattering diagram $\mathfrak{D}$ with two non parallel walls. Then we will define its completion $\mathfrak{D}_\infty $ by the solution $\Phi$ and we will prove it is consistent. 
\subsection{From scattering diagram to solution of Maurer-Cartan}
Let the initial scattering diagram $\mathfrak{D}=\lbrace\mathsf{w}_1, \mathsf{w}_2\rbrace$ be such that $\mathsf{w}_1=(m_1,P_1,\theta_1)$ and $\mathsf{w}_2=(m_2,P_2,\theta_2)$ are two non-parallel walls and \[\log(\theta_i)=\sum_{j_i,k_i}\Big(A_{j_i,k_i}\mathfrak{w}^{k_im_i}t^{j_i},  a_{j_i,k_i}\mathfrak{w}^{k_im_i}t^{j_i}\partial_{n_i}\Big)\]
for $i=1,2$. 
As we have already done in Section \ref{sec:single wall}, we can define $\bar{\Pi}_1$ and $\bar{\Pi}_2$ to be solutions of Maurer-Cartan equation, respectively supported on $\mathsf{w}_1$ and $\mathsf{w}_2$. 

Although $\bar{\Pi}\defeq\bar{\Pi}_1+\bar{\Pi}_2$ is not a solution of Maurer-Cartan, by Kuranishi's method we can construct $\Xi=\sum_{j\geq 2}\Xi^{(j)}t^j$ such that the one form $\Phi\in\Omega^1(U,\End E\oplus TM)[\![ t ]\!]$ is $\Phi=\bar{\Pi}+\Xi$ and it is a solution of Maurer-Cartan up to higher order in $\hbar$. Indeed let us we write $\Phi$ as a formal power series in the parameter $t$, $\Phi=\sum_{j\geq 1}\Phi^{(j)}t^j$, then it is a solution of Maurer-Cartan if and only if:
\begin{equation*}
\begin{split}
&d_W\Phi^{(1)}=0\\
&d_W\Phi^{(2)}+\frac{1}{2}\lbrace\Phi^{(1)},\Phi^{(1)}\rbrace_\sim=0\\
&\vdots\\
&d_W\Phi^{(k)}+\frac{1}{2}\left(\sum_{s=1}^{k-1}\lbrace\Phi^{(s)},\Phi^{(k-s)}\rbrace_\sim\right)=0
\end{split}
\end{equation*}
 Moreover, recall from \eqref{rmk:closed} that $\bar{\Pi}_i\defeq\big(\Pi_{E,i}+\Pi_{E,R,i},\Pi^{CLM}_i\big)$, $i=1,2$ are solutions of the Maurer-Cartan equation and they are $d_W$-closed. Therefore at any order in the formal parameter $t$, the solution $\Phi=\bar{\Pi}+\Xi$ is computed as follows:
\begin{equation}\label{eq:recPhi}
\begin{split}
\Phi^{(1)}&=\bar{\Pi}^{(1)}\\
\Phi^{(2)}&=\bar{\Pi}^{(2)}+\Xi^{(2)}, \text{where } d_W\Xi^{(2)}=-\frac{1}{2}(\lbrace\Phi^{(1)},\Phi^{(1)}\rbrace_\sim) \\
\Phi^{(3)}&=\bar{\Pi}^{(3)}+\Xi^{(3)}, \text{where }  d_W\Xi^{(3)}=-\frac{1}{2}\left(\lbrace\Phi^{(1)},\bar{\Pi}^{(2)}+\Xi^{(2)}\rbrace_\sim+\lbrace\bar{\Pi}^{(2)}+\Xi^{(2)},\Phi^{(1)}\rbrace_\sim\right)\\
&\vdots\\
\Phi^{(k)}&=\bar{\Pi}^{(k)}+\Xi^{(k)}, \text{where } d_W\Xi^{(k)}=-\frac{1}{2}\left(\lbrace\Phi,\Phi\rbrace_\sim\right)^{(k)}. 
\end{split}
\end{equation}  
In order to explicitly compute $\Xi$ we want to ``invert'' the differential $d_W$ and this can be done by choosing a homotopy operator. Let us recall that a homotopy operator is a homotopy $H$ of morphisms $p$ and $\iota$, namely $H\colon \Omega^{\bullet}(U)\to \Omega^{\bullet}[-1](U)$, $p\colon\Omega^{\bullet}(U)\to H^{\bullet}(U)$ and $\iota\colon H^{\bullet}(U)\to\Omega^{\bullet}(U)$ such that $\text{id}_{\Omega^\bullet}-\iota\circ p=d_WH+Hd_W$. 
Let us now explicitly define the homotopy operator $\mathbf{H}$. Let $U$ be an open affine neighbourhood of $m_0=P_1\cap P_2$, and fix $q_0\in \left(H_{-,m_1}\cap H_{-,m_2}\right)\cap U$. Then choose a set of coordinates centred in $q_0$ and denote by $(u_{m},u_{m^\perp})$ a choice of such coordinates such that with respect to a ray $P_m=m_0+\R_{\geq 0}m$, $u_{m^\perp}$ is the coordinate orthogonal to $P_m$ and $u_{m}$ is tangential to $P_m$. 
Moreover recall the definition of morphisms $p$ and $\iota$, namely 
$p\defeq\bigoplus_m p_m$ and $p_m$ maps functions $\alpha\mathfrak{w}^m\in\Omega^0(U)\mathfrak{w}^m$
to $\alpha(q_0)\mathfrak{w}^m$, while $\iota\defeq\bigoplus_m\iota_m$ and $\iota_m$ is the embedding of constant function at degree zero, and it is zero otherwise. 

\begin{definition}\label{def:homotopyH}
The homotopy operator $\mathbf{H}=\bigoplus_m\mathbf{H}_m\colon\bigoplus_m\Omega^{\bullet}(U)\mathfrak{w}^m\to\bigoplus_m\Omega^{\bullet}(U)[-1]\mathfrak{w}^m$ is defined as follows. 
For any $0$-form $\alpha\in\Omega^0(U)$, $\mathbf{H}(\alpha\mathfrak{w}^m)=0$, since there are no degree $-1$-forms. 
For any $1$-form $\alpha\in\Omega^1(U)$, in local coordinates we have $\alpha=f_0(u_m,u_{m^\perp})du_m+f_1(u_m,u_{m^\perp})du_{m^\perp}$ and 
\[
\mathbf{H}(\alpha\mathfrak{w}^m)\defeq\mathfrak{w}^m\left(\int_0^{u_m}f_0(s,u_{m^\perp})ds+\int_0^{u_{m^\perp}}f_1(0,r)dr\right)
\] 

Finally since any $2$-forms $\alpha\in\Omega^2(U)$ in local coordinates can be written $\alpha=f(u_m,u_{m^\perp})du_m\wedge du_{m^\perp}$, then 
\[
\mathbf{H}(\alpha\mathfrak{w}^m)\defeq\mathfrak{w}^m\left(\int_0^{u_m}f(s,u_{m^\perp})ds\right)du_{m^\perp}.
\]
\end{definition}

The homotopy $\mathbf{H}$ seems defined \textit{ad hoc} for each degree of forms, however it can be written in an intrinsic way for every degree, as in Definition 5.12 \cite{MCscattering}. We have defined $\mathbf{H}$ in this way because it is clearer how to compute it in practice. 

\begin{lemma}
The following identity 
\begin{equation}\label{eq:idH2}
\text{id}_{\Omega^\bullet}-\mathbf{\iota}_m\circ p_m=\mathbf{H}_md_W+d_W\mathbf{H}_m
\end{equation}
holds true for all $m\in\Lambda$. 
\end{lemma}
\begin{proof} We are going to prove the identity separately for $0, 1$ and $2$ forms.

Let $\alpha=\alpha_0$ be of degree zero, then by definition $\mathbf{H}_m(\alpha\mathfrak{w}^m)=0$ and $\mathbf{\iota}_m\circ p_m(\alpha\mathfrak{w}^m)=\alpha_0(q_0)$. Then $d_W(\alpha\mathfrak{w}^m)=\mathfrak{w}^m\big(\frac{\partial\alpha_0}{\partial u_m}du_m+\frac{\partial\alpha_0}{\partial u_{m^\perp}}du_{m^\perp}\big)$. Hence
\[
\begin{split}
\mathbf{H}_md_W(\alpha_0\mathfrak{w}^m)&=\mathfrak{w}^m\int_{0}^{u_m}\frac{\partial\alpha_0(s,u_{m^\perp})}{\partial s}ds+\mathfrak{w}^m\int_0^{u_{m^\perp}}\frac{\partial\alpha_0}{\partial u_{m^\perp}}(0,r)dr \\
&=\mathfrak{w}^m[\alpha_0(u_m,u_{m^\perp})-\alpha_0(0,u_{m^\perp})+\alpha_0(0,u_{m^\perp})-\alpha_0(0,0)].
\end{split}
\]

Then consider $\alpha\in\Omega^1(U)\mathfrak{w}^m$. By definition $\mathbf{\iota}_m\circ p_m(\alpha\mathfrak{w}^m)=0$ and 
\[
\begin{split}
\mathbf{H}_md_W(\alpha\mathfrak{w}^m)&=\mathbf{H}d_W\left(\mathfrak{w}^m(f_0du_m+f_1du_{m^\perp})\right)\\
&=\mathbf{H}_m\left(\mathfrak{w}^m\left(\frac{\partial f_0}{\partial u_{m^\perp}}du_{m^\perp}\wedge du_m+\frac{\partial f_1}{\partial u_{m}}du_m\wedge du_{m^\perp}\right)\right)\\
&=\mathfrak{w}^m\left(\int_0^{u_m}\left(-\frac{\partial f_0}{\partial u_{m^\perp}}+\frac{\partial f_1}{\partial s}\right)ds\right)du_{m^\perp}\\
&=\left(-\int_0^{u_m}\frac{\partial f_0}{\partial u_{m^\perp}}(s,u_{m^\perp})ds +f_1(u_m,u_{m^\perp})-f_1(0,u_{m^\perp})\right)\mathfrak{w}^mdu_{m^\perp}
\end{split}
\] 
\[
\begin{split}
d_W\mathbf{H}_m(\alpha\mathfrak{w}^m)&=d_W\mathbf{H}(\mathfrak{w}^m\left(f_0du_m+f_1du_{m^\perp}\right))\\
&=d_W\left(\mathfrak{w}^m\left(\int_0^{u_m}f_0(s,u_{m^\perp})ds +\int_0^{u_{m^\perp}}f_1(0,r)dr\right)\right)\\
&=\mathfrak{w}^m d\left(\int_0^{u_m}f_0(s,u_{m^\perp})ds +\int_0^{u_{m^\perp}}f_1(0,r)dr\right)\\
&=\mathfrak{w}^m\left(f_0(u_m,u_{m^\perp})du_m+\frac{\partial}{\partial u_{m^\perp}}\left(\int_0^{u_m}f_0(s,u_{m^\perp})ds\right)du_{m^\perp}+f_1(0,u_{m^\perp})du_{m^\perp}\right).
\end{split}
\] 

We are left to prove the identity when $\alpha$ is of degree two: by degree reasons $d_W(\alpha\mathfrak{w}^m)=0$ and $\iota_m\circ p_m(\alpha\mathfrak{w}^m)=0$. Then 
\[
\begin{split}
d_W\mathbf{H}_m(\alpha\mathfrak{w}^m)&=\mathfrak{w}^md\left(\left(\int_{0}^{u_m}f(s,u_{m^\perp})ds\right)du_{m^\perp}\right)\\
&=\mathfrak{w}^mf(u_m,u_{m^\perp})du_m\wedge du_{m^\perp}. 
\end{split}
\] 
\end{proof} 

\begin{prop}[{see prop 5.1 in \cite{MCscattering}}]\label{prop:Kuranishi}
Assume that $\Phi$ is a solution of 
\begin{equation}\label{eq:Phi}
\Phi=\bar{\Pi}-\frac{1}{2}\mathbf{H}\left(\lbrace\Phi,\Phi\rbrace_\sim\right)   
\end{equation}
Then $\Phi$ is a solution of the Maurer-Cartan equation.
\end{prop}
\begin{proof}
First notice that by definition $p(\lbrace\Phi,\Phi\rbrace_\sim)=0$ and by degree reasons $d_W(\lbrace\Phi,\Phi\rbrace_\sim)=0$ too. Hence by identity \eqref{eq:idH2} we get that $\lbrace\Phi,\Phi,\rbrace_\sim=d_W\mathbf{H}(\lbrace\Phi,\Phi\rbrace_\sim)$, and if $\Phi$ is a solution of equation \eqref{eq:Phi} then $d_W\Phi=d_W\bar{\Pi}-\frac{1}{2}d_W\mathbf{H}(\lbrace\Phi,\Phi\rbrace_\sim)=-\frac{1}{2}d_W\mathbf{H}(\lbrace\Phi,\Phi\rbrace_\sim)$. 
\end{proof}
From now on we will look for solutions $\Phi$ of equation \eqref{eq:Phi} rather than to the Maurer-Cartan equation. The advantage is that we have an integral equation instead of a differential equation, and $\Phi$ can be computed by its expansion in the formal parameter $t$, namely $\Phi=\sum_{j\geq 1}\Phi^{(j)}t^j$.  
\begin{notation}
Let $P_{m_1}=m_0-m_1\mathbb{R}$ and $P_{m_2}=m_0-m_2\mathbb{R}$ and let $(u_{m_1},u_{m_1^\perp})$ and $(u_{m_2},u_{m_2^\perp})$ be respectively two basis of coordinates in $U$, centred in $q_0$ as above. Let $m_a\defeq a_1m_1+a_2m_2$, consider the ray $P_{m_a}\defeq m_0-m_a\mathbb{R}_{\geq 0}$ and choose coordinates $u_{m_{a}}\defeq \left(-a_2u_{m_1^\perp}+a_1u_{m_2^\perp}\right)$ and $u_{m_a^\perp}\defeq\left(a_1u_{m_1^\perp}+a_2u_{m_2^\perp}\right)$.    
\end{notation}
\begin{remark}\label{rmk:H(delta wedge delta)}
If $\alpha=\delta_{m_1}\wedge\delta_{m_2}$, then by the previous choice of coordinates 
\[\delta_{m_1}\wedge\delta_{m_2}=\frac{e^{-\frac{u_{m_1^\perp}^2+u_{m_2^\perp}^2}{\hbar}}}{{\hbar\pi}}du_{m_1^\perp}\wedge du_{m_2^\perp}=\frac{e^{-\frac{u_{m_{a}}^2+u_{m_a^\perp}^2}{(a_1^2+a_2^2)\hbar}}}{{\hbar\pi}}du_{u_{m_a}^\perp}\wedge du_{m_{a}}.\]
In particular we explicitly compute $\mathbf{H}(\delta_{m_1}\wedge\delta_{m_2}\mathfrak{w}^{lm_a})$:
\begin{equation}
\begin{split}
\mathbf{H}(\alpha\mathfrak{w}^{lm_a})&=\mathfrak{w}^{lm_a}\left(\int_{0}^{u_{m_{a}}}\frac{e^{-\frac{s^2+u_{m_a^\perp}^2}{(a_1^2+a_2^2)\hbar}}}{\hbar\pi}ds\right) du_{u_a^\perp}=\mathfrak{w}^{lm_a}\left(\int_{0}^{u_{m_{a}}}\frac{e^{-\frac{s^2}{(a_1^2+a_2^2)\hbar}}}{\sqrt{\hbar\pi}}ds\right) \frac{e^{-\frac{u_{m_a^{\perp}}^2}{(a_1^2+a_2^2)\hbar}}}{\sqrt{\hbar\pi}} du_{m_a^{\perp}}
\end{split}
\end{equation}
Hence $\mathbf{H}\big(\delta_{m_1}\wedge\delta_{m_2}\mathfrak{w}^{lm_a}\big)=f(\hbar,u_{m_{a}})\delta_{m_a}$ where $f(\hbar,u_{m_{a}})=\int_{0}^{u_{m_{a}}}\frac{e^{-\frac{s^2}{(a_1^2+a_2^2)\hbar}}}{\sqrt{\hbar\pi}}ds\in O_{loc}(1)$.
\end{remark}

In order to construct a consistent scattering diagram from the solution $\Phi$ we introduce labeled ribbon trees. Indeed via the combinatorial of such trees we  can rewrite $\Phi$ as a sum over primitive Fourier mode, coming from the contribution of the out-going edge of the trees. 
\subsubsection{Labeled ribbon trees} 
Let us briefly recall the definition of labeled ribbon trees, which was introduced in \cite{MCscattering}. 
\begin{definition}[Definition 5.2 in \cite{MCscattering}]\label{def:trees}
A k-tree $T$ is the datum of a finite set of vertices $V$, together with a decomposition $V=V_{in}\sqcup V_0\sqcup\lbrace v_T\rbrace$, and a finite set of edges $E$, such that, given the two boundary maps $\partial_{in},\partial_{out}:E\rightarrow V$ (which respectively assign to each edge its incoming and outgoing vertices), satisfies the following assumption:
\begin{enumerate}
\item $\#V_{in}=k$ and for any vertex $v\in V_{in}$, $\#\partial_{in}^{-1}(v)=0$ and $\#\partial_{out}^{-1}(v)=1$; 
\item for any vertex $v\in V_0$, $\#\partial_{in}^{-1}(v)=1$ and $\#\partial_{out}^{-1}(v)=2$;
\item $v_T$ is such that $\#\partial_{in}^{-1}(v_T)=0$ and $\#\partial_{out}^{-1}(v)=1$.
\end{enumerate}
We also define $e_T=\partial_{in}^{-1}(v_T)$. 
\end{definition}
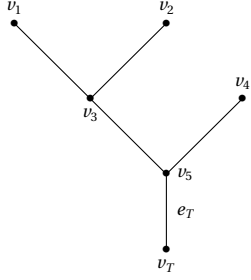
\begin{figure}[h]
\center
\begin{tikzpicture}
\draw (0,4) -- (1,3);
\draw (2,4) -- (1,3);
\draw (1,3) -- (2,2);
\draw (3,3) -- (2,2);
\draw (2,2) -- (2,1);
\node[below, font=\tiny] at (2,1) {$v_T$};
\node[above, font=\tiny] at (0,4) {$v_1$};
\node[above, font=\tiny] at (2,4) {$v_2$};
\node[below, font=\tiny] at (1,3) {$v_3$};
\node[above, font=\tiny] at (3,3) {$v_4$};
\node[right, font=\tiny] at (2,2) {$v_5$};
\node[font=\tiny] at (2,1) {$\bullet$};
\node[ font=\tiny] at (0,4) {$\bullet$};
\node[font=\tiny] at (2,4) {$\bullet$};
\node[font=\tiny] at (1,3) {$\bullet$};
\node[font=\tiny] at (3,3) {$\bullet$};
\node[font=\tiny] at (2,2) {$\bullet$};
\node[right, font=\tiny] at (2,1.5) {$e_T$};
\end{tikzpicture}
\caption{This is an example of a 3-tree, where the set of vertices is decomposed by $V_{in}=\lbrace v_1, v_2, v_4\rbrace$, $V_0=\lbrace v_3, v_5\rbrace$.}
\label{t:3tree}
\end{figure}
Two k-trees $T$ and $T'$ are \textit{isomorphic} if there are bijections $V\cong V'$ and $E\cong E'$ preserving the decomposition $V_0\cong V'_0$, $V_{in}\cong V'_{in}$ and $\lbrace v_T\rbrace\cong\lbrace v_{T'}\rbrace$ and the boundary maps $\partial_{in}, \partial_{out}$. 
  
It will be useful in the following to introduce the definition of topological realization $\mathcal{T}(T)$ of a k-tree $T$, namely $\mathcal{T}(T)\defeq\big(\coprod_{e\in E}[0,1]\big)/\thicksim$, where $\thicksim$ is the equivalence relation that identifies boundary points of edges with the same image in $V\setminus\lbrace v_T\rbrace$. 

Since we need to keep track of all the possible combinations while we compute commutators (for instance for $\Phi^{(3)}$  there is the contribution of $\lbrace\Phi^{(1)},\Phi^{(2)}\rbrace_\sim$ and $\lbrace\Phi^{(2)},\Phi^{(1)}\rbrace_\sim$), we introduce the notion of ribbon trees:
\begin{definition}[Definition 5.3 in \cite{MCscattering}]\label{def:ribbon trees}
A ribbon structure on a k-tree is a cyclic ordering of the vertices. It can be viewed as an embedding $\mathcal{T}(T)\hookrightarrow\mathbb{D}$, where $\mathbb{D}$ is the disk in $\mathbb{R}^2$, and the cyclic ordering is given according to the anticlockwise orientation of $\mathbb{D}$. 
\end{definition}
Two ribbon k-trees $T$ and $T'$ are \textit{isomorphic} if they are isomorphic as k-trees and the isomorphism preserves the cyclic order. The set of all isomorphism classes of ribbon k-trees will be denoted by $\mathbb{R}\mathbb{T}_k$. 
As an example, the following two 2-trees are not isomorphic:
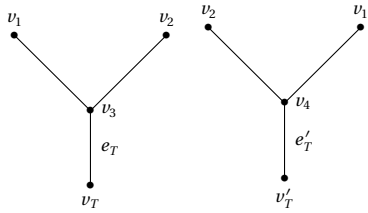
\begin{figure}[h]
\begin{tikzpicture}
\draw (0,2) -- (1,1);
\draw (2,2) -- (1,1);
\draw (1,1) -- (1,0);
\node[below, font=\tiny] at (1,0) {$v_T$};
\node[above, font=\tiny] at (0,2) {$v_1$};
\node[above, font=\tiny] at (2,2) {$v_2$};
\node[right, font=\tiny] at (1,1) {$v_3$};
\node[font=\tiny] at (1,0) {$\bullet$};
\node[ font=\tiny] at (0,2) {$\bullet$};
\node[font=\tiny] at (2,2) {$\bullet$};
\node[font=\tiny] at (1,1) {$\bullet$};
\node[right, font=\tiny] at (1,0.5) {$e_T$};
\end{tikzpicture}
\begin{tikzpicture}
\draw (0,2) -- (1,1);
\draw (2,2) -- (1,1);
\draw (1,1) -- (1,0);
\node[below, font=\tiny] at (1,0) {$v_T'$};
\node[above, font=\tiny] at (0,2) {$v_2$};
\node[above, font=\tiny] at (2,2) {$v_1$};
\node[right, font=\tiny] at (1,1) {$v_4$};
\node[font=\tiny] at (1,0) {$\bullet$};
\node[ font=\tiny] at (0,2) {$\bullet$};
\node[font=\tiny] at (2,2) {$\bullet$};
\node[font=\tiny] at (1,1) {$\bullet$};
\node[right, font=\tiny] at (1,0.5) {$e_T'$};
\end{tikzpicture}
\end{figure}

In order to keep track of the $\hbar$ behaviour while we compute the contribution from the commutators, let us decompose the bracket on the dgLa as follows:
\begin{definition}\label{label}
Let $(A,\alpha)=(A_Jdx^J\mathfrak{w}^{m_1},\alpha_Jdx^J\mathfrak{w}^{m_1}\partial_{n_1})\in\Omega^p(U,\End E\oplus TM)\mathfrak{w}^{m_1}$ and $(B,\beta)=(B_Kdx^K\mathfrak{w}^{m_2},\beta_Kdx^K\mathfrak{w}^{m_2}\partial_{n_2})\in\Omega^q(U,\End E\oplus TM)\mathfrak{w}^{m_2}$. Then we decompose $\lbrace\cdot ,\cdot\rbrace_\sim$ as the sum of: 
\begin{enumerate}
\item[$\natural$] $\lbrace(A,\alpha),(B,\beta)\rbrace_{\natural}\defeq\left(\alpha\wedge B\langle n_1,m_2\rangle- \beta\wedge A\langle n_2,m_1\rangle+ \lbrace A,B\rbrace_{\End E},\lbrace\alpha,\beta\rbrace\right)$
\item[$\flat$] $\lbrace(A,\alpha),(B,\beta)\rbrace_{\flat}\defeq\left(i\hbar \beta_Kn_2^q\frac{\partial A_J}{\partial x_q}dx^J\wedge dx^K\mathfrak{w}^{m_1+m_2}, \beta(\nabla_{\partial_{n_2}}\alpha)\mathfrak{w}^{m_1+m_2}\partial_{n_1}\right)$
\item[$\sharp$] $\lbrace(A,\alpha),(B,\beta)\rbrace_{\sharp}\defeq\left(i\hbar \alpha_Jn_1^q\frac{\partial B_K}{\partial x_q}dx^K\wedge dx^J\mathfrak{w}^{m_1+m_2}, \alpha(\nabla_{\partial_{n_1}}\beta)\mathfrak{w}^{m_1+m_2}\partial_{n_2}\right)$
\item[$\star$] $\lbrace(A,\alpha),(N,\beta)\rbrace_{\star}\defeq i\hbar\left(\alpha_Jn_1^qB_KA_q(\phi)dx^J\wedge dx^K-n_2^q\beta_KA_q(\phi)A_Jdx^K\wedge  dx^J,0\right)$.
\end{enumerate}
\end{definition}

The previous definition is motivated by the following observation: the label $\natural$ contains terms of the Lie bracket $\lbrace\cdot,\cdot\rbrace_\sim$ which leave unchanged the behaviour in $\hbar$. Then both the labels $\flat$ and $\sharp$ contain terms which contribute with an extra $\hbar$ factor and at the same time contain derivatives. The last label $\star$ contains terms which contribute with an extra $\hbar$ but do not contain derivatives. 
 
\begin{definition}\label{def:labeled ribbon tree}
A labeled ribbon k-tree is a ribbon k-tree $T$ together with:
\begin{enumerate}
\item[(i)] a label $\natural$, $\sharp$, $\flat$, $\star$ -as defined in Definition \ref{label}- for each vertex in $V_0$;
\item[(ii)] a label $(m_e,j_e)$ for each incoming edge $e$, where $m_e$ is the Fourier mode of the incoming vertex and $j_e\in\Z_{>0}$ gives the order in the formal parameter $t$.
\end{enumerate} 
\end{definition}
There is an induced labeling of all the edges of the trees defined as follows: at any trivalent vertex with incoming edges $e_1,e_2$ and outgoing edge $e_3$ we define $(m_{e_3},j_{e_3})=(m_{e_1}+m_{e_2}, j_{e_1}+j_{e_2})$. We will denote by $(m_T,j_T)$ the label corresponding to the unique incoming edge of $\nu_T$. 
Two labeled ribbon k-trees $T$ and $T'$ are \textit{isomorphic} if they are isomorphic as ribbon k-trees and the isomorphism preserves the labeling. The set of equivalence classes of labeled ribbon k-trees will be denoted by $\mathbb{L}\mathbb{R}\mathbb{T}_k$. We also introduce the following notation for equivalence classes of labeled ribbon trees:
\begin{notation}
We denote by $\mathbb{LRT}_{k,0}$ the set of equivalence classes of $k$ labeled ribbon trees such that they have only the label $\natural$. We denote by $\mathbb{LRT}_{k,1}$ the complement set, namely $\mathbb{LRT}_{k,1}=\mathbb{LRT}_k-\mathbb{LRT}_{k,0}$.
\end{notation}
Let us now define the operator $\mathfrak{t}_{k,T}$ which allows to write the solution $\Phi$ in terms of labeled ribbon trees.
\begin{definition}
Let $T$ be a labeled ribbon k-tree, then the operator 
\begin{equation}
\mathfrak{t}_{k,T}:\Omega^1(U,\End E\oplus TM)^{\otimes k}\rightarrow\Omega^{1}(U, \End E\oplus TM)
\end{equation}
is defined as follows: it aligns the input with the incoming vertices according with the cyclic ordering and it labels the incoming edges (as in part (ii) of Definition \ref{def:labeled ribbon tree}). Then it assigns at each vertex in $V_0$ the commutator according with the part (i) of Definition \ref{label}. Finally it assigns the homotopy operator $-\mathbf{H}$ to each outgoing edge. 
\end{definition}
In particular the solution $\Phi$ of equation \eqref{eq:Phi} can be written as a sum on labeled ribbon k-trees as follows:
\begin{equation}\label{def:Phitree}
\Phi=\sum_{k\geq 1}\sum_{T\in\mathbb{LRT}_k}\frac{1}{2^{k-1}}\mathfrak{t}_{k,T}(\bar{\Pi},\cdots ,\bar{\Pi}).
\end{equation} 
Recall that by definition 
\[\lbrace\left(A,\alpha\partial_{n_1}\right)\mathfrak{w}^{k_1m_1},\left(B,\beta\partial_{n_2}\right)\mathfrak{w}^{k_2m_2}\rbrace_\sim=\left(C,\gamma\partial_{\langle k_2m_2,n_1\rangle n_2-\langle k_1m_1,n_2\rangle n_1}\right)\mathfrak{w}^{k_1m_1+k_2m_2}\]
for some $(A,\alpha)\in\Omega^s(U,\End E\oplus TM), (B,\beta)\Omega^r(U,\End E\oplus TM)$ and $(C,\gamma)\in\Omega^{r+s}(U,\End E\oplus TM)$, hence the Fourier mode of any labeled brackets has the same frequency $m_e=k_1m_1+k_2m_2$ independently of the label $\natural, \flat, \sharp,\star$. In particular each $m_e$ can be written as $m_e=l (a_1m_1+a_2m_2)$ for some primitive elements $(a_1,a_2)\in\big(\Z^2_{\geq 0}\big)_{\text{prim}}$. Let us introduce the following notation: 
\begin{notation}
Let $a=(a_1,a_2)\in\left(\Z_{\geq0}^2\right)_{\text{prim}}$ and define $m_a\defeq a_1m_1+a_2m_2$. Then we define $\Phi_a$ to be the sum over all trees of the contribution to $\mathfrak{t}_{k,T}(\bar{\Pi},\cdots, \bar{\Pi})$ with Fourier mode $\mathfrak{w}^{lm_a}$ for every $l\geq1$. In particular we define $\Phi_{(1,0)}\defeq\bar{\Pi}_1$ and $\Phi_{(0,1)}\defeq\bar{\Pi}_2$. 
\end{notation}
It follows that the solution $\Phi$ can be written as a sum on primitive Fourier mode as follows:
\begin{equation}\label{eq:Phi_primitive}
    \Phi=\sum_{a\in\big(\Z^2_{\geq 0}\big)_{\text{prim}}}\Phi_{a}.
\end{equation}

As an example, let us consider $\Phi^{(2)}$. From equation \eqref{eq:Phi} we get   
\[
\Phi^{(2)}=\bar{\Pi}^{(2)}-\frac{1}{2}\mathbf{H}\left(\lbrace\bar{\Pi}^{(1)},\bar{\Pi}^{(1)}\rbrace_\sim\right)\]
and the possible 2-trees $T\in\mathbb{LRT}_2$, up to choice of the initial Fourier modes, are represented in figure \ref{fig:2trees}.    
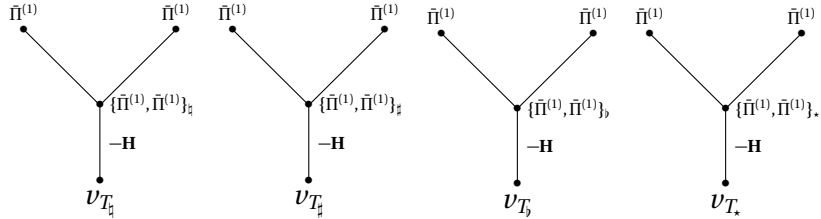
\begin{figure}[h]
\begin{tikzpicture}
\draw (0,2) -- (1,1);
\draw (2,2) -- (1,1);
\draw (1,1) -- (1,0);
\node[below] at (1,0) {$v_{T_\natural}$};
\node[above, font=\tiny] at (0,2) {$\bar{\Pi}^{(1)}$};
\node[above, font=\tiny] at (2,2) {$\bar{\Pi}^{(1)}$};
\node[right, font=\tiny] at (1,1) {$\lbrace \bar{\Pi}^{(1)},\bar{\Pi}^{(1)}\rbrace_\natural $};
\node[font=\tiny] at (1,0) {$\bullet$};
\node[ font=\tiny] at (0,2) {$\bullet$};
\node[font=\tiny] at (2,2) {$\bullet$};
\node[font=\tiny] at (1,1) {$\bullet$};
\node[right, font=\tiny] at (1,0.5) {$-\mathbf{H}$};
\end{tikzpicture}
\begin{tikzpicture}
\draw (0,2) -- (1,1);
\draw (2,2) -- (1,1);
\draw (1,1) -- (1,0);
\node[below] at (1,0) {$v_{T_\sharp}$};
\node[above, font=\tiny] at (0,2) {$\bar{\Pi}^{(1)}$};
\node[above, font=\tiny] at (2,2) {$\bar{\Pi}^{(1)}$};
\node[right, font=\tiny] at (1,1) {$\lbrace \bar{\Pi}^{(1)},\bar{\Pi}^{(1)}\rbrace_\sharp $};
\node[font=\tiny] at (1,0) {$\bullet$};
\node[ font=\tiny] at (0,2) {$\bullet$};
\node[font=\tiny] at (2,2) {$\bullet$};
\node[font=\tiny] at (1,1) {$\bullet$};
\node[right, font=\tiny] at (1,0.5) {$-\mathbf{H}$};
\end{tikzpicture}
\begin{tikzpicture}
\draw (0,2) -- (1,1);
\draw (2,2) -- (1,1);
\draw (1,1) -- (1,0);
\node[below] at (1,0) {$v_{T_\flat}$};
\node[above, font=\tiny] at (0,2) {$\bar{\Pi}^{(1)}$};
\node[above, font=\tiny] at (2,2) {$\bar{\Pi}^{(1)}$};
\node[right, font=\tiny] at (1,1) {$\lbrace \bar{\Pi}^{(1)},\bar{\Pi}^{(1)}\rbrace_\flat $};
\node[font=\tiny] at (1,0) {$\bullet$};
\node[ font=\tiny] at (0,2) {$\bullet$};
\node[font=\tiny] at (2,2) {$\bullet$};
\node[font=\tiny] at (1,1) {$\bullet$};
\node[right, font=\tiny] at (1,0.5) {$-\mathbf{H}$};
\end{tikzpicture}
\begin{tikzpicture}
\draw (0,2) -- (1,1);
\draw (2,2) -- (1,1);
\draw (1,1) -- (1,0);
\node[below] at (1,0) {$v_{T_\star}$};
\node[above, font=\tiny] at (0,2) {$\bar{\Pi}^{(1)}$};
\node[above, font=\tiny] at (2,2) {$\bar{\Pi}^{(1)}$};
\node[right, font=\tiny] at (1,1) {$\lbrace \bar{\Pi}^{(1)},\bar{\Pi}^{(1)}\rbrace_\star $};
\node[font=\tiny] at (1,0) {$\bullet$};
\node[ font=\tiny] at (0,2) {$\bullet$};
\node[font=\tiny] at (2,2) {$\bullet$};
\node[font=\tiny] at (1,1) {$\bullet$};
\node[right, font=\tiny] at (1,0.5) {$-\mathbf{H}$};
\end{tikzpicture}
\caption{$2$-trees labeled ribbon trees, which contribute to the solution $\Phi$.}\label{fig:2trees}
\end{figure}
Hence 
\begin{equation*}
\begin{split}
\Phi^{(2)}&=\bar{\Pi}_1^{(2)}+\bar{\Pi}_2^{(2)}-\frac{1}{2}\mathbf{H}\left(\lbrace \bar{\Pi}^{(1)},\bar{\Pi}^{(1)}\rbrace_\natural+\lbrace \bar{\Pi}^{(1)},\bar{\Pi}^{(1)}\rbrace_\sharp+\lbrace \bar{\Pi}^{(1)},\bar{\Pi}^{(1)}\rbrace_\flat+\lbrace \bar{\Pi}^{(1)},\bar{\Pi}^{(1)}\rbrace_\star \right)\\
&=\Phi_{(1,0)}^{(2)}+\Phi_{(0,1)}^{(2)}+\\
&-\Big(\left(a_{1,k_1}A_{1,k_2}\langle n_1,k_2m_2\rangle- a_{1,k_2}A_{1,k_1}\langle n_2,k_1m_1\rangle+[A_{1,k_1},A_{1,k_2}]\right), a_{1,k_1} a_{1,k_2}\partial_{\langle k_2m_2,n_1\rangle n_2-\langle k_1m_1,n_2\rangle n_1}\Big)\cdot\\
&\cdot\mathbf{H}(\delta_{m_1}\wedge \delta_{m_2}\mathfrak{w}^{k_1m_1+k_2m_2})\\
&+\Big(a_{1,k_2}A_{1,k_1}, -a_{1,k_2}a_{1,k_1}\partial_{n_1}\Big)\mathbf{H}\left(i\hbar n_2^q\frac{\partial \delta_{m_1}}{\partial x_q}\wedge \delta_{m_2}\mathfrak{w}^{k_1m_1+k_2m_2}\right)\\
&+\Big(a_{1,k_1}A_{1,k_2}, a_{1,k_1}a_{1,k_2}\partial_{n_2}\Big)\mathbf{H}\left(i\hbar n_1^q\frac{\partial \delta_{m_2}}{\partial x_q}\wedge \delta_{m_1}\mathfrak{w}^{k_1m_1+k_2m_2}\right)\\
&+i\hbar\Big(a_{1,k_1}A_{1,k_2}n_1^q\mathbf{H}\left(A_q(\phi)\mathfrak{w}^{k_1m_1+m_2}\delta_{m_1}\wedge \delta_{m_2}\right)-a_{1,k_2} A_{1}n_2^q\mathbf{H}\left(A_q(\phi)\mathfrak{w}^{k_1m_1+k_2m_2}\delta_{m_1}\wedge \delta_{m_2}\right) ,0\Big)
\end{split}
\end{equation*}
By Remark \ref{rmk:H(delta wedge delta)}
\[\mathbf{H}(\delta_{m_1}\wedge \delta_{m_2}\mathfrak{w}^{k_1m_1+k_2m_2})=f(\hbar,u_{m_{a}})\mathfrak{w}^{lm_a}\delta_{m_a}\] and $f(\hbar,u_{m_{a}})\in O_{loc}(1)$, for $k_1m_1+k_2m_2=lm_a$. Analogously \[\mathbf{H}\left(A_q(\phi)\mathfrak{w}^{k_1m_1+k_2m_2}\delta_{m_1}\wedge \delta_{m_2}\right)=f(\hbar,A(\phi),u_{m_{a}})\mathfrak{w}^{lm_a}\delta_{m_a}\] and $f(\hbar,A(\phi),u_{m_{a}})\in O_{loc}(1)$. Then \[\mathbf{H}\left(\hbar\frac{\partial \delta_{m_1}}{\partial x_q}\wedge \delta_{m_2}\mathfrak{w}^{k_1m_1+k_2m_2}\right)=\mathfrak{w}^{lm_a} f(\hbar,u_{m_{a}})\delta_{m_a}\] and $f(\hbar,u_{m_{a}})\in O_{loc}(\hbar^{1/2})$. This shows that every term in the sum above, is a function of some order in $\hbar$ times a delta supported along a ray of slope $m_{(a_1,a_2)}=a_1m_1+a_2m_2$. For any given $a\in\left(\Z^2_{\geq 0}\right)_{\text{prim}}$, these contributions are by definition $\Phi_{(a_1,a_2)}^{(2)}$, hence 
\[\Phi^{(2)}=\Phi_{(1,0)}^{(2)}+\sum_{a\in\left(\Z^2_{\geq0}\right)_{\text{prim}}}\Phi_{(a_1,a_2)}^{(2)}+\Phi_{(0,1)}^{(2)}.\]
In general, the expression of $\Phi_a$ will be much more complicated, but as a consequence of the definition of $\mathbf{H}$ it always contains a delta supported on a ray of slope $m_a$.

  
\subsection{From solution of Maurer-Cartan to the saturated scattering diagram $\mathfrak{D}_{\infty}$}

Let us first introduce the following notation:
\begin{notation}
Let $A\defeq U\setminus\lbrace m_0\rbrace$ be an annulus and let $\tilde{A}$ be the universal cover of $A$ with projection $\varpi\colon\tilde{A}\to A$. Then let us denote by $\tilde{\Phi}$ the pullback of $\Phi$ by $\varpi$, in particular by the decomposition in its primitive Fourier mode $\tilde{\Phi}=\sum_{a\in\left(\Z^2_{\geq 0}\right)}\varpi^*(\Phi_a)\defeq\sum_{a\in\left(\Z^2_{\geq 0}\right)}\tilde{\Phi}_a$. 
\end{notation}
\begin{notation}
We introduce polar coordinates in $m_0$, centred in $m_0=P_{m_1}\cup P_{m_2}$, denoted by $(r, \vartheta)$ and we fix a reference angle $\vartheta_0$ such that the ray with slope $\vartheta_0$ trough $m_0$ contains the base point $q_0$ (see Figure \ref{fig:rif_theta}). 

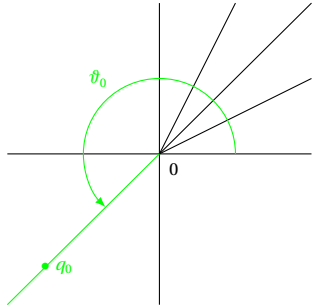
\begin{figure}[h]
\begin{tikzpicture}
\draw (2,0) -- (2,4);
\draw (0,2) -- (4,2);
\draw (2,2) -- (4,4) ;
\draw (2,2) -- (3,4) ;
\draw (2,2) -- (4,3) ;
\draw [green] (2,2) -- (0,0);
\node [below right, font=\tiny] at (2,2) {0};
\node[green, right, font=\tiny] at (0.5,0.5) {$q_0$};
\node[ green, font=\tiny] at (0.5,0.5) {$\bullet$};
\draw [-latex, green] (3,2) arc (0:225:1);
\node[green,font=\tiny] at (1.2,3) {$\vartheta_0$}; 
\end{tikzpicture}
\caption{The reference angle $\vartheta_0$.}
\label{fig:rif_theta}
\end{figure}

Then for every $a\in\left(\Z_{\geq0}^2\right)_{\text{prim}}$ we associate to the ray $P_{m_a}\defeq m_0+\R_{\geq 0}m_a$ an angle $\vartheta_a\in\left(\vartheta_0,\vartheta_0+2\pi\right)$. We identify $P_{m_a}\cap A$ with its lifting in $\tilde{A}$ and by abuse of notation we will denote it by $P_{m_a}$. We finally define $\tilde{A}_0\defeq\lbrace (r, \vartheta)\vert \vartheta_0-\epsilon_0<\vartheta <\vartheta_0+2\pi\rbrace$, for some small positive $\epsilon_0$.  
\end{notation}
   
\begin{lemma}[see Lemma 5.40 \cite{MCscattering}]\label{lem:Phi_aMC}
Let $\Phi$ be a solution of equation \eqref{eq:Phi} which has been decomposed as a sum over primitive Fourier mode, as in \eqref{eq:Phi_primitive}. Then for any $a\in\big(\Z_{\geq 0}^2\big)_{\text{prim}}$, $\tilde{\Phi}_{a}$ is a solution of the Maurer-Cartan equation in $\tilde{A}_0$, up to higher order in $\hbar$, namely $\lbrace\tilde{\Phi}_{a},\tilde{\Phi}_{a'}\rbrace_\sim\in\mathcal{W}_2^{-\infty}(\tilde{A}_0,\End E\oplus TM)\mathfrak{w}^{km_a+k'm_{a'}}$ and $d_W\tilde{\Phi}_{a}\in\mathcal{W}_{2}^{-\infty}(\tilde{A}_0,\End E\oplus TM)\mathfrak{w}^{km_a}$ for some $k,k'\geq 1$.
\end{lemma}

\begin{proof}
Recall that $\Phi$ is a solution of Maurer-Cartan $d_W\Phi+\lbrace\Phi,\Phi\rbrace_\sim=0$, hence its pullback $\varpi^*(\Phi)$ is such that $\sum_{a\in\big(\Z_{\geq0}^2\big)_{\text{prim}}}d_W\tilde{\Phi}_{a}+\sum_{a,a'\in\big(\Z_{\geq0}^2\big)_{\text{prim}}}\lbrace\tilde{\Phi}_{a},\tilde{\Phi}_{a'}\rbrace_\sim=0$. 
Looking at the bracket there are two possibilities: first of all, if $a\neq a'$ then $\lbrace\tilde{\Phi}_{a},\tilde{\Phi}_{a'}\rbrace_\sim$  is proportional to $\delta_{m_{a}}\wedge\delta_{m_{a'}}\mathfrak{w}^{km+k'm'}$. Since $P_{m_a}\cap P_{m{_a'}}\cap \tilde{A}=\varnothing$ then $\delta_{m_{a}}\wedge\delta_{m_{a'}}\in\mathcal{W}_2^{-\infty}(\tilde{A}_0)$, indeed writing $\delta_{m_{a}}\wedge\delta_{m_{a'}}$ in polar coordinates it is a $2$-form with coefficient a 
Gaussian function in two variables centred in $m_0\not\in\tilde{A}_0$. Hence it is bounded by $e^{-\frac{c_V}{\hbar}}$ in the open subset $V\subset\tilde{A}_0$.     
Secondly, if $a=a'$ then by definition $\lbrace\tilde{\Phi}_{a},\tilde{\Phi}_{a}\rbrace_\sim=0$. Finally, by the fact that $d_W\tilde{\Phi}_a=-\sum_{a',a''}\lbrace\tilde{\Phi}_{a'},\tilde{\Phi}_{a''}\rbrace_\sim$ it follows that $d_W\tilde{\Phi}_{a}\in\mathcal{W}_{2}^{-\infty}(\tilde{A}_0)\mathfrak{w}^{km_a}$ for some $k\geq 1$. 
\end{proof}


Now recall that the homotopy operator we have defined in Section \ref{sec:single wall} gives a gauge fixing condition, hence for every $a\in\left(\Z^2_{\geq 0}\right)_{\text{prim}}$ there exists a unique gauge $\varphi_a$ such that $e^{\varphi_a}\ast 0=\tilde{\Phi}_a$ and $p(\varphi_a)=0$. To be more precise we should consider $\tilde{p}\defeq\varpi^*( p)$ as gauge fixing condition and similarly $\tilde{\iota}\defeq\varpi^*(\iota)$ and $\tilde{H}\defeq\varpi^*(H)$ as homotopy operator, however if we consider affine coordinates on $\tilde{A}$, these operators are equal to $p$, $\iota$ and $H$ respectively. In addition in affine coordinates on $\tilde{A}$ the solution $\tilde{\Phi}_a$ is also equal to $\Phi_a$. Therefore in the following computations we will always use the original operators and the affine coordinates on $\tilde{A}$. We compute the asymptotic behaviour of the gauge $\varphi_a$ in the following theorem:

\begin{theorem}\label{thm:asymptotic_gauge}
Let $\varphi_a\in \Omega^0(\tilde{A}_0,\End E\oplus TM)$ be the unique gauge such that $p(\varphi_a)=0$ and $e^{\varphi_a}\ast 0=\tilde{\Phi}_{a}$. Then the asymptotic behaviour of $\varphi_a$ is 
\[
\varphi_a^{(s)}\in\begin{cases}\sum_{l}
\left(B_l, b_l\partial_{n_a}\right)t^{s}\mathfrak{w}^{l m_a}+\bigoplus_{l\geq 1}\mathcal{W}_0^{0}(\tilde{A}_0,\End E\oplus TM)\mathfrak{w}^{lm_a} & \text{on }  H_{m_a,+}\\
\bigoplus_{l\geq 1}\mathcal{W}_0^{-\infty}(\tilde{A}_0,\End E\oplus TM)\mathfrak{w}^{lm_a} & \text{on }  H_{m_a,-}
\end{cases}\]
for every $s\geq 0$, where $\big(B_l, b_l\partial_{n_a}\big)\mathfrak{w}^{lm_a}\in\tilde{\mathfrak{h}}$.
\end{theorem}
\begin{remark}
Notice that, from Theorem \ref{thm:asymptotic_gauge} the gauge $\varphi_a$ is asymptotically an element of the dgLa $\tilde{\mathfrak{h}}$. Hence the saturated scattering diagram (see Definition \ref{def:D_infty}) is strictly contained in the mirror dgLa $G$ (see Definition \ref{def:mirror dgLa}). 
\end{remark}
We first need the following lemma (for a proof see Lemma 5.27 \cite{MCscattering}
) which gives the explicit asymptotic behaviour of each component of the Lie bracket $\lbrace\cdot,\cdot\rbrace_\sim$:
\begin{lemma}\label{lem:bracket natural}
Let $P_m$ and $P_{m'}$ be two rays on $U$ such that $P_{m_a}\cap P_{m_{a'}}=\lbrace m_0\rbrace$. 

If $(A\mathfrak{w}^{m},\alpha\mathfrak{w}^{m}\partial_n)\in\mathcal{W}_{0}^s(U,\End E\oplus TM)\mathfrak{w}^{m}$ and $(B\mathfrak{w}^{m'},\beta\mathfrak{w}^{m'}\partial_{n'})\in\mathcal{W}_{0}^r(U,\End E\oplus TM)\mathfrak{w}^{m'}$, then 
\begin{equation*}
\begin{split}
&\mathbf{H}\left(\lbrace (A\delta_m\mathfrak{w}^{m},\alpha\delta_m\mathfrak{w}^{m}\partial_n), (B\delta_{m'}\mathfrak{w}^{m'},\beta\delta_{m'}\mathfrak{w}^{m'}\partial_{n'})\rbrace_{\natural}\right)\in\mathcal{W}_{P_{m+m'}}^{s+r+1}(U,\End E\oplus TM)\mathfrak{w}^{m+m'}\\
&\mathbf{H}\left(\lbrace (\delta_m\mathfrak{w}^{m},\alpha\delta_m\mathfrak{w}^{m}\partial_n), (B\delta_{m'}\mathfrak{w}^{m'},\beta\delta_{m'}\mathfrak{w}^{m'}\partial_{n'})\rbrace_{\flat}\right)\in\mathcal{W}_{P_{m+m'}}^{s+r}(U,\End E\oplus TM)\mathfrak{w}^{m+m'}\\
&\mathbf{H}\left(\lbrace (A\delta_m\mathfrak{w}^{m},\alpha\delta_m\mathfrak{w}^{m}\partial_n), (B\delta_{m'}\mathfrak{w}^{m'},\beta\delta_{m'}\mathfrak{w}^{m'}\partial_{n'})\rbrace_{\sharp}\right)\in\mathcal{W}_{P_{m+m'}}^{s+r}(U,\End E\oplus TM)\mathfrak{w}^{m+m'}\\
&\mathbf{H}\left(\lbrace (A\delta_m\mathfrak{w}^{m},\alpha\delta_m\mathfrak{w}^{m}\partial_n), (B\delta_{m'}\mathfrak{w}^{m'},\beta\delta_{m'}\mathfrak{w}^{m'}\partial_{n'})\rbrace_{\star}\right)\in\mathcal{W}_{P_{m+m'}}^{s+r-1}(U,\End E\oplus TM)\mathfrak{w}^{m+m'}.
\end{split}
\end{equation*}
\end{lemma}
\begin{remark}
The homotopy operators $H$ and $\mathbf{H}$ are different. However it is not a problem because the operator $H$ produce a solution of Maurer-Cartan and not of equation \eqref{eq:Phi}.\footnote{Recall that we were looking for a solution $\Phi$ of Maurer-Cartan of the form $\Phi=\bar{\Pi}+\Xi$ and since $d_W(\bar{\Pi})=0$, the correction term $\Xi$ is a solution of $d_W\Xi=-\frac{1}{2}\lbrace\Phi,\Phi\rbrace_\sim$. At this point we have introduced the homotopy operator $\mathbf{H}$ in order to compute $\Xi$ and we got $\Xi=-\frac{1}{2}\mathbf{H}(\lbrace\Phi,\Phi\rbrace_\sim)$.}     
\end{remark}
\begin{proof}{[Theorem \ref{thm:asymptotic_gauge}]}
First of all recall that for every $s\geq 0$
\begin{equation}
\varphi_a^{(s+1)}=-H\left(\tilde{\Phi}_a+\sum_{k\geq 1}\frac{\mathsf{ad}^k_{\varphi_a^s}}{k!}d_W\varphi_a^s\right)^{(s+1)}
\end{equation}
where $H$ is the homotopy operator defined in \eqref{def:homotopy_1wall} with the same choice of the path $\varrho$ as in \eqref{varro}. In addition as in the proof of Proposition \ref{prop:asymp1},
\begin{equation}
-H\left(\sum_{k\geq 1}\frac{\mathsf{ad}_{\varphi^s_a}^k}{k!}d_W\varphi_a^s\right)^{(s+1)}\in\bigoplus_{l\geq 1}\mathcal{W}_{0}^0(\tilde{A}_0,\End E\oplus TM)\mathfrak{w}^{lm_a}
\end{equation}
hence we are left to study the asymptotic of $H(\tilde{\Phi}_a^{(s+1)})$. By definition $\tilde{\Phi}_a^{(s+1)}$ is the sum over all $k\leq s$-trees such that they have outgoing vertex with label $m_T=lm_a$ for some $l\geq 1$. 

We claim the following:
\begin{equation}
H(\tilde{\Phi}_a^{(s+1)})\in H(\nu_{T_\natural})+\bigoplus_{l\geq 1}\mathcal{W}_{0}^r(\tilde{A}_0,\End E\oplus TM)\mathfrak{w}^{lm_a}
\end{equation}
for every $T\in\mathbb{LRT}_{k,0}$ such that $\mathfrak{t}_{k,T}(\bar{\Pi},\cdots ,\bar{\Pi})$ has Fourier mode $m_T=lm_a$, for every $k\leq s$ and for some $r\leq 0$. Indeed if $k=1$ the tree has only one root and there is nothing to prove because there is no label. In particular 
\[H(\nu_{T})=H(\bar{\Pi}^{(s+1)})=H(\tilde{\Phi}_{(1,0)}^{(s+1)}+\tilde{\Phi}_{(0,1)}^{(s+1)})\]
and we will explicitly compute it below. Then at $k\geq 2$, every tree can be considered as a $2$-tree where the incoming edges are the roots of two sub-trees $T_1$ and $T_2$, not necessary in $\mathbb{LRT}_{0}$, such that their outgoing vertices look like 
\[
\begin{split}
\nu_{T_1}&=(A_k\delta_{m_{a'}}\mathfrak{w}^{km_{a'}},\alpha_k\delta_{m_{a'}}\mathfrak{w}^{km_{a'}}\partial_{n_{a'}})\in\bigoplus_{k\geq 1}\mathcal{W}_{P_{m_{a'}}}^{r'}(\tilde{A}_0,\End E\oplus TM)\mathfrak{w}^{km_{a'}}\\
\nu_{T_2}&=(B_{k''}\delta_{m_{a''}}\mathfrak{w}^{k''m_{a''}},\beta_{k''}\delta_{m_{a''}}\mathfrak{w}^{k''m_{a''}}\partial_{n_{a''}})\in\bigoplus_{k''\geq 1}\mathcal{W}_{P_{m_{a''}}}^{r''}(\tilde{A}_0,\End E\oplus TM)\mathfrak{w}^{k''m_{a''}}
\end{split}
\]
where $k'm_{a'}+k''m_{a''}=lm_a$. Thus it is enough to prove the claim for a $2$-tree with ingoing vertex $\nu_{T_1}$ and $\nu_{T_2}$ as above. If $T\in\mathbb{LRT}_2$, then $\nu_T=\nu_{T_\natural}+\nu_{T_\flat}+\nu_{T_\sharp}+\nu_{T_\star}
$ and we explicitly compute $H(\nu_{T_\flat}), H(\nu_{T_\sharp})$ and $H(\nu_{T_\star})$. 
\begin{equation*}
\begin{split}
H(\nu_{T_\flat})&=-\frac{1}{2}H\left(\mathbf{H}\left(\lbrace\nu_{T_1},\nu_{T_2}\rbrace_\flat\right)\right)\\
&=-\frac{1}{2}H\Big(\mathbf{H}\Big(\lbrace(A_k\delta_{m_{a'}}\mathfrak{w}^{km_{a'}},\alpha_k\delta_{m_{a'}}\mathfrak{w}^{km_{a'}}\delta_{n_{a'}}), (B_{k''}\delta_{m_{a''}}\mathfrak{w}^{k''m_{a''}},\beta_{k''}\delta_{m_{a''}}\mathfrak{w}^{k''m_{a''}}\partial_{n_{a''}})\rbrace_{\flat}\Big)\Big)\\
&=-\frac{1}{2}i\hbar H(\mathbf{H}\Big(\beta_{k''}n_2^q\frac{\partial}{\partial x_q}(A_k\delta_{m_{a''}})\wedge\delta_{m_{a'}}\mathfrak{w}^{km_{a'}+k''m_{a''}},\\
&\quad\beta_{k''}n_2^q\delta_{m_{a''}}\wedge\frac{\partial}{\partial x_q}\left(\alpha_k\delta_{m_{a'}}\right)\mathfrak{w}^{km_{a'}+k''m_{a''}}\partial_{n_1}\Big))\\
&=\frac{1}{2}i\hbar H(\Big(\left(\int_0^{u_{m_a}}\beta_{k''}n_2^q\frac{\partial A_k}{\partial x_q}\frac{e^{-\frac{s^2}{\hbar}}}{\sqrt{\pi\hbar}}ds\right)\delta_{m_a}\mathfrak{w}^{lm_a}, \\
&\quad\left(\int_0^{u_{m_a}}\beta_{k''}\frac{\partial\alpha_k}{\partial x_q}n_2^q\frac{e^{-\frac{s^2}{\hbar}}}{\sqrt{\pi\hbar}}ds\right)\delta_{m_a}\mathfrak{w}^{km_{a'}+k''m_{a''}}\partial_{n_{a'}}\Big))+\\
&\quad+\frac{1}{2}i\hbar H(\left(\int_0^{u_{m_a}}\frac{e^{-\frac{s^2}{\hbar}}}{\sqrt{\pi\hbar}}\left(\beta_{k''}n_2^qA_k, \beta_{k''}n_2^q\partial_{n_{a'}}\right)\frac{2\gamma_q(s)}{\hbar}ds\right)\delta_{m_a}\mathfrak{w}^{lm_a})\\
&\in\bigoplus_{l\geq 1}\mathcal{W}_{0}^{r'+r''-2}(\tilde{A}_0)\mathfrak{w}^{lm_a}+\mathcal{W}_{0}^{r'+r''-1}(\tilde{A}_0)\mathfrak{w}^{lm_a}
\end{split}
\end{equation*}
where we assume $k'm_{a'}+k''m_{a''}=lm_a$ and in the last step we use Lemma \ref{lem:H(W^s)} to compute the asymptotic behaviour of $H(\delta_{m_a})$.
We denote by $\gamma^q(s)$ the coordinates $u_{m_a^\perp}$ written as functions of $x^q(s)$. In particular, since $\tilde{\Phi}_{(1,0)}^{(1)}\in\bigoplus_{k_1\geq 1}\mathcal{W}_{P_{m_1}}^1(\tilde{A}_0,\End E\oplus TM)\mathfrak{w}^{k_1m_1}$ and $\tilde{\Phi}_{(0,1)}^{(1)}\in\bigoplus_{k_2\geq 1}\mathcal{W}_{P_{m_2}}^1(\tilde{A}_0,\End E\oplus TM)\mathfrak{w}^{k_2m_2}$, we have $H(\tilde{\Phi}_{(a_1,a_2)}^{(2)})\in\bigoplus_{l\geq 1}\mathcal{W}_{0}^0\mathfrak{w}^{lm_a}$. 
The same holds true for $H(\nu_{T_\sharp})$ by permuting $A,\alpha$ and $B,\beta$.

Then we compute the behaviour of $H(\nu_{T_\star})$:
\begin{equation*}
\begin{split}
H(\nu_{T_\star})&=-\frac{1}{2}H\left(\mathbf{H}\left(\lbrace\nu_{T_1},\nu_{T_2}\rbrace_\star\right)\right)\\
&=H\left(\mathbf{H}\left(\lbrace (A_k\delta_{m_{a'}}\mathfrak{w}^{km_{a'}},\alpha_k\delta_{m_{a'}}\mathfrak{w}^{km_{a'}}\delta_{n_{a'}}), (B_{k''}\delta_{m_{a''}}\mathfrak{w}^{k''m_{a''}},\beta_{k''}\delta_{m_{a''}}\mathfrak{w}^{k''m_{a''}}\partial_{n_{a''}})\rbrace_{\star}\right)\right)\\
&=i\hbar H( \mathbf{H}\left((\alpha_kn_1^qB_{k''}A_q(\phi)\delta_{m_{a''}}\wedge\delta_{m_{a'}}-n_2^q\beta_{k''}A_q(\phi)A_k\delta_{m_{a'}}\wedge\delta_{m_{a''}}, 0)\right))\\
&=i\hbar H\Bigg(\Big(\Bigg(\int_0^{u_{m_{a}}}\alpha_kn_1^qB_{k''}A_q(\phi)\frac{e^{-\frac{s^2}{\hbar}}}{\sqrt{\pi\hbar}}ds-\int_0^{u_{m_{a}}}n_2^q\beta_{k''}A_q(\phi)A_k\frac{e^{-\frac{s^2}{\hbar}}}{\sqrt{\pi\hbar}}ds \Bigg)\delta_{m_a}\mathfrak{w}^{lm_a}, 0\Big)\Bigg)\\
&\in\bigoplus_{l\geq 1}\mathcal{W}_{P_{m_{a}}}^{r'+r''-2}(\tilde{A}_0)\mathfrak{w}^{lm_a}
\end{split}
\end{equation*}
where we denote by $m_a$ the primitive vector such that for some $l\geq1$ $lm_a=k'm_{a'}+k''m_{a''}$. 
Finally let us compute $H(\nu_{T_\natural})$: at $k=1$ there is only a $1$-tree, hence $H(\nu_{T_\natural})=H(\tilde{\Phi}_{(1,0)}^{(s+1)}+\tilde{\Phi}_{(0,1)}^{(s+1)})$ and for all $k_1,k_2\geq 1$
\begin{equation*}
\begin{split}
H(\tilde{\Phi}_{(1,0)}^{(s+1)})&\in(A_{s+1,k_1}t^{s+1}, a_{s+1,k_1}t^{s+1}\partial_{n_1})H(\delta_{m_1}\mathfrak{w}^{k_1m_1})+\mathcal{W}_{P_{m_1}}^0(\tilde{A}_0,\End E\oplus TM)\mathfrak{w}^{k_1m_1} \\
&\in (A_{s+1,k_1}t^{s+1}, a_{s+1,k_1}t^{s+1}\partial_{n_1})\mathfrak{w}^{k_1m_1}+\mathcal{W}_{P_{m_1}}^{-1}(\tilde{A}_0\End E\oplus TM)\mathfrak{w}^{k_1m_1}\\
& \\
H(\tilde{\Phi}_{(0,1)}^{(s+1)})&\in (A_{s+1,k_2}t, a_{s+1,k_2}t\partial_{n_2})H(\delta_{m_2}\mathfrak{w}^{k_2m_2})+\mathcal{W}_{P_{m_2}}^{-1}(\tilde{A}_0,\End E\oplus TM)\mathfrak{w}^{k_2m_2}\\
&\in(A_{s+1,k_2}t^{s+1}, a_{s+1,k_2}t^{s+1}\partial_{n_2})\mathfrak{w}^{k_2m_2}+\mathcal{W}_{P_{m_2}}^{-1}(\tilde{A}_0,\End E\oplus TM)\mathfrak{w}^{k_2m_2}
\end{split}
\end{equation*} 
Then every other $k$-tree ($k\leq s$) can be decomposed in two sub-trees $T_1$ and $T_2$ as above, and we can further assume $T_1, T_2\in\mathbb{LRT}_0$, because if either $T_1$ or $T_2$ contains at least a label different from $\natural$ then by Lemma \ref{lem:bracket natural} their asymptotic behaviour is of higher order in $\hbar$.
We explicitly compute $H(\nu_{T_\natural})$ at $s=1$: 
\begin{equation*}
\begin{split}
H(\nu_{T_\natural})&=H(-\frac{1}{2}\mathbf{H}\left(\lbrace\bar{\Pi}^{(1)},\bar{\Pi}^{(1)}\rbrace_\natural\right))\\
&=-\frac{1}{2}H\left(\mathbf{H}\left(\lbrace\bar{\Pi}_1^{(1)},\bar{\Pi}_2^{(1)}\rbrace_{\natural}+\lbrace\bar{\Pi}_2^{(1)},\bar{\Pi}_1^{(1)}\rbrace_{\natural}\right)\right)\\
&=-H\Big(\mathbf{H}\Big(\big(a_{1,k_1}A_{1,k_2}\langle n_1,k_2m_2\rangle- a_{1,k_2}A_{1,k_1}\langle n_2,k_1m_1\rangle+ [A_{1,k_1},A_{1,k_2}],\\
&\quad a_{1,k_1}a_{1,k_2}\partial_{\langle n_1,k_2m_2\rangle n_2+\langle n_2,k_1m_1\rangle n_1}\big)t^2\delta_{m_1}\wedge\delta_{m_2}\mathfrak{w}^{k_1m_1+k_2m_2})\Big)\Big)\\
&=-H\Big(\Big(a_{1,k_1}A_{1,k_2}\langle n_1,k_2m_2\rangle- a_{1,k_2}A_{1,k_1}\langle n_2,k_1m_1\rangle+ [A_{1,k_1},A_{1,k_2}],\\
&a_{1,k_1}a_{1,k_2}\partial_{\langle n_1,k_2m_2\rangle n_2+\langle n_2,k_1m_1\rangle n_1}\Big)t^2\left(\int_0^{u_{m_a}}\frac{e^{-\frac{s^2}{\hbar}}}{\sqrt{\hbar\pi}}ds\right)\delta_{m_a}\mathfrak{w}^{lm_a}\Big)\\
&=-\Big(a_{1,k_1}A_{1,k_2}\langle n_1,k_2m_2\rangle- a_{1,k_2}A_{1,k_1}\langle n_2,k_1m_1\rangle+ [A_{1,k_1},A_{1,k_2}],\\
&\quad a_{1,k_1}a_{1,k_2}\partial_{\langle n_1,k_2m_2\rangle n_2+\langle n_2,k_1m_1\rangle n_1}\Big)t^2 H\left(\left(\int_0^{u_{m_a}}\frac{e^{-\frac{s^2}{\hbar}}}{\sqrt{\hbar\pi}}ds\right)\delta_{m_a}\mathfrak{w}^{lm_a}\right)\\
&=-\Big(a_{1,k_1}A_{1,k_2}\langle n_1,k_2m_2\rangle- a_{1,k_2}A_{1,k_1}\langle n_2,k_1m_1\rangle+ [A_{1,k_1},A_{1,k_2}],\\
&\quad a_{1,k_1}a_{1,k_2}\partial_{\langle n_1,k_2m_2\rangle n_2+\langle n_2,k_1m_1\rangle n_1}\Big)t^2 \mathfrak{w}^{lm_a}H_{lm_a}\left(\left(\int_0^{u_{m_a}}\frac{e^{-\frac{s^2}{\hbar}}}{\sqrt{\hbar\pi}}ds\right)\delta_{m_a}\right)
\end{split}
\end{equation*}
Now 
\begin{equation*}
\mathfrak{w}^{lm_a}H_{lm_a}\left(\left(\int_0^{u_{m_a}}\frac{e^{-\frac{s^2}{\hbar}}}{\sqrt{\hbar\pi}}ds\right)\delta_{m_a}\right)\in\mathfrak{w}^{lm_a}+\mathcal{W}_0^0(\tilde{A}_0)\mathfrak{w}^{lm_a}
\end{equation*}
Therefore the leading order term of $H(\nu_{T_\natural})$ with labels $m_T=lm_a=k_1m_1+k_2m_2$ at $s=1$ is
\[
\big(a_{1,k_1}A_{1,k_2}\langle n_1,k_2m_2\rangle- a_{1,k_1}A_{1,k_2}\langle n_2,k_1m_1\rangle+ [A_{1,k_1},A_{1,k_2}],a_{1,k_1}a_{1,k_2}\partial_{\langle n_1,k_2m_2\rangle n_2+\langle n_2,k_1m_1\rangle n_1}\big)t^2\mathfrak{w}^{lm_a}
\]
At $s\geq 2$, every other $k$-tree $T$ ($k\leq s$) can be decomposed in two sub-trees, say $T_1$ and $T_2$ such that $\nu_{T_\natural}=-\mathbf{H}(\lbrace\nu_{T_{1,\natural}},\nu_{T_{2,\natural}}\rbrace_\natural)+\bigoplus_{l\geq  1}\mathcal{W}_{P_{m_a}}^1(\tilde{A}_0)\mathfrak{w}^{lm_a}$. 

Notice that the leading order term of $H\left(\mathbf{H}(\lbrace\bar{\Pi}_1^{(1)},\bar{\Pi}_2^{(1)}\rbrace_{\natural})\right)$ is the Lie bracket of \[\lbrace(A_{1,k_1}\mathfrak{w}^{k_1m_1}, a_{1,k_1}\mathfrak{w}^{k_1m_1}\partial_{n_1}),(A_{1,k_2}\mathfrak{w}^{k_2m_2},a_{1,k_2}\mathfrak{w}^{k_2m_2}\partial_{n_2})\rbrace_{\tilde{\mathfrak{h}}}\] hence the leading order term of $H(\nu_{T_\natural})$ belongs to $\tilde{\mathfrak{h}}$.  

\end{proof}

Notice that at any order in the formal parameter $t$, there are only a finite number of terms which contribute to the solution $\Phi$ in the sum \eqref{def:Phitree}, hence we define the set $\mathbb{W}(N)$ as
\begin{equation}
\mathbb{W}(N)\defeq\lbrace a\in\left(\Z^2_{\geq 0}\right)_{\text{prim}}\vert lm_a=m_T \text{ for some } l\geq 0 \text{and } T\in\mathbb{LRT}_k \text{with } 1\leq j_T\leq N\rbrace.
\end{equation}
\begin{definition}[Scattering diagram $\mathfrak{D}_\infty$]\label{def:D_infty}
The order $N$ scattering diagram $\mathfrak{D}_N$ associated to the solution $\Phi$ is \[\mathfrak{D}_N\defeq\big\lbrace\mathsf{w}_1, \mathsf{w}_2\big\rbrace\cup\big\lbrace\mathsf{w}_a=\big(m_a, P_{m_a}, \theta_a\big)\big\rbrace_{a\in\mathbb{W}(N)}\] where
\begin{itemize}
    \item $m_a=a_1m_1+a_2m_2$;
    \item $P_{m_a}=m_0+m_a\R_{\geq 0}$
    \item $\log(\theta_a)$ is the leading order term of the unique gauge $\varphi_a$, as computed in Theorem \eqref{thm:asymptotic_gauge}.
\end{itemize}
The scattering diagram $\mathfrak{D}_\infty\defeq\varinjlim_N\mathfrak{D}_N$.
\end{definition}

\subsection{Consistency of $\mathfrak{D}_\infty$}
We are left to prove consistency of the scattering diagram $\mathfrak{D}_\infty$ associated to the solution $\Phi$. In order to do that we are going to use a monodromy argument (the same approach was used in \cite{MCscattering}).

Let us define the following regions
\begin{align}
&\tilde{\mathbb{A}}\defeq\lbrace(r,\vartheta)\vert \vartheta_0-\epsilon_0+2\pi<\vartheta<\vartheta_0+2\pi\rbrace\\
&\tilde{\mathbb{A}}-2\pi\defeq\lbrace(r,\vartheta)\vert \vartheta_0-\epsilon_0<\vartheta<\vartheta_0\rbrace.
\end{align}  
for small enough $\epsilon_0>0$, such that $\tilde{\mathbb{A}}-2\pi$ is away from all possible walls in $\mathfrak{D}_\infty$.

\begin{theorem}\label{thm:consistentD}
Let $\mathfrak{D}_\infty$ be the scattering diagram defined in \eqref{def:D_infty}. Then it is consistent, i.e. $\Theta_{\mathfrak{D}_\infty, \gamma}=\text{Id}$ for any closed path $\gamma$ embedded in $U\setminus\lbrace m_0\rbrace$, which intersects $\mathfrak{D}_\infty$ generically. 
\end{theorem} 
\begin{proof}
It is enough to prove that $\mathfrak{D}_N$ is consistent for any fixed $N>0$. First of all recall that $\Theta_{\gamma,\mathfrak{D}_N}=\prod_{a\in\mathbb{W}(N)}^\gamma \theta_a $.  
Then let us prove that the following identity \begin{equation}\label{eq:step1}
\prod_{a\in\mathbb{W}(N)}^\gamma e^{\varphi_a}\ast 0=\sum_{a\in\mathbb{W}(N)}\tilde{\Phi}_{a}
\end{equation}
holds true. 
Indeed \[\left(e^{\varphi_a}\ast e^{\varphi_{a'}}\right)\ast 0=e^{\varphi_a}\ast\left(\tilde{\Phi}_{a'}\right)=\tilde{\Phi}_{a'}-\sum_k\frac{\textsf{ad}^k_{\varphi_a}}{k!}\left(d_W\varphi_a+\lbrace\varphi_a,\tilde{\Phi}_{a'}\rbrace_\sim\right).\]
For degree reason $\lbrace \varphi_a,\tilde{\Phi}_{a'}\rbrace_\sim=0$ and by definition \[-\sum_k\frac{\textsf{ad}^k_{\varphi_a}}{k!}\big(d_W\varphi_a\big)=e^{\varphi_a}\ast 0=\tilde{\Phi}_{a} .\] 
Iterating the same procedure for more than two rays, we get the result. 

Recall that if $\varphi$ is the unique gauge such that $p(\varphi)=0$ and $e^\varphi\ast0=\Phi$, then $e^{\varpi^*(\varphi)}\ast 0=\varpi^*(\Phi)$ on $\tilde{A}$. Hence
$e^{\varpi^*(\varphi)}\ast 0=\varpi^*(\Phi)=\sum_{a\in\mathbb{W}(N)}\varpi^*(\Phi_{a})=\sum_{a\in\mathbb{W}(N)}\tilde{\Phi}_{a}$ and by equation \eqref{eq:step1}
\[e^{\varpi^*(\varphi)}\ast 0=\prod_{a\in\mathbb{W}(N)}^\gamma e^{\varphi_a}\ast 0.\]
In particular, by uniqueness of the gauge, $e^{\varpi^*(\varphi)}=\prod_{a\in\mathbb{W}(N)}^\gamma e^{\varphi_a}.$ Since $\varpi^*(\varphi)$ is defined on all $U$, $e^{\varpi^*(\varphi)}$ is monodromy free, i.e.
\[\prod_{a\in\mathbb{W}(N)}^\gamma e^{\varphi_a}\vert_{\tilde{\mathbb{A}}}=\prod_{a\in\mathbb{W}(N)}^\gamma e^{\varphi_a}\vert_{\tilde{\mathbb{A}}-2\pi}.\]
Notice that $\tilde{\mathbb{A}}-2\pi$ does not contain the support of $\varphi_a$ $\forall a\in\left(\Z_{\geq 0}^2\right)_{\text{prim}}$, therefore \[\prod_{a\in\mathbb{W}(N)}^\gamma e^{\varphi_a}\vert_{\tilde{\mathbb{A}}-2\pi}=\prod_{a\in\mathbb{W}(N)}^\gamma e^{0}=\text{Id}.\]

\end{proof}
\section{Relation with the wall-crossing formulas in coupled $2d$-$4d$ systems}\label{sec:WCF}

We are going to show how wall-crossing formulas in coupled $2d$-$4d$ systems, introduced by Gaiotto, Moore and Nietzke in \cite{WCF2d-4d},  can be interpreted in the framework we were discussing before. Let us first recall the setting for the $2d$-$4d$ WCFs: 

\begin{itemize}
	\item let $\Gamma$ be a lattice, whose elements are denoted by $\gamma$;
	\item  define an antisymmetric bilinear form $\langle\cdot ,\cdot\rangle_D\colon\Gamma\times\Gamma\to\Z$, called the Dirac pairing;
	\item  let $\Omega\colon\Gamma\to\Z$ be a homomorphism;
	\item denote by $\mathcal{V}$ a finite set of indices, $\mathcal{V}=\lbrace i,j,k,\cdots\rbrace$;
    \item define a $\Gamma$-torsor $\Gamma_i$, for every $i\in\mathcal{V}$. Elements of $\Gamma_i$ are denoted by $\gamma_i$ and the action of $\Gamma$ on $\Gamma_i$ is $\gamma+\gamma_i=\gamma_i+\gamma$;
    \item define another $\Gamma$-torsor $\Gamma_{ij}\defeq\Gamma_i-\Gamma_j$ whose elements are formal differences $\gamma_{ij}\defeq\gamma_i-\gamma_j$ up to equivalence, i.e. $\gamma_{ij}=(\gamma_{i}+\gamma)-(\gamma_j+\gamma)$ for every $\gamma\in\Gamma$. If $i=j$, then $\Gamma_{ii}$ is identify with $\Gamma$. The action of $\Gamma$ on $\Gamma_{ij}$ is $\gamma_{ij}+\gamma=\gamma+\gamma_{ij}$. Usually it is not possible to sum elements of $\Gamma_{ij}$ and $\Gamma_{kl}$, for instance $\gamma_{ij}+\gamma_{kl}$ is well defined only if $j=k$ and in this case it is an element of $\Gamma_{il}$;
    \item let $Z\colon\Gamma\to\C$ be a homomorphism and define its extension as an additive map $Z\colon\amalg_{i\in\mathcal{V}}\Gamma_{i}\to\C$, such that $Z({\gamma+\gamma_i})=Z(\gamma)+Z({\gamma_i})$. In particular, by additivity $Z$ is a map from $\amalg_{i,j\in\mathcal{V}}\Gamma_{ij}$ to $\C$, namely $Z(\gamma_{ij})=Z(\gamma_i)-Z(\gamma_j)$. The map $Z$ is usually called the central charge;
    \item let $\sigma(a,b)$ be a \textit{twisting} function defined whenever $a+b$ is defined for $a,b\in\Gamma\sqcup\amalg_i\Gamma_i\sqcup\amalg_{i\neq j}\Gamma_{ij}$ and valued in $\lbrace\pm 1\rbrace$. Moreover it satisfies the following conditions:
\begin{equation}
\begin{split}
&\text{(i)}\quad \sigma(a,b+c)\sigma(b,c)=\sigma(a,b)\sigma(a+b,c) \\
&\text{(ii)}\quad {\sigma(a,b)=\sigma(b,a)} \text{ if both $a+b$ and $b+a$ are defined} \\
&\text{(iii)}\quad \sigma(\gamma,\gamma')=(-1)^{\langle\gamma,\gamma'\rangle_D}\,\,\forall\gamma,\gamma'\in\Gamma;
\end{split}
\end{equation}
    \item let $X_a$ denote formal variables, for every $a\in\Gamma\sqcup\amalg_i\Gamma_i\sqcup\amalg_{i\neq j}\Gamma_{ij}$. There is a notion of associative product: 
\[
X_{a}X_{b}\defeq\begin{cases}
\sigma(a,b)X_{a+b} & \text{if the sum $a+b$ is defined }\\
0 & \text{otherwise}
\end{cases}
\]
\end{itemize}

The previous data fit well in the definition of a pointed groupoid $\mathbb{G}$, as it is defined in \cite{WCF2d-4d}. In particular $\Ob(\mathbb{G})=\mathcal{V}\sqcup\lbrace o\rbrace$ and $\Mor(\mathbb{G})=\amalg_{i,j\in\Ob(\mathbb{G})}\Gamma_{ij}$, where the torsor $\Gamma_i$ is identify with $\Gamma_{io}$ and elements of $\Gamma$ are identify with $\amalg_i\Gamma_{ii}$. The composition of morphism is written as a sum, and the formal variables $X_a$ are elements of the groupoid ring $\C[\mathbb{G}]$. 
In this setting, BPS rays are defined as 
\[
\begin{split}
&l_{ij}\defeq Z(\gamma_{ij})\R_{>0}\\
&l\defeq Z(\gamma)\R_{>0}
\end{split}
\] and they are decorated with automorphisms of $\C[\mathbb{G}][\![ t ]\!]$ respectively of type $S$ and of type $K$, defined as follows: let $X_a\in\C[\mathbb{G}]$, then

\begin{equation}\label{eq:autS}
S_{\gamma_{ij}}^{\mu}(X_a)\defeq\big(1-\mu({\gamma_{ij}})tX_{\gamma_{ij}}\big)X_a\big(1+\mu(\gamma_{ij})tX_{\gamma_{ij}}\big)    
\end{equation}
where $\mu\colon\amalg_{i,j\in\mathcal{V}}\Gamma_{ij}\to\Z$ is a homomorphism;
\begin{equation}\label{eq:autK}
 K_{\gamma}^{\omega}(X_a)\defeq\big(1-X_{\gamma}t\big)^{-\omega(\gamma,a)}X_a   
\end{equation}
where $\omega\colon\Gamma\times\amalg_{i\in\mathcal{V}}\Gamma_i\to\Z$ is a homomorphism such that $\omega(\gamma,\gamma')=\Omega(\gamma)\langle\gamma,\gamma'\rangle_D$ and $\omega(\gamma,a+b)=\omega(\gamma,a)+\omega(\gamma,b)$ for $a,b\in\mathbb{G}$.

In particular under the previous assumption, the action of $S_{\gamma_{ij}}^\mu$ and $K_\gamma^\omega$ can be explicitly computed on variables $X_\gamma$ and $X_{\gamma_k}$ as follows:
\begin{equation}\label{eq:actionP}
\begin{split}
    S_{\gamma_{ij}}^\mu &\colon X_{\gamma'}\to X_{\gamma'}\\
    S_{\gamma_{ij}}^\mu &\colon X_{\gamma_k}\to\begin{cases}  X_{\gamma_k} & \text{if } k\neq j\\
    X_{\gamma_j}- \mu(\gamma_{ij})tX_{\gamma_{ij}}X_{\gamma_j} & \text{if } k=j
     \end{cases}\\
    K_\gamma^\omega &\colon X_{\gamma'}\to (1-tX_\gamma)^{-\omega(\gamma,\gamma')}X_{\gamma'}\\
    K_\gamma^\omega &\colon X_{\gamma_k}\to (1-tX_\gamma)^{-\omega(\gamma,\gamma_k)}X_{\gamma_k}
   \end{split}
\end{equation}

In order to interpret the automorphisms $S$ and $K$ as elements of $\exp(\tilde{\mathfrak{h}})$ we are going to introduce their infinitesimal generators. Let $\Der\left(\C[\mathbb{G}]\right)$ be the Lie algebra of the derivations of $\C[\mathbb{G}]$ and define:
\begin{equation}
    \mathfrak{d}_{\gamma_{ij}}\defeq \textsf{ad}_{X_{\gamma_{ij}}}
\end{equation}
where $\mathfrak{d}_{\gamma_{ij}}(X_a)\defeq \left( X_{\gamma_{ij}}X_a-X_aX_{\gamma_{ij}}\right)$, for every $X_a\in\C[\mathbb{G}]$; 
\begin{equation}
    \mathfrak{d}_\gamma\defeq \omega(\gamma,\cdot)X_\gamma
\end{equation}
where $\mathfrak{d}_\gamma(X_a)\defeq \left(\omega(\gamma,a)X_{\gamma}X_{a}\right)$, for every $X_a\in\C[\mathbb{G}]$. 
\begin{definition}
Let $\mathbf{L}_\Gamma$ be the $\C[\Gamma]$- module generated by $\mathfrak{d}_{\gamma_{ij}}$ and $\mathfrak{d}_\gamma$, for every $i\neq j\in\mathcal{V}$, $\gamma\in\Gamma$. 
\end{definition}
For instance a generic element of $\mathbf{L}_\Gamma$ is given by \[\sum_{i,j\in\mathcal{V}}\sum_{l\geq 1}c_l^{(\gamma_{ij})}X_{a_{l}}\mathfrak{d}_{\gamma_{ij}}+\sum_{\gamma\in\Gamma}\sum_{l\geq 1}c_l^{(\gamma)}X_{a_l}\mathfrak{d}_{\gamma}\] 
where $c_l^{(\bullet)}X_{a_l}\in\C[\Gamma]$.  
\begin{lemma}\label{lem:CGamma_module}
Let $\mathbf{L}_\Gamma$ be the $\C[\Gamma]$-module defined above. Then, it is a Lie ring\footnote{A Lie ring $\mathbf{L}_\Gamma$ is an abelian group $(L,+)$ with a bilinear form $[,]\colon:L\times L\to L$ such that 
\begin{enumerate}
\item $[\cdot,\cdot]$ is antisymmetric, i.e. $[a,b]=-[b,a]$;
\item $[\cdot,\cdot]$ satisfy the Jacobi identity. 
\end{enumerate}
} with the the Lie bracket $[\cdot, \cdot]_{\Der(\C[\mathbb{G}])}$ induced by $\Der(\C[\mathbb{G}])$\footnote{$\mathbf{L}_\Gamma$ is not a Lie algebra over $\C[\Gamma]$ because the bracket induced from $\Der(\C[\mathbb{G}])$ is not $\C[\Gamma]-$linear.}.  
\end{lemma}
A proof of this Lemma is in appendix \ref{proof:CGamma_module}.

We can now define the infinitesimal generators of $S_{\gamma_{ij}}^\mu$ and $K_\gamma^\omega$ as elements of $\mathbf{L}_\Gamma$: we first define  
\begin{equation}
    \mathfrak{s}_{\gamma_{ij}}\defeq -\mu(\gamma_{ij})t\mathfrak{d}_{\gamma_{ij}}
\end{equation}
then  $\exp(\mathfrak{s}_{\gamma_{ij}})=S_{\gamma_{ij}}^\mu$, indeed
\begin{equation*}
    \begin{split}
     \exp(\mathfrak{s}_{\gamma_{ij}})(X_a)&=\sum_{k\geq0 }\frac{1}{k!}\mathfrak{s}_{\gamma_{ij}}^k(X_a)\\
     &=\sum_{k\geq0}\frac{(-1)^k}{k!}t^k\mu(\gamma_{ij})^k\mathsf{ad}_{X_{\gamma_{ij}}}^k(X_a)\\
     &=X_a-\mu(\gamma_{ij})t\ad_{\gamma_{ij}}(X_a)+\frac{1}{2}t^2\mu(\gamma_{ij})^2\mathsf{ad}_{\gamma_{ij}}^2(X_a),
  \end{split}
\end{equation*}
where $\mathsf{ad}_{\gamma_{ij}}(X_a)=X_{\gamma_{ij}}X_{a}-X_{a}X_{\gamma_{ij}}$. Hence 

\[\mathsf{ad}_{\gamma_{ij}}^2(X_a)=-2X_{\gamma_{ij}}X_aX_{\gamma_{ij}}-t^2X_{\gamma_{ij}}X_{a}X_{\gamma_{ij}}\]

and since $\gamma_{ij}$ can not be composed with $\gamma_{ij}+a+\gamma_{ij}$, then  $\ad_{\gamma_{ij}}^3(X_a)=0$. Moreover if $a\in\Gamma$ then $\ad_{\gamma_{ij}}X_a=0$, while if $a=\gamma_{ok}$ then $\mathsf{ad}^2X_a=0$ and we recover the formulas \eqref{eq:actionP}.
 
Then we define 
\begin{equation}
    \mathfrak{k}_{\gamma}\defeq\sum_{l\geq 1}\frac{1}{l}t^{l}X_{\gamma}^{(l-1)}\mathfrak{d}_\gamma
\end{equation}
and we claim $\exp(\mathfrak{k}_\gamma)=K_{\gamma}^\omega$, indeed
     
\begin{equation*}
    \begin{split}
     \exp(\mathfrak{k}_\gamma)(X_a)&=\sum_{k\geq 0} \frac{1}{k!}\mathfrak{k}_\gamma^k(X_a)  \\
&=\sum_{k\geq0}\frac{1}{k!}\left(\sum_{l_k\geq1}\frac{1}{l_k}t^{l_k}X_{\gamma}^{l_k}\omega(\gamma,\cdot)\left(\cdots\left(\sum_{l_2\geq1}t^{l_2}\frac{1}{l_2}X_{\gamma}^{l_2}\omega(\gamma,\cdot)\left(\sum_{l_1\geq1}\frac{1}{l_1}t^{l_1}\omega(\gamma,a)X_{l_1\gamma}X_{a}\right)\right)\cdots\right)\right)\\
&=\sum_{k\geq0}\frac{1}{k!}\left(\sum_{l_k\geq1}\frac{1}{l_k}t^{l_k}X_{\gamma}^{l_k}\omega(\gamma,\cdot)\left(\cdots\left(\sum_{l_2\geq1}\sum_{l_1\geq1}\frac{1}{l_1l_2}t^{l_1+l_2}\omega(\gamma,l_1\gamma+a)\omega(\gamma,a)X_{\gamma}^{l_2}X_{\gamma}^{l_1}X_{a}\right)\cdots\right)\right)\\
&=\sum_{k\geq0}\frac{1}{k!}\left(\sum_{l_k\geq1}\frac{1}{l_k}t^{l_k}X_{\gamma}^{l_k}\omega(\gamma,\cdot)\left(\cdots\left(\sum_{l_2\geq1}\sum_{l_1\geq1}\frac{1}{l_1l_2}t^{l_1+l_2}\omega(\gamma,a)\omega(\gamma,a)X_{\gamma}^{l_2+l_1}X_{a}\right)\cdots\right)\right)\\
     &=\sum_{k\geq0}\frac{1}{k!}\omega(\gamma+a)^k\left(\sum_{l\geq1}\frac{1}{l}t^lX_{\gamma}^l\right)^kX_a\\ 
&=\exp\left(-\omega(\gamma,a)\log(1-tX_\gamma)\right)X_a.
    \end{split}
\end{equation*}

From now on we are going to assume that $\Gamma\cong\Z^2\cong\Lambda$. We distinguish between polynomial in $\C[\Gamma]$ and $\C[\Lambda]$ by writing $X_\gamma$ for a variable in $\C[\Gamma]$ and $\mathfrak{w}^\gamma$ as a variable in $\C[\Lambda]$.   
\begin{remark}
The group ring $\C[\Gamma]$ is isomorphic to $\C[\Lambda]$ even if there two different products: on $\C[\Gamma]$ the product is $X_{\gamma}X_{\gamma'}\defeq\sigma(\gamma,\gamma')X_{\gamma+\gamma'}=(-1)^{\langle\gamma,\gamma'\rangle_D}X_{\gamma+\gamma'}$ while the product in $\C[\Lambda]$ is defined by $\mathfrak{w}^{\gamma}\mathfrak{w}^{\gamma'}=\mathfrak{w}^{\gamma+\gamma'}$. In particular the isomorphism depends on the choice of $\sigma$. 
\end{remark}

Let us choose an element $e_{ij}\in\Gamma_{ij}$ for every $i\neq j\in\mathcal{V}$ and set $e_{ii}\defeq0\in\Gamma$ for every $i\in\mathcal{V}$. Under this assumption $\mathbf{L}_\Gamma$ turns out to be generated by $ \mathfrak{d}_{e_{ij}}$ for all $i\neq j\in\mathcal{V}$ and by $\mathfrak{d}_{\gamma}$ for every $\gamma\in\Gamma$. Indeed every $\gamma_{ij}\in\Gamma_{ij}$ can be written as $e_{ij}+\gamma$ for some $\gamma\in\Gamma$ and $\mathfrak{d}_{\gamma_{ij}}=\mathfrak{d}_{e_{ij}+\gamma}=X_\gamma\mathfrak{d}_{e_{ij}}$. Then, we define an additive map \[
\begin{split}
m&\colon\amalg_{i,j\in\mathcal{V}}\Gamma_{ij}\to\Gamma\\
m&(\gamma_{ij})\defeq \gamma_{ij}-e_{ij}
\end{split}
\]
In particular, notice that $m(\gamma_{ii})=\gamma_{ii}-e_{ii}=\gamma_{ii}$, hence, since $\Gamma=\amalg_i\Gamma_{ii}$, $m(\Gamma)=\Gamma$.  

We now define a $\C[\Gamma]$-module in the Lie algebra $\tilde{\mathfrak{h}}$: 
\begin{definition}
Define $\tilde{\mathbf{L}}$ as the $\C[\Lambda]$-module generated by $\tilde{\mathfrak{l}}_{\gamma_{ij}}\defeq\left(E_{ij}\mathfrak{w}^{m(\gamma_{ij})},0\right)$ for every $i\neq j\in\mathcal{V}$ and $\tilde{\mathfrak{l}}_\gamma\defeq\left(0, \Omega(\gamma)\mathfrak{w}^{\gamma}\partial_{n_\gamma}\right)$ for every $\gamma\in\Gamma$, where $E_{ij}\in\mathfrak{gl}(r)$ is an elementary matrix with all zeros and a $1$ in position $ij$.
\end{definition}
 
\begin{lemma}\label{lem:CLambdamodule}
The $\C[\Lambda]$-module $\tilde{\mathbf{L}}$ is a Lie ring with respect to the Lie bracket induced by $\tilde{\mathfrak{h}}$ (see Definition \eqref{eq:Liebracket}).\footnote{$\tilde{\mathbf{L}}$ is not a Lie sub-algebra of $\\tilde{\mathfrak{h}}$ because the Lie bracket is not $\C[\Lambda]-$linear.}
\end{lemma}
The proof of this Lemma is postponed in the appendix (see proof \ref{proof:CLambdamodule}) as well as the proof of the following theorem (see proof \ref{proof:homoLiering}).

\begin{theorem}\label{thm:homomLiering}Let $\left(\mathbf{L}_\Gamma,[\cdot,\cdot]_{\Der(\C[\mathbb{G}])}\right)$ and $\left(\tilde{\mathbf{L}},[\cdot,\cdot]_{\tilde{\mathfrak{h}}}\right)$ be the $\C[\Gamma]$-modules defined before. 
Assume $\omega(\gamma,a)=\Omega(\gamma)\langle{ a},n_{\gamma}\rangle$, then there exists a homomorphism of $\C[\Gamma]$-modules and of Lie rings $\Upsilon\colon \mathbf{L}_\Gamma\to\tilde{\mathbf{L}}$, which is defined as follows:
\begin{equation}
\begin{split}
&\Upsilon(X_{\gamma}\mathfrak{d}_{\gamma_{ij}})\defeq\mathfrak{w}^{\gamma}\left(-E_{ij}\mathfrak{w}^{m(\gamma_{ij})},0\right), \forall i\neq j\in\mathcal{V}, \forall\gamma\in\Gamma;\\
&\Upsilon(X_{\gamma'}\mathfrak{d}_\gamma)\defeq \mathfrak{w}^{\gamma'}\left(0, \Omega(\gamma)\mathfrak{w}^{\gamma}\partial_{n_\gamma}\right), \forall \gamma',\gamma\in\Gamma
\end{split} 
\end{equation} 
\end{theorem}

\begin{remark}
The assumption on $\omega$ is compatible with its Definition \eqref{eq:autK}, indeed by linearity of the pairing $\langle\cdot,\cdot\rangle$, 
\[\omega(\gamma,a+b)=\Omega(\gamma)\langle {a+b},n_\gamma\rangle=\Omega(\gamma)\langle{a},n_\gamma\rangle+\Omega(\gamma)\langle {b},n_\gamma\rangle=\omega(\gamma,a)+\omega(\gamma, b).\] 
Moreover notice that by the assumption on $\omega$, $\mathbf{L}_\Gamma$ turns out to be the $\C[\Gamma]$-module generated by $\mathfrak{d}_{e_{ij}}$ for every $i\neq j\in\mathcal{V}$ and by $\mathfrak{d}_{\gamma}$ for every primitive $\gamma\in\Gamma$. Indeed if $\gamma'$ is not primitive, then there exists a $\gamma\in\Gamma_{\text{prim}}$ such that $\gamma'=k\gamma$. Hence $\mathfrak{d}_{k\gamma}=\omega(k\gamma,\cdot)X_{\gamma}^{(k-1)}X_\gamma=CX_{(k-1)\gamma}\mathfrak{d}_{\gamma}$, where $C=\frac{k\Omega(k\gamma)}{\Omega(\gamma)}$.  In particular, if $\gamma,\gamma'$ are primitive vectors in $\Gamma$, then $\omega(\gamma,\gamma')=\Omega(\gamma)\langle {\gamma'},n_{\gamma}\rangle=\Omega(\gamma)\langle \gamma,\gamma'\rangle_D$.     
\end{remark}

Let us now show which is the correspondence between WCFs in coupled $2d$-$4d$ systems and scattering diagrams which come from solutions of the Maurer-Cartan equation for deformations of holomorphic pairs:
\begin{enumerate}
    \item to every BPS ray $l_a=Z(a)\R_{>0}$ we associate a ray $P_{a}=m(a)\R_{>0}$ if either $\mu(a)\neq\mu(a)'$ or $\omega(a,\cdot)\neq\omega(a,\cdot)'$. Conversely we associate a line $P_a=m(a)\R$;
    \item to the automorphism $S_{\gamma_{ij}}^\mu$ we associate an automorphism $\theta_S\in\exp(\tilde{\mathfrak{h}})$ such that $\log(\theta_S)=\Upsilon(\mathfrak{s}_{\gamma_{ij}})=\left(-\mu(\gamma_{ij})tE_{ij}\mathfrak{w}^{m(\gamma_{ij})},0\right)$;
\item to the automorphism $K_\gamma^\omega$ we associate an automorphism $\theta_K\in\exp(\tilde{\mathfrak{h}})$ such that $\log(\theta_K)=\Upsilon(\mathfrak{k}_\gamma)=\left(0,\Omega(\gamma)\sum_l\frac{1}{l}t^l\mathfrak{w}^{l\gamma}\partial_{n_{\gamma}}\right)$. 
\end{enumerate}
\begin{remark}
If $m(\gamma_{ij})=m(\gamma_{il}+\gamma_{lj})$ then $\log(\theta_S)=\Upsilon(\mathfrak{s}_{\gamma_{ij}})=\Upsilon(\mathfrak{s}_{\gamma_{il}})\Upsilon(\mathfrak{s}_{\gamma_{lj}})$. Analogously if $m(\gamma_{ij}')=m(\gamma_{ij})+k\gamma$ then $\log(\theta_S')=\Upsilon(\mathfrak{s}_{\gamma_{ij}'})=t^k\Upsilon(\mathfrak{s}_{\gamma_{ij}})$. 
\end{remark}
In the following examples we will show this correspondence in practice: we consider two examples of WCFs computed in \cite{WCF2d-4d} and we construct the corresponding consistent scattering diagram. 

\subsubsection{Example 1}\label{Ex:1}
Let $\mathcal{V}=\lbrace i,j,k=l\rbrace$ and set $\gamma_{kk}=\gamma\in\Gamma$. Assume $\omega(\gamma,\gamma_{ij})=-1$ and $\mu(\gamma_{ij})=1$, then the wall-crossing formula (equation 2.39 in \cite{WCF2d-4d}) is 
\begin{equation}
    K_\gamma^\omega S_{\gamma_{ij}}^\mu=S_{\gamma_{ij}}^{\mu'}S_{\gamma_{ij}+\gamma}^{\mu'}K_\gamma^{\omega'} 
\end{equation}
with $\mu'(\gamma_{ij})=1$, $\mu'(\gamma+\gamma_{ij})=-1$ and $\omega'=\omega$.

Since $\mu'(\gamma_{ij})=\mu(\gamma_{ij})$ and $\omega'=\omega$ the initial scattering diagram has two lines. In addition, since $-1=\omega(\gamma,\gamma_{ij})=\Omega(\gamma)\langle m(\gamma_{ij}), n_\gamma\rangle$, we can assume $\Omega(\gamma)=1$, $m(\gamma_{ij})=(1,0)$ and $\gamma=(0,1)$. Therefore the initial scattering diagram is   
\[
\mathfrak{D}=\left\lbrace\mathsf{w}_S=\left(m_S=m(\gamma_{ij}), P_S, \theta_S\right),\mathsf{w}_K=\left(m_K=\gamma, P_K, \theta_K\right) \right\rbrace\] where $\log\theta_S=\left(-tE_{ij}\mathfrak{w}^{m(\gamma_{ij})}, 0\right)$ and $\log\theta_K=\left(0, \sum_{l\geq 1} \frac{1}{l}t^l \mathfrak{w}^{l \gamma}\partial_{n_\gamma}\right)$. Then the wall crossing formula says that the complete scattering diagram $\mathfrak{D}_\infty$ has one more S-ray, $P_{S+K}=(\gamma+m(\gamma_{ij}))\R_{\geq 0}$ and wall-crossing factor $\log\theta_{S+K}={\left(t^2E_{ij}\mathfrak{w}^{\gamma+m(\gamma_{ij})},0\right)}$. 
\begin{figure}[h!]
    \centering
    \begin{tikzpicture}
    \draw (0,2) -- (4,2);
    \draw (2,0) -- (2,4);
    \draw[red](2,2) -- (4,4);
    \node[font=\tiny, below right] at (2,2) {0};
    \node[font=\tiny, below] at (4,2) {$\theta_S$};
    \node[font=\tiny, above] at (2,4) {$\theta_{K}$};
    \node[font=\tiny, below] at (2,0) {$\theta_{K}^{-1}$};
    \node[font=\tiny, below left] at (0,2) {$\theta_S^{-1}$};
    \node[font=\tiny, above right] at (4,4) {$\theta_{S+K}$};
    \end{tikzpicture}
    \caption{The complete scattering diagram with K and S rays.}
    \label{fig:es2}
\end{figure}
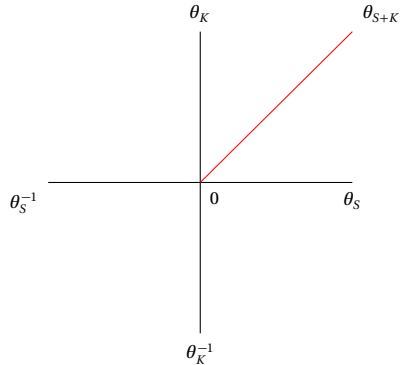

We can check that $\mathfrak{D}_\infty$ is consistent (see Definition \ref{def:pathorderedprod}). In particular we need to prove the following identity:
\begin{equation}
\theta_{K}\circ\theta_S\circ\theta_{K^{-1}}=\theta_S\circ\theta_{S+K}
\end{equation}
\begin{equation*}
\begin{split}
\text{RHS}&=\theta_S\circ\theta_{S+K}\\
&=\exp(\log\theta_S)\circ\exp(\log\theta_{S+K})\\
&=\exp(\log\theta_S{\bullet}_{\text{BCH}}\log\theta_{S+K})\\
&=\exp\left(\log\theta_S+\log\theta_{S+K}\right)\\
\text{LHS}&=\theta_{K}\circ\theta_S\circ\theta_{K^{-1}}\\
&=\exp(\log\theta_K\bullet_{\text{BCH}}\log\theta_S\bullet_{\text{BCH}}\log\theta_{K^{-1}})\\
&=\exp\left(\log\theta_S+\sum_{l\geq1}\frac{1}{l!}\ad_{\log\theta_K}^l\log\theta_S\right)\\
&=\exp\left(\log\theta_S+[\log\theta_K,\log\theta_s]+\sum_{l\geq2}\frac{1}{l!}\ad_{\log\theta_K}^l\log\theta_S\right)\\
&=\exp\left(\log\theta_S-\left(E_{ij}\sum_{k\geq1}\frac{1}{k}\mathfrak{w}^{m(\gamma_{ij})+k\gamma}\langle m(\gamma_{ij}), n_{\gamma}\rangle,0\right)+\sum_{l\geq2}\frac{1}{l!}\ad_{\log\theta_K}^l\log\theta_S\right)\\
&=\exp\left(\log\theta_S+\left(t^2E_{ij}\mathfrak{w}^{m(\gamma_{ij})+\gamma},0\right)+\left(E_{ij}\sum_{k\geq2}\frac{1}{k}t^{k+1}\mathfrak{w}^{m(\gamma_{ij})+k\gamma},0\right)+\sum_{l\geq2}\frac{1}{l!}\ad_{\log\theta_K}^l\log\theta_S\right).
\end{split}
\end{equation*}
We claim that 
\begin{equation}\label{comput:claim}
-\left(E_{ij}\sum_{k\geq2}\frac{1}{k}t^{k+1}\mathfrak{w}^{m(\gamma_{ij})+k\gamma},0\right)=\sum_{l\geq2}\frac{1}{l!}\ad_{\log\theta_K}^l\log\theta_S
\end{equation}
and we compute it explicitly:
\begin{equation}\label{comput:ex1}
\begin{split}
\sum_{l\geq2}\frac{1}{l!}\ad_{\log\theta_K}^l\log\theta_S&=\sum_{l\geq2}\frac{1}{l!}\left(-tE_{ij}(-1)^l\sum_{k_1,\cdots,k_l\geq1}\frac{1}{k_1\cdots k_l}t^{k_1+\cdots+k_l}\mathfrak{w}^{(k_1+\cdots+k_l)\gamma+m(\gamma_{ij})},0\right)\\
&=-\sum_{l\geq2}\frac{(-1)^l}{l!}\left(tE_{ij}\mathfrak{w}^{m(\gamma_{ij})}\left(\sum_{k_1,\cdots,k_l\geq1}\frac{1}{k_1\cdots k_l}t^{k_1+\cdots+k_l}\mathfrak{w}^{(k_1+\cdots+k_l)\gamma}\right),0\right)\\
&=-\sum_{l\geq2}\frac{(-1)^l}{l!}\left(tE_{ij}\mathfrak{w}^{m(\gamma_{ij})}\left(\sum_{k\geq1}\frac{1}{k}t^k\mathfrak{w}^{k\gamma}\right)^l,0\right)\\
&=-\left(tE_{ij}\mathfrak{w}^{m(\gamma_{ij})}\sum_{l\geq2}\frac{(-1)^l}{l!}\left(\sum_{k\geq 1}\frac{1}{k}t^k\mathfrak{w}^{k\gamma}\right)^l,0\right)\\
&=-\left(tE_{ij}\mathfrak{w}^{m(\gamma_{ij})}\left(\exp\left(-\sum_{k\geq1}\frac{1}{k}t^k\mathfrak{w}^{k\gamma}\right)+\sum_{k\geq1}\frac{1}{k}t^k\mathfrak{w}^{k\gamma}-1\right),0\right)\\
&=-\left(tE_{ij}\mathfrak{w}^{m(\gamma_{ij})}\left( \left(1-t\mathfrak{w}^{\gamma}\right)+\sum_{k\geq 1}\frac{1}{k}t^k\mathfrak{w}^{k\gamma}-1\right),0\right)\\
&=-\left(tE_{ij}\mathfrak{w}^{m(\gamma_{ij})} \sum_{k\geq 2}\frac{1}{k}t^k\mathfrak{w}^{k\gamma},0\right).
\end{split}
\end{equation}

\subsubsection{Example 2}\label{Ex:2}
Finally let us give a example with only $S$-rays: assume $\mathcal{V}={i=l,j=k}$, then the wall-crossing formula (equation (2.35) in \cite{WCF2d-4d}) is
\begin{equation}
S_{\gamma_{ij}}^\mu S_{\gamma_{il}}^\mu S_{\gamma_{jl}}^\mu=S_{\gamma_{jl}}^\mu S_{\gamma_{il}}^{\mu'} S_{\gamma_{ij}}^\mu
\end{equation}
with $\gamma_{il}\defeq\gamma_{ij}+\gamma_{jl}$ and $\mu'(\gamma_{il})=\mu(\gamma_{il})-\mu(\gamma_{ij})\mu(\gamma_{jl})$. Let us further assume that $\mu(il)=0$, then the associated initial scattering has two lines:
\begin{equation*}
\mathfrak{D}=\left\lbrace\mathsf{w}_1=\left(m_1=m(\gamma_{ij}),P_1=m_1+\R,\theta_{S_1}\right), \mathsf{w}_2=\left\lbrace m_2=m(\gamma_{jk}), P_2=m_2\R, \theta_{S_2}\right\rbrace\right\rbrace
\end{equation*}
with \begin{equation*}
\begin{split}
\log\theta_{S_1}&=-\left(\mu(\gamma_{ij})tE_{ij}\mathfrak{w}^{m(\gamma_{ij})},0\right)\\
\log\theta_{S_2}&=-\left(\mu(\gamma_{jl})tE_{jl}\mathfrak{w}^{m(\gamma_{jl})},0\right).\\
\end{split}
\end{equation*}
\begin{figure}[h]
    \centering
    \begin{tikzpicture}
    \draw (0,1) -- (4,3);
    \draw (2,0) -- (2,4);
    \draw [red](2,2) -- (3,3.5);
    \node[font=\tiny, below right] at (2,2) {0};
    \node[font=\tiny, below] at (4,3.5) {$\theta_{S_1}$};
    \node[font=\tiny, above] at (2,4) {$\theta_{S_2}$};
    \node[font=\tiny, below] at (2,0) {$\theta_{S_2}^{-1}$};
    \node[font=\tiny, below left] at (0,1.5) {$\theta_{S_1}^{-1}$};
    \node[font=\tiny, above right] at (3,3.5) {$\theta_{S_3}'$};
    \end{tikzpicture}
    \caption{The complete scattering diagram with only S rays.}
    \label{fig:es3}
\end{figure}
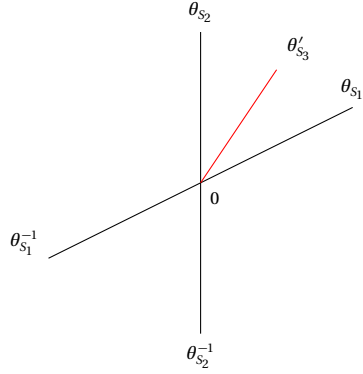 

Its completion has one more ray $P_3'=(m_1+m_2)\R_{\geq0}$ decorated with the automorphism $\theta_{3}'$ such that  
\[\log\theta_3'=\left(\mu(\gamma_{ij})\mu(\gamma_{jl})t^2E_{il}\mathfrak{w}^{m(\gamma_{il})},0\right).\] 
 
Since the path order product involves matrix commutators, the consistency of $\mathfrak{D}_\infty$ can be easily verified. 

\begin{remark}
In the latter example we assume $\mu(\gamma_{il})=0$ in order to have an initial scattering diagram with only two rays, as in our construction of solution of Maurer-Cartan equation in Section \ref{sec:two_walls}. However the general formula can be computed with the same rules, by adding a wall $\mathsf{w}=\left(-(m_1+m_2),\log\theta_3=\left(-\mu(\gamma_{il})t^2E_{il}\mathfrak{w}(\gamma_{il})\right),0\right)$ to the initial scattering diagram.   
\end{remark}

\newpage
\appendix
\section{Computations}

We collect some proofs.

\begin{lemma}[Lemma \ref{lem:dgLa}]\label{proof:dgLa}
$\Big(\tilde{\mathfrak{h}} , [\cdot ,\cdot]_{\sim}\Big)$ is a Lie algebra, where the bracket $[\cdot,\cdot]_{\sim}$ is defined by: 
\begin{equation*}
[(A,\partial_n)z^m , (A',\partial_{n'})z^{m'} ]_{\sim}\defeq([A,A']_{\mathfrak{gl}}z^{m+m'}+A'\langle m',n\rangle z^{m+m'}- A\langle m,n'\rangle z^{m+m'}, [z^m\partial_n,z^{m'}\partial_{n'}]_{\mathfrak{h}} ).
\end{equation*}
\end{lemma}
\begin{proof}
First of all the the bracket is antisymmetric: 
\begin{equation*}
\begin{split}
&[(A,\partial_n)z^m , (A',\partial_{n'})z^{m'} ]_{\sim}=([A,A']_{\mathfrak{gl}}z^{m+m'}+A'\langle m',n\rangle z^{m+m'}- A\langle m,n'\rangle z^{m+m'}, [z^m\partial_n,z^{m'}\partial_{n'}]_{\mathfrak{h}} )\\
&=(-[A',A]_{\mathfrak{gl}}z^{m+m'}+A'\langle m',n\rangle z^{m+m'}- A\langle m,n'\rangle z^{m+m'}, -[z^{m'}\partial_{n'},z^{m}\partial_{n}]_{\mathfrak{h}} )\\
&=-([A',A]_{\mathfrak{gl}}z^{m+m'}+A\langle m,n'\rangle z^{m+m'}- A'\langle m',n\rangle z^{m+m'}, [z^{m'}\partial_{n'},z^{m}\partial_{n}]_{\mathfrak{h}} ).
\end{split}
\end{equation*}
Moreover the Jacobi identity is satisfied: 
\begin{equation*}
\begin{split}
&\big[[(A_1,\partial_{n_1})z^{m_1} , (A_2,\partial_{n_2})z^{m_2} ]_{\sim},(A_3,\partial_{n_3})z^{m_3} \big]_{\sim}\\
&=\big[\big([A_1,A_2]_{\mathfrak{gl}}+A_2\langle m_2,n_1\rangle-A_1\langle m_1,n_2\rangle, \partial_{\langle m_2,n_1\rangle n_2-\langle m_1,n_2\rangle n_1}\big)z^{m_1+m_2}, (A_3,\partial_{n_3})z^{m_3}\big]_{\sim}\\
&=\Big(\big[([A_1,A_2]_{\mathfrak{gl}}+A_2\langle m_2, n_1\rangle -A_1\langle m_1,n_2\rangle),A_3\big]_{\mathfrak{gl}}+A_3\langle m_2,n_1\rangle\langle m_3,n_2\rangle -A_3\langle m_1,n_2\rangle\langle m_3,n_1\rangle +\\
&\quad-\big([A_1,A_2]+A_2\langle m_2, n_1\rangle -A_1\langle m_1, n_2\rangle\big)\langle m_1+ m_2,n_3\rangle, (m_1+m_2+m_3)^{\perp})\Big)z^{m_1+m_2+m_3}\\
&=\Big([[A_1,A_2]_{\mathfrak{gl}}, A_3]_{\mathfrak{gl}}+[A_2,A_3]_{\mathfrak{gl}}\langle m_2,n_1\rangle -\langle m_1, n_2\rangle [A_1,A_3]_{\mathfrak{gl}}+\\
&\quad-[A_1,A_2]_{\mathfrak{gl}}\langle m_1+m_2,n_3\rangle+A_3\langle m_2,n_1\rangle\langle m_3,n_2\rangle+\\
&\quad-A_3\langle m_1,n_2\rangle\langle m_3,n_1\rangle-A_2\langle m_2,n_1\rangle\langle m_1+m_2,n_3\rangle+\\
&\quad+A_1\langle m_1,n_2\rangle\langle m_1+m_2,n_3\rangle ,(m_1+m_2+m_3)^{\perp}\Big)z^{m_1+m_2+m_3}
\end{split}
\end{equation*} 
Then by cyclic permutation we have 
\begin{equation*}
\begin{split}
&\big[[(A_2,\partial_{n_2})z^{m_2}, (A_3,\partial_{n_3})z^{m_3} ]_{\sim},(A_1,\partial_{n_1})z^{m_1}\big]_{\sim}=\\
&=\Big(\big[[A_2,A_3]_{\mathfrak{gl}}, A_1]\big]+[A_3,A_1]_{\mathfrak{gl}}\langle m_3,n_2\rangle -\langle m_2,n_3\rangle [A_2,A_1]_{\mathfrak{gl}}+\\
&-[A_2,A_3]_{\mathfrak{gl}}\langle m_3+m_2,n_1\rangle+A_1\langle m_3,n_2\rangle\langle m_1,n_3\rangle+\\
&-A_1\langle m_2,n_3\rangle\langle m_1,n_2\rangle-A_3\langle m_3,n_2\rangle\langle m_2+m_3, n_1\rangle+\\
&+A_2\langle m_2,n_3\rangle\langle m_2+m_3, n_1\rangle ,(m_1+m_2+m_3)^{\perp}\Big)z^{m_1+m_2+m_3}
\end{split}
\end{equation*} 
and also 
\begin{equation*}
\begin{split}
&\big[[(A_1,\partial_{n_3})z^{m_3}, (A_1,\partial_{n_1})z^{m_1} ]_{\sim},(A_2,\partial_{n_2})z^{m_2} \big]_{\sim}=\\
&=\Big(\big[[A_3,A_1]_{\mathfrak{gl}}, A_2]\big]+[A_1,A_2]\langle m_1,n_3\rangle -\langle m_3,n_1\rangle [A_3,A_2]_{\mathfrak{gl}}+\\
&-[A_3,A_1]_{\mathfrak{gl}}\langle m_1+m_3, n_2\rangle+A_2\langle m_1,n_3\rangle\langle m_2,n_1\rangle+\\
&-A_2\langle m_3,n_1\rangle\langle m_2,n_3\rangle-A_1\langle m_1,n_3\rangle\langle m_3+m_1,n_2\rangle+\\
&+A_3\langle m_3, n_1\rangle\langle m_3+m_1,n_2\rangle ,(m_1+m_2+m_3)^{\perp}\Big)z^{m_1+m_2+m_3}
\end{split}
\end{equation*}
Since Jacobi identity holds for $[\cdot,\cdot]_{\mathfrak{gl}}$ and $[\cdot,\cdot]_{\mathfrak{h}}$, we are left to check that the remaining terms sum to zero. Indeed the coefficient of $[A_2,A_3]_{\mathfrak{gl}}$ is $\langle n_1,m_2\rangle-\langle n_1,m_2+m_3\rangle-\langle n_1,m_3\rangle$, and it is zero. By permuting the indexes, the same hold true for the coefficients in front of the other bracket $[A_1,A_3]_{\mathfrak{gl}}$ and $[A_2,A_1]_{\mathfrak{gl}}$. In addition the coefficient of $A_3$ is $\langle m_2,n_1\rangle\langle m_3,n_2\rangle-\langle m_1,n_2\rangle\langle m_3,n_1\rangle-\langle m_3,n_2\rangle\langle m_2,n_1\rangle-\langle m_3,n_2\rangle\langle m_3,n_1\rangle+ \langle m_3,n_1\rangle\langle m_1,n_2\rangle+\langle m_3,n_1\rangle\langle m_3,n_2\rangle$ and it is zero. By permuting the indexes the same holds true for the coefficient in front of $A_1$ and $A_2$.  
\end{proof}

\begin{lemma}[Lemma \ref{lem:CGamma_module}]\label{proof:CGamma_module}
Let $\mathbf{L}_\Gamma$ be the $\C[\Gamma]-$module defined above. Then, it is a Lie ring with the the Lie bracket $[\cdot, \cdot]_{\Der(\C[\mathbb{G}])}$ induced by $\Der(\C[\mathbb{G}])$.  
\end{lemma}
\begin{proof}
It is enough to prove that $\mathbf{L}_\Gamma$ is a Lie sub-algebra of $\Der(\C[\mathbb{G}])$, i.e. it is closed under $[\cdot,\cdot]_{\Der(\C[\mathbb{G}])}$. By $\C$-linearity it is enough to prove the following claims:
\begin{itemize}
    \item[$(1)$] $[X_{\gamma}\mathfrak{d}_{\gamma_{ij}}, X_{\gamma'}\mathfrak{d}_{\gamma_{kl}}]\in \mathbf{L}_\Gamma$: indeed 
    \begin{equation*}
        \begin{split}
        [X_{\gamma}\mathfrak{d}_{\gamma_{ij}}, X_{\gamma'}\mathfrak{d}_{\gamma_{kl}}]&=X_{\gamma}\mathsf{ad}_{\gamma_{ij}}(X_\gamma'\mathsf{ad}_{\gamma_{kl}})-X_{\gamma'}\mathsf{ad}_{\gamma_{kl}}(X_\gamma\mathsf{ad}_{\gamma_{ij}})\\
        &=X_{\gamma}X_{\gamma'}\mathsf{ad}_{\gamma_{ij}}(\ad_{\gamma_{kl}})-X_{\gamma'}X_{\gamma}\mathsf{ad}_{\gamma_{kl}}(\mathsf{ad}_{\gamma_{ij}})\\
        &={\sigma(\gamma_{ij},\gamma_{kl})}X_{\gamma}X_{\gamma'}\ad_{\gamma_{ij}+\gamma_{kl}}-{\sigma(\gamma_{kl},\gamma_{ij})}X_{\gamma}X_{\gamma'}\ad_{\gamma_{kl}+\gamma_{ij}};
        \end{split}
    \end{equation*}
    \item[$(2)$] $[X_\gamma\mathfrak{d}_{\gamma_{ij}}, X_{\gamma'}\mathfrak{d}_{\gamma''}]\in\mathbf{L}_{\Gamma}$:indeed for every $X_a\in\C[\mathbb{G}]$
    \begin{equation*}
        \begin{split}
        [X_\gamma\mathfrak{d}_{\gamma_{ij}}, X_{\gamma'}\mathfrak{d}_{\gamma''}]X_a&=X_\gamma\mathfrak{d}_{\gamma_{ij}}\left(X_{\gamma'}\mathfrak{d}_{\gamma''}X_a\right)-X_{\gamma'}\mathfrak{d}_{\gamma''}\left(X_\gamma\mathfrak{d}_{\gamma_{ij}}X_a\right)\\
        &=X_{\gamma}\mathfrak{d}_{\gamma_{ij}}\left(X_{\gamma'}\omega(\gamma'',a)X_{\gamma''}X_{a}\right)-X_{\gamma'}\mathfrak{d}_{\gamma''}\left(X_{\gamma}X_{\gamma_{ij}}X_{a}-X_{\gamma}X_{a}X_{\gamma_{ij}}\right)\\
        &=\omega(\gamma'',a)X_{\gamma}\left(X_{\gamma_{ij}}X_{\gamma'}X_{\gamma''}X_{a}-X_{\gamma'}X_{\gamma''}X_{a}X_{\gamma_{ij}}\right)\\
        &-X_{\gamma'}\omega(\gamma'',\gamma+\gamma_{ij}+a)X_{\gamma''}X_{\gamma}X_{\gamma_{ij}}X_{a}+X_{\gamma'}\omega(\gamma'',\gamma+a+\gamma_{ij})X_{\gamma''}X_{\gamma}X_{a}X_{\gamma_{ij}}\\
        &=\left(\omega(\gamma'',a)-\omega(\gamma'',\gamma+\gamma_{ij})-\omega(\gamma'',a)\right)X_{\gamma}X_{\gamma'}X_{\gamma''}X_{\gamma_{ij}}X_{a}\\
        &-\left(\omega(\gamma'',a)-\omega(\gamma'',a)-\omega(\gamma'',\gamma+\gamma_{ij})\right)X_{\gamma}X_{\gamma'}X_{\gamma''}X_{a}X_{\gamma_{ij}}\\
        &=-\omega(\gamma'',\gamma+\gamma_{ij})X_{\gamma}X_{\gamma'}X_{\gamma''}\mathfrak{d}_{\gamma_{ij}}(X_a);
       \end{split}
    \end{equation*}
    \item[$(3)$] $[X_\gamma\mathfrak{d}_{\gamma'}, X_{\gamma''}\mathfrak{d}_{\gamma'''}]\in L$;indeed for every $X_a\in\C[\mathbb{G}]$

\begin{equation*}
\begin{split}
[X_\gamma\mathfrak{d}_{\gamma'}, X_{\gamma''}\mathfrak{d}_{\gamma'''}]X_a &= X_{\gamma} \mathfrak{d}_{\gamma'}\left(X_{\gamma''}\omega({\gamma'''} ,a) X_{{\gamma'''}}X_{a} \right)-X_{\gamma''}\mathfrak{d}_{\gamma'''}\left(\omega(\gamma',a)X_{\gamma'}X_{a}\right)\\ 
&=\omega(\gamma''',a)\omega(\gamma',\gamma'''+\gamma''+a)X_{\gamma}X_{\gamma'}X_{\gamma''}X_{\gamma'''}X_{a}+\\
&-\omega(\gamma',a)\omega(\gamma''',\gamma+\gamma'+a)X_{\gamma''}X_{\gamma'''}X_{\gamma}X_{\gamma'}X_{a}\\
&=\left(\omega(\gamma''' ,a)\left(\omega(\gamma' ,a)+\omega(\gamma',\gamma'''+\gamma'')\right)\right)X_{\gamma}X_{\gamma'}X_{\gamma''}X_{\gamma'''}X_{a}+\\
&-\left(\omega(\gamma' ,a)\left(\omega(\gamma''' ,a)+\omega(\gamma''',\gamma+\gamma')\right)\right)X_{\gamma}X_{\gamma'}X_{\gamma''}X_{\gamma'''}X_{a}\\
&=\omega(\gamma',\gamma'''+\gamma'')X_{\gamma'}X_{\gamma''}X_{\gamma}\mathfrak{d}_{\gamma'''}(X_a)-\omega(\gamma''',\gamma+\gamma')X_{\gamma'''}X_{\gamma''}X_{\gamma}\mathfrak{d}_{\gamma'}(X_a).
\end{split}
\end{equation*}
\end{itemize}

\end{proof}

\begin{lemma}[Lemma \ref{lem:CLambdamodule}]\label{proof:CLambdamodule}
The $\C[\Gamma]-$module $\tilde{\mathbf{L}}$ is a Lie ring with respect to the Lie bracket induced by $\tilde{\mathfrak{h}}$.
\end{lemma}
\begin{proof}
As we have already comment in the proof of Lemma \ref{lem:CGamma_module}, since the bracket is induced by the Lie bracket $[\cdot,\cdot]_{{\tilde{\mathfrak{h}}}}$, we are left to prove that $\tilde{\mathbf{L}}$ is closed under $[\cdot,\cdot]_{\tilde{\mathfrak{h}}}$. In particular by $\C$-linearity it is enough to show the following: 
\begin{align*}
(1)&\, [\mathfrak{w}^{\gamma}\left(E_{ij}\mathfrak{w}^{m(\gamma_{ij})},0\right), \mathfrak{w}^{\gamma'}\left(E_{kl}\mathfrak{w}^{m(\gamma_{kl})},0\right)]\in\tilde{\mathbf{L}}\\
(2)&\, [\mathfrak{w}^{\gamma}\left(E_{ij}\mathfrak{w}^{m(\gamma_{ij})},0\right), \mathfrak{w}^{\gamma'}\left(0, \Omega(\gamma)\mathfrak{w}^{\gamma'}\partial_{n_{\gamma'}}\right)]\in \tilde{\mathbf{L}}\\
(3)&\, [\mathfrak{w}^{\gamma}\left(0, \Omega(\gamma)\mathfrak{w}^{\gamma'}t\partial_{n_{\gamma'}}\right), \mathfrak{w}^{\gamma''}\left(0, \Omega(\gamma''')\mathfrak{w}^{\gamma'''}\partial_{n_{\gamma'''}}\right)]\in \tilde{\mathbf{L}}
\end{align*}
and they are explicitly computed below:
\begin{align*}
(1)&\,[\mathfrak{w}^{\gamma}\left(E_{ij}t\mathfrak{w}^{m(\gamma_{ij})},0\right), \mathfrak{w}^{\gamma'}\left(E_{kl}t\mathfrak{w}^{m(\gamma_{kl})},0\right)]=\left(\mathfrak{w}^{\gamma}\mathfrak{w}^{\gamma'}[E_{ij},E_{kl}]_{\mathfrak{gl}(n)}\mathfrak{w}^{m(\gamma_{ij})}\mathfrak{w}^{m(\gamma_{kl})},0\right) \\
(2)&\,[\mathfrak{w}^{\gamma}\left(E_{ij}\mathfrak{w}^{m(\gamma_{ij})},0\right), \mathfrak{w}^{\gamma''}\left(0, \Omega(\gamma)\mathfrak{w}^{\gamma'}\partial_{n_{\gamma'}}\right)]=\left(-E_{ij}\Omega(\gamma')\langle\gamma+m(\gamma_{ij}),n_{\gamma'}\rangle \mathfrak{w}^{m(\gamma_{ij})}\mathfrak{w}^{\gamma}\mathfrak{w}^{\gamma''}\mathfrak{w}^{\gamma'},0\right)\\
(3)&\,[\mathfrak{w}^{\gamma}\left(0, \Omega(\gamma)\mathfrak{w}^{\gamma'}\partial_{n_{\gamma'}}\right), \mathfrak{w}^{\gamma''}\left(0, \Omega(\gamma''')\mathfrak{w}^{\gamma'''}\partial_{n_{\gamma'''}}\right)]=\Big(0,\Omega(\gamma)\Omega(\gamma''')\mathfrak{w}^{\gamma}\mathfrak{w}^{\gamma'}\mathfrak{w}^{\gamma''}\mathfrak{w}^{\gamma'''}\cdot\\
&\qquad\qquad\cdot\left(\langle\gamma''+\gamma''',n_{\gamma'}\rangle\partial_{n_{\gamma'''}}-\langle\gamma+\gamma',n_{\gamma'''}\rangle\partial_{n_{\gamma'}}\right)\Big).
\end{align*}

\end{proof}

\begin{theorem}[Theorem \ref{thm:homomLiering}]\label{proof:homoLiering}
Let $\left(\mathbf{L}_\Gamma,[\cdot,\cdot]_{\Der(\C[\mathbb{G}])}\right)$ and $\left(\tilde{\mathbf{L}},[\cdot,\cdot]_{\tilde{\mathfrak{h}}}\right)$ be the $\C[\Gamma]$-modules defined before. 
Assume $\omega(\gamma,a)=\Omega(\gamma)\langle{ a},n_{\gamma}\rangle$, then there exists a homomorphism of $\C[\Gamma]$-modules and of Lie rings $\Upsilon\colon \mathbf{L}_\Gamma\to\tilde{\mathbf{L}}$, which is defined as follows:
\begin{equation}
\begin{split}
&\Upsilon(X_{\gamma}\mathfrak{d}_{\gamma_{ij}})\defeq\mathfrak{w}^{\gamma}\left(E_{ij}\mathfrak{w}^{m(\gamma_{ij})},0\right), \forall i\neq j\in\mathcal{V}, \forall\gamma\in\Gamma;\\
&\Upsilon(X_{\gamma'}\mathfrak{d}_\gamma)\defeq \mathfrak{w}^{\gamma'}\left(0, \Omega(\gamma)\mathfrak{w}^{\gamma}\partial_{n_\gamma}\right), \forall \gamma',\gamma\in\Gamma.
\end{split} 
\end{equation}  
\end{theorem}
\begin{proof}

We have to prove that $\Upsilon$ preserves the Lie-bracket, i.e. that for every $l_1,l_2\in L$, then $\Upsilon\left(\left[l_1,l_2\right]_{\mathbf{L}_\Gamma}\right)=\left[\Upsilon(l_1),\Upsilon(l_2)\right]_{\tilde{\mathbf{L}}}$. In particular, by $\C$-linearity it is enough to prove the following identities:
\begin{equation*}
\begin{split}
(1)& \Upsilon\left(\left[X_\gamma\mathfrak{d}_{\gamma_{ij}},X_{\gamma'}\mathfrak{d}_{\gamma_{kl}}\right]_{\mathbf{L}_\Gamma}\right)=\left[\Upsilon(X_\gamma\mathfrak{d}_{\gamma_{ij}}),\Upsilon(X_{\gamma'}\mathfrak{d}_{\gamma_{kl}})\right]_{\tilde{\mathbf{L}}}\\
(2)& \Upsilon\left(\left[X_\gamma\mathfrak{d}_{\gamma_{ij}},X_{\gamma'}\mathfrak{d}_{\gamma''}\right]_{\mathbf{L}_\Gamma}\right)=\left[\Upsilon(X_{\gamma}\mathfrak{d}_{\gamma'}),\Upsilon(X_{\gamma''}\mathfrak{d}_{\gamma_{kl}})\right]_{\tilde{\mathbf{L}}}\\
(3)& \Upsilon\left(\left[X_\gamma\mathfrak{d}_{\gamma'},X_{\gamma''}\mathfrak{d}_{\gamma'''}\right]_{\mathbf{L}_\Gamma}\right)=\left[\Upsilon(X_{\gamma}\mathfrak{d}_{\gamma'}),\Upsilon(X_{\gamma''}\mathfrak{d}_{\gamma'''})\right]_{\tilde{\mathbf{L}}}.
\end{split}
\end{equation*}
The identity $(1)$ is proved below:
\begin{equation*}
\begin{split}
   \text{LHS}&=\Upsilon\left(X_{\gamma}X_{\gamma'}{\sigma(\gamma_{ij},\gamma_{kl})}\mathsf{ad}_{\gamma_{ij}+\gamma_{kl}}-X_{\gamma}X_{\gamma'}{\sigma(\gamma_{kl},\gamma_{ij})}\mathsf{ad}_{\gamma_{kl}+\gamma_{ij}}\right)\\
            &=\mathfrak{w}^{\gamma}\mathfrak{w}^{\gamma'}\left(E_{ij}E_{kl}\mathfrak{w}^{m(\gamma_{ij})}\mathfrak{w}^{m(\gamma_{kl})}-E_{kl}E_{ij}\mathfrak{w}^{m(\gamma_{kl})}\mathfrak{w}^{m(\gamma_{ij})},0\right)\\
            &=\left(\mathfrak{w}^{\gamma}\mathfrak{w}^{\gamma'}[E_{ij}\mathfrak{w}^{m(\gamma_{ij})},E_{kl}\mathfrak{w}^{m(\gamma_{kl})}],0\right)\\
            \text{RHS}&=\left[\left(\mathfrak{w}^{\gamma} E_{ij}\mathfrak{w}^{m(\gamma_{ij})},0\right),\left(\mathfrak{w}^{\gamma'}E_{kl}\mathfrak{w}^{m(\gamma_{kl})},0\right)\right]_{\tilde{\mathfrak{h}}}\\
            &=\left(\mathfrak{w}^{\gamma}\mathfrak{w}^{\gamma'}[E_{ij}\mathfrak{w}^{m(\gamma_{ij})},E_{kl}\mathfrak{w}^{m(\gamma_{kl})}]_{\mathfrak{gl}(n)},0\right)\\
            &=\left(\mathfrak{w}^{\gamma}\mathfrak{w}^{\gamma'}[E_{ij}\mathfrak{w}^{m(\gamma_{ij})},E_{kl}\mathfrak{w}^{m(\gamma_{kl})}],0\right).
        \end{split}
    \end{equation*}
Then the second identity can be proved as follows:    
\begin{equation*}
    \begin{split}
            \text{LHS}&=\Upsilon\left(\omega(\gamma'',\gamma+\gamma_{ij})X_{\gamma+\gamma'+\gamma''}\mathfrak{d}_{\gamma_{ij}}\right)\\
            &=-\omega(\gamma'',\gamma+\gamma_{ij})\mathfrak{w}^{\gamma}\mathfrak{w}^{\gamma'}\mathfrak{w}^{\gamma''}\left(E_{ij}\mathfrak{w}^{m(\gamma_{ij})},0\right)\\
            \text{RHS}&=\left[\left(\mathfrak{w}^{\gamma} E_{ij}\mathfrak{w}^{m(\gamma_{ij})},0\right),\left(0,\mathfrak{w}^{\gamma'}\Omega(\gamma'')\mathfrak{w}^{\gamma''}\partial_{n_{\gamma''}}\right)\right]_{\tilde{\mathfrak{h}}}\\
            &=-\left(\mathfrak{w}^{\gamma}\mathfrak{w}^{\gamma'}E_{ij}\Omega(\gamma'')\langle m(\gamma_{ij})+\gamma, n_{\gamma''} \rangle \mathfrak{w}^{m(\gamma_{ij})}\mathfrak{w}^{\gamma''},0\right).
        \end{split}
    \end{equation*}
Finally the third identity is proved below:    
\begin{equation*}
            \begin{split}
            \text{LHS}&=\Upsilon\left(\omega(\gamma',\gamma''+\gamma''')X_{\gamma}X_{\gamma''}X_{\gamma'''}\mathfrak{d}_{\gamma'''}-\omega(\gamma''',\gamma+\gamma')X_{\gamma'''}X_{\gamma''}X_{\gamma}\mathfrak{d}_{\gamma'}\right)\\
            &=\left(0,\omega(\gamma',\gamma''+\gamma''')\mathfrak{w}^{\gamma}\mathfrak{w}^{\gamma'}\mathfrak{w}^{\gamma''}\Omega(\gamma''')\mathfrak{w}^{\gamma'''}\partial_{n_{\gamma'''}}-\omega(\gamma''',\gamma+\gamma')\mathfrak{w}^{\gamma'''}\mathfrak{w}^{\gamma''}\mathfrak{w}^{\gamma}\Omega(\gamma')\mathfrak{w}^{\gamma'}\partial_{n_{\gamma'}}\right)\\
            \text{RHS}&=\left[\left(0,\mathfrak{w}^{\gamma}\Omega(\gamma')\mathfrak{w}^{\gamma'}\partial_{n_{\gamma'}} \right),\left(0,\mathfrak{w}^{\gamma''}\Omega(\gamma''')\mathfrak{w}^{\gamma'''}\partial_{n_{\gamma'''}}\right)\right]_{\tilde{\mathfrak{h}}}\\
            &=\left(0, \Omega(\gamma')\Omega(\gamma''')\left[\mathfrak{w}^{\gamma}\mathfrak{w}^{\gamma'}\partial_{n_{\gamma'}},\mathfrak{w}^{\gamma''}\mathfrak{w}^{\gamma'''}\partial_{n_{\gamma'''}}\right]_{\tilde{\mathfrak{h}}}\right)\\
            &=\left(0, \Omega(\gamma')\Omega(\gamma''')\mathfrak{w}^{\gamma}\mathfrak{w}^{\gamma'}\mathfrak{w}^{\gamma''}\mathfrak{w}^{\gamma'''}\left(\langle\gamma''+\gamma''', n_{\gamma'}\rangle\partial_{n_{\gamma'''}}-\langle\gamma+\gamma', n_{\gamma'''}\rangle\partial_{n_{\gamma'}}\right)\right).
        \end{split}
    \end{equation*}    
\end{proof}

\bibliographystyle{amsplain}
\bibliography{biblioMC}

\end{document}